\documentclass{amsart}
\usepackage{amssymb}
\usepackage{braket}
\usepackage{mathrsfs}
\usepackage[all]{xy}
\usepackage{graphicx}
\usepackage{color}

\newtheorem{thm}{Theorem}[section]
\newtheorem{cor}[thm]{Corollary}
\newtheorem{prop}[thm]{Proposition}
\theoremstyle{definition}
\newtheorem{dfn}[thm]{Definition}
\newtheorem{ex}[thm]{Example}
\newtheorem{claim}[thm]{Claim}
\newtheorem{lem}[thm]{Lemma}
\newtheorem{fact}[thm]{Fact}

\newtheorem{cond}[thm]{Condition}

\newtheorem{assumption}[thm]{Assumption}
\theoremstyle{remark}
\newtheorem{rem}[thm]{Remark}
\newcommand{\A}{\Asc}               
\newcommand{\C}{\Csc}               
\newcommand{\D}{\Dsc}               
\newcommand{\E}{\Ebb}               
\newcommand{\F}{\mathbb{F}}         

\newcommand{\Xd}{\langle X^{\mr},\del\rangle}  
\newcommand{\Yd}{\langle Y^{\mr},\del\rangle}  
\newcommand{\Zd}{\langle Z^{\mr},\del\rangle}  
\newcommand{\Yr}{\langle Y^{\mr},\rho\rangle}  

\newcommand{\Mf}{M^{\mr}_f}  
\newcommand{\Mg}{M^{\mr}_g}  

\newcommand{\Ep}{\mathbb{E}_{\dag}}  
\newcommand{\cupp}{\cup_{\dag}}  

\newcommand{\Ebb}{\mathbb{E}}        
\newcommand{\Xbb}{\mathbb{X}}        
\newcommand{\Ybb}{\mathbb{Y}}        
\newcommand{\Zbb}{\mathbb{Z}}        

\newcommand{\CEs}{(\C,\E,\sfr)}          
\newcommand{\CFs}{(\C,\F,\sfr|_{\F})}    

\newcommand{\CC}{\Cbf^{n+2}_{\C}}        
\newcommand{\KC}{\Kbf^{n+2}_{\C}}        
\newcommand{\CAC}{\Cbf^{n+2}_{(A,C)}}    
\newcommand{\CACp}{\Cbf^{n+2}_{(A,C\ppr)}} 
\newcommand{\LAC}{\Lambda^{n+2}_{(A,C)}} 
\newcommand{\CX}{(\C,\Xcal)}             
\newcommand{\CACt}{\Cbf^3_{(A,C)}}       
\newcommand{\TAC}{\Tbf^{n+2}_{(A,C)}}    
\newcommand{\CAD}{\Cbf^{n+2}_{(A,D)}}    
\newcommand{\KAC}{\Kbf^{n+2}_{(A,C)}}    
\newcommand{\mr}{\hbox{\boldmath$\cdot$}}

\newcommand{\Hom}{\mathrm{Hom}}      
\newcommand{\add}{\mathrm{add}}      
\newcommand{\modu}{\mathrm{mod}}      
\newcommand{\Ob}{\mathrm{Ob}}        
\newcommand{\Ab}{\mathit{Ab}}        
\newcommand{\Sets}{\mathit{Set}}     
\newcommand{\op}{^\mathrm{op}}       
\newcommand{\ush}{^\sharp}           
\newcommand{\ssh}{_\sharp}           

\newcommand{\id}{1}         
\newcommand{\Id}{\mathrm{Id}}         
\newcommand{\se}{\subseteq}           
\newcommand{\ppr}{^{\prime}}          
\newcommand{\pprr}{^{\prime\prime}}   
\newcommand{\Ker}{\mathrm{Ker}}       
\newcommand{\co}{\colon}              
\newcommand{\ci}{\circ}               
\newcommand{\iv}{^{-1}}               
\newcommand{\uas}{^{\ast}}            
\newcommand{\sas}{_{\ast}}            
\newcommand{\Ext}{\mathrm{Ext}}       
\newcommand{\Ima}{\mathrm{Im}}        

\newcommand{\lla}{\longleftarrow}     
\newcommand{\lra}{\longrightarrow}    
\newcommand{\tc}{\Rightarrow}         
\newcommand{\ltc}{\Longrightarrow}    
\newcommand{\dra}{\dashrightarrow}  

\newcommand{\sfr}{\mathfrak{s}}  
\newcommand{\tfr}{\mathfrak{t}}  
\newcommand{\Yfr}{\mathfrak{Y}}  
\newcommand{\Rfr}{\mathfrak{R}}  

\newcommand{\Asc}{\mathscr{A}}   
\newcommand{\Csc}{\mathscr{C}}   
\newcommand{\Dsc}{\mathscr{D}}   

\newcommand{\Ccal}{\mathcal{C}} 
\newcommand{\Ecal}{\mathcal{E}} 
\newcommand{\Ical}{\mathcal{I}} 
\newcommand{\Pcal}{\mathcal{P}} 
\newcommand{\Tcal}{\mathcal{T}} 
\newcommand{\Xcal}{\mathcal{X}} 
\newcommand{\Ycal}{\mathcal{Y}} 
\newcommand{\Kbf}{\mathbf{K}}   
\newcommand{\Cbf}{\mathbf{C}}   
\newcommand{\Tbf}{\mathbf{T}}   

\newcommand{\und}{\underline}   
\newcommand{\ovl}{\overline}    
\newcommand{\ov}{\overset}      
\newcommand{\un}{\underset}     
\newcommand{\ti}{\times}        

\newcommand{\bsm}{\begin{smallmatrix}}
\newcommand{\esm}{\end{smallmatrix}}

\newcommand{\be}{\beta}          
\newcommand{\vp}{\varphi}        
\newcommand{\vt}{\vartheta}      
\newcommand{\lam}{\lambda}       
\newcommand{\om}{\omega}         
\newcommand{\thh}{\theta}        
\newcommand{\del}{\delta}        

\newcommand{\Sig}{\Sigma}        
\newcommand{\Om}{\Omega}         

\numberwithin{equation}{section}

\newcommand{\coh}{\mathrm{coh}}     
\newcommand{\vect}{\mathrm{vect}}      

\newcommand{\Dbdd}{\mathcal{D}^{\rm b}}
\newcommand{\Dsing}{\mathcal{D}_{\rm sing}}
\newcommand{\Dtstr}{\mathcal{D}^{\rm \le 0}}
\newcommand{\T}{\mathscr{T}}
\newcommand{\Ttstr}{\mathcal{T}^{\rm \le 0}}

\begin{document}

\title{$n$-exangulated categories}
%
%

\author{Martin Herschend}
\address{Department of Mathematics, Uppsala University, Box 480, 751 06 Uppsala, Sweden}
\email{martin.herschend@math.uu.se}

\author{Yu Liu}
\address{School of Mathematics, Southwest Jiaotong University, 610031, Chengdu, Sichuan, China}
\email{liuyu86@swjtu.edu.cn}

\author{Hiroyuki Nakaoka}
\address{Research and Education Assembly, Science and Engineering Area, Research Field in Science, Kagoshima University, 1-21-35 Korimoto, Kagoshima, 890-0065 Japan}
\email{nakaoka@sci.kagoshima-u.ac.jp}

\thanks{The authors wish to thank Yann Palu for his advice and Gustavo Jasso for useful discussions. The authors also wish to thank Ivo Dell'Ambrogio and Osamu Iyama for the arguments on relative theory.
The first author is supported by the Swedish Research Councils Project Research Grant for Junior Researchers. 
The second author was supported by the Knut and Alice Wallenberg Foundation Postdoctoral Program.
The third author is supported by JSPS KAKENHI Grant Number 17K18727.}

\begin{abstract}
For each positive integer $n$ we introduce the notion of $n$-exangulated categories as higher dimensional analogues of extriangulated categories defined by Nakaoka-Palu. We characterize which $n$-exangulated categories are $n$-exact in the sense of Jasso and which are $(n+2)$-angulated in the sense of Geiss-Keller-Oppermann. For extriangulated categories with enough projectives and injectives we introduce the notion of $n$-cluster tilting subcategories and show that under certain conditions such $n$-cluster tilting subcategories are $n$-exangulated.
\end{abstract}

\maketitle


\tableofcontents


\section{Introduction}
A fundamental idea in Iyama's higher dimensional Auslander-Reiten theory \cite{I1} is to replace short exact sequences as the basic building blocks for homological algebra, by longer exact sequences. A typical setting is to consider an $n$-cluster tilting subcategory $\C$ of an abelian category $\A$. For instance one may take $\A$ to be the category $\modu A$ of finitely generated modules over a $n$-representation finite algebra $A$ in the sense of \cite{IO} (see \cite{HI1}, \cite{HI2} for further examples). Then $\modu A$ has a unique $n$-cluster tilting subcategory $\C$. In this case there is also an $n$-cluster tilting subcategory of the bounded derived category $\D^{b}(\modu A)$, obtained by closing $\C$ under shift by $n$ and direct sums. To generalize, one may take $A$ to be an $n$-complete algebra in the sense of \cite{I2} (see \cite{P1} for further examples). Then $\modu A$ has a distinguished exact subcategory that admits an $n$-cluster tilting subcategory.

To summarize $n$-cluster tilting subcategories of abelian, exact and triangulated categories play a crucial role in higher dimensional Auslander-Reiten theory and what is sometimes called higher homological algebra. These three settings have all been axiomatized leading to the notions of $n$-abelian and $n$-exact categories introduced in \cite{J} as well as the notion of $(n+2)$-angulated categories introduced in \cite{GKO} (see also \cite{BT} for more discussion of the axioms and \cite{BJT} for a more recent class of examples). Setting $n =1$ recovers the notions of abelian, exact and triangulated categories. Any $n$-cluster tilting subcategory $\C$ of an abelian or exact category is $n$-abelian respectively $n$-exact (see \cite[Theorem 3.16]{J} and \cite[Theorem 4.14]{J}). Similarly \cite[Theorem 1]{GKO} show that if $\C$ is an $n$-cluster tilting subcategory of a triangulated category closed under shift by $n$, then it is $(n+2)$-angulated. The condition that $\C$ is closed under shift by $n$ is crucial and no reasonable axiomatization of arbitrary $n$-cluster tilting subcategories of triangulated categories has to our knowledge been proposed.

The notion of extriangulated categories was recently introduced in \cite{NP} as a common generalization of exact and triangulated categories. The data of such a category is a triplet $\CEs$, where $\C$ is an additive category, $\E\co\C\op\ti\C\to\Ab$ is a biadditive functor (modelled after $\Ext^1$) and $\sfr$ assigns to each $\delta \in \E(C,A)$ a class of $3$-term sequences with end terms $A$ and $C$ such that certain axioms hold. The aim of this paper is to introduce an $n$-analogue of this notion called $n$-exangulated categories. Such a category is a similar triplet $\CEs$, with the main distinction being that the $3$-term sequences mentioned above are replaced by $(n+2)$-term sequences. The precise definition is given in Definition~\ref{DefEACat} (see also Definition~\ref{DefReal}). It is a true analogue in the sense that $1$-exangulated categories are the same as extriangulated categories (see Proposition~\ref{Prop1EAandET}). As typical examples we have that $n$-exact and $(n+2)$-angulated categories are $n$-exangulated (see Proposition~\ref{nExactIsnExangulated} and Proposition~\ref{PropAtoEA}). 

Assuming that an extriangulated category $\CEs$ has enough projectives and injectives in the sense of \cite[Definition 1.12]{LN} we can introduce higher extension groups $\E^i$ using dimension shift. This allows us furthermore to define the notion of an $n$-cluster tilting subcategory $\Tcal \subseteq \C$ (see Definition~\ref{DefnCluster}). We show in this setting that if $\C$ satisfies Condition~\ref{CondWIC} and $\Tcal$ satisfies a Condition~\ref{CondnCluster}, then $\Tcal$ is $n$-exangulated (see Theorem~\ref{ExactSaigo}). Condition~\ref{CondWIC} on $\C$ is automatic if $\C$ is triangulated and equivalent to weak idempotent completeness in case $\C$ is exact. On the other hand Condition~\ref{CondnCluster} on $\Tcal$ is automatic if $\C$ is exact and if $\C$ is triangulated it follows from the condition $\Tcal = \Tcal[n]$. In this sense Theorem~\ref{ExactSaigo} can be seen as an analogue of \cite[Theorem 3.16]{J}, \cite[Theorem 4.14]{J}) and \cite[Theorem 1]{GKO} mentioned above. 

As illustrated by the results above, $n$-exangulated categories provide a common ground for studying the different settings of higher homological algebra. Compared to the classical setting many important questions regarding the interplay of $n$-abelian, $n$-exact and $(n+2)$-angulated categories remain open. For instance from any abelian category we obtain a triangulated category by taking its derived category. As far as we know, no satisfactory higher analogue of this procedure has been proposed and in view of \cite{JK1} it seems that finding such is non-trivial. It is our hope that by providing the common framework of $n$-exangulated categories we might contribute to answering some of these questions.

The paper is organized as follows. In Section~\ref{section_EA} we introduce $n$-exangulated categories and related notions. In Section~\ref{section_Fund} we present basic properties of $n$-exangulated categories. In Section~\ref{section_Typical}, we show that $1$-exangulated categories are the same as extriangulated categories. We also characterize $(n+2)$-angulated categories as $n$-exangulated categories $\CEs$ for which $\E = \C(-,\Sigma-)$ for some automorphism $\Sigma$ of $\C$. Similarly we characterize $n$-exact categories as $n$-exangulated categories for which inflations are monomorphisms and deflations are epimorphisms. In Section~\ref{section_nCT}, we consider a $1$-exangulated category $\CEs$ with enough projectives and injectives. We introduce syzygies, cosyzygies and higher extension groups for $\C$. These are then used to introduce cup products and $n$-cluster tilting subcategories of $\C$. Finally, we apply these tools to show that under certain conditions an $n$-cluster tilting subcategory of $\C$ is $n$-exangulated. In Section~\ref{section_Examples} we provide several families of examples of $n$-exangulated categories including ones given explicitly by quivers and relations.

\section{$n$-exangulated categories}\label{section_EA}

\subsection{$\E$-extensions}
Throughout this paper, let $\C$ be an additive category.
\begin{dfn}\label{DefExtension}
Suppose $\C$ is equipped with a biadditive functor $\E\co\C\op\ti\C\to\Ab$. For any pair of objects $A,C\in\C$, an element $\del\in\E(C,A)$ is called an {\it $\E$-extension} or simply an {\it extension}. We also write such $\del$ as ${}_A\del_C$ when we indicate $A$ and $C$.
\end{dfn}

\begin{rem}
Let ${}_A\del_C$ be any extension. Since $\E$ is a bifunctor, for any $a\in\C(A,A\ppr)$ and $c\in\C(C\ppr,C)$, we have extensions
\[ \E(C,a)(\del)\in\E(C,A\ppr)\ \ \text{and}\ \ \E(c,A)(\del)\in\E(C\ppr,A). \]
We abbreviately denote them by $a\sas\del$ and $c\uas\del$.
In this terminology, we have
\[ \E(c,a)(\del)=c\uas a\sas\del=a\sas c\uas\del \]
in $\E(C\ppr,A\ppr)$.
\end{rem}

\begin{dfn}\label{DefMorphExt}
Let ${}_A\del_C,{}_B\rho_D$ be any pair of $\E$-extensions. A {\it morphism} $(a,c)\co\del\to\rho$ of extensions is a pair of morphisms $a\in\C(A,B)$ and $c\in\C(C,D)$ in $\C$, satisfying the equality
\[ a\sas\del=c\uas\rho. \]
\end{dfn}

\begin{rem}\label{RemMorphExt}
Let ${}_A\del_C$ be any extension. We have the following.
\begin{enumerate}
\item Any morphism $a\in\C(A,B)$ gives rise to a morphism of $\E$-extensions
\[ (a,\id_C)\co\del\to a\sas\del. \]
\item Any morphism $c\in\C(D,C)$ gives rise to a morphism of $\E$-extensions
\[ (\id_A,c)\co c\uas\del\to\del. \]
\end{enumerate}
\end{rem}

\begin{dfn}\label{DefSplitExtension}
For any $A,C\in\C$, the zero element ${}_A0_C=0\in\E(C,A)$ is called the {\it split $\E$-extension}.
\end{dfn}

\begin{dfn}\label{DefSumExtension}
Let ${}_A\del_C,{}_B\rho_D$ be any pair of $\E$-extensions. Let
\[ C\ov{\iota_C}{\lra}C\oplus D\ov{\iota_D}{\lla}D \]
and 
\[ A\ov{p_A}{\lla}A\oplus B\ov{p_B}{\lra}B \]
be coproduct and product in $\C$, respectively. Remark that, by the biadditivity of $\E$, we have a natural isomorphism
\[ \E(C\oplus D,A\oplus B)\cong \E(C,A)\oplus\E(C,B)\oplus\E(D,A)\oplus\E(D,B). \]

Let $\del\oplus\rho\in\E(C\oplus D,A\oplus B)$ be the element corresponding to $(\del,0,0,\rho)$ through this isomorphism. This is the unique element which satisfies
\begin{eqnarray*}
\E(\iota_C,p_A)(\del\oplus\rho)=\del&,&\E(\iota_C,p_{B})(\del\oplus\rho)=0,\\
\E(\iota_{D},p_A)(\del\oplus\rho)=0&,&\E(\iota_{D},p_{B})(\del\oplus\rho)=\rho.
\end{eqnarray*}

If $A=B$ and $C=D$, then the sum $\del+\rho\in\E(C,A)$ of $\del,\rho\in\E(C,A)$ is obtained by
\[ \del+\rho=\E(\Delta_C,\nabla_A)(\del\oplus\rho), \]
where $\Delta_C=\begin{bmatrix}1\\1\end{bmatrix}\co C\to C\oplus C$, $\nabla_A=[1\ 1]\co A\oplus A\to A$.
\end{dfn}

\subsection{$n$-exangles}
Let $\C$ be an additive category as before, and let $n$ be any positive integer.

\begin{dfn}\label{DefComplex}
Let $\Cbf_{\C}$ be the category of complexes in $\C$. As its full subcategory, define $\CC$ to be the category of complexes in $\C$ whose components are zero in the degrees outside of $\{0,1,\ldots,n+1\}$. Namely, an object in $\CC$ is a complex $X^{\mr}=\{X^i,d_X^i\}$ of the form
\[ X^0\ov{d_X^0}{\lra}X^1\ov{d_X^1}{\lra}\cdots\ov{d_X^{n-1}}{\lra}X^n\ov{d_X^n}{\lra}X^{n+1}. \]
We write a morphism $f^{\mr}\co X^{\mr}\to Y^{\mr}$ simply $f^{\mr}=(f^0,f^1,\ldots,f^{n+1})$, only indicating the terms of degrees $0,\ldots,n+1$.

We define the homotopy relation on the morphism sets in the usual way. Thus morphisms $f^{\mr},g^{\mr}\in\CC(X^{\mr},Y^{\mr})$ are {\it homotopic} if there is a {\it homotopy}, i.e., a sequence of morphisms $\vp^{\mr}=(\vp^1,\ldots,\vp^{n+1})$ of $\vp^i\in\C(X^i,Y^{i-1})$ satisfying
\begin{eqnarray*}
g^0-f^0&=&\vp^1\ci d_X^0,\\
g^i-f^i&=&d_Y^{i-1}\ci\vp^i+\vp^{i+1}\ci d_X^i\quad (1\le i\le n),\\
g^{n+1}-f^{n+1}&=&d_Y^n\ci\vp^{n+1}.
\end{eqnarray*}
In this case we write as $f^{\mr}\sim g^{\mr}$, or $f^{\mr}\un{\vp^{\mr}}{\sim} g^{\mr}$. 
We denote the homotopy category by $\KC$, which is the quotient of $\CC$ by the ideal of null-homotopic morphisms. If $f^{\mr}\in\CC(X^{\mr},Y^{\mr})$ gives an isomorphism in $\KC$, we call it a {\it homotopy equivalence}, as usual. Similarly a {\it homotopy inverse} of $f^{\mr}$ is a morphism $g^{\mr}\in\CC(Y^{\mr},X^{\mr})$ which gives the inverse in $\KC$.
\end{dfn}

\begin{claim}\label{ClaimDEA}
Assume that $f^{\mr}\in\CC(X^{\mr},Y^{\mr})$ is a homotopy equivalence in $\CC$. For a homotopy inverse $g^{\mr}$ of $f^{\mr}$, we have the following.
\begin{enumerate}
\item If $X^0=Y^0=A$ and $f^0=\id_A$, then $g^{\mr}$ can be chosen to satisfy $g^0=\id_A$.
\item Dually, if $X^{n+1}=Y^{n+1}=C$ and $f^{n+1}=\id_C$, then $g^{\mr}$ can be chosen to satisfy $g^{n+1}=\id_C$.
\item If $f^0=\id_A$ and $f^{n+1}=\id_C$, then $g^{\mr}$ can be chosen to satisfy both $g^0=\id_A$ and $g^{n+1}=\id_C$.
\end{enumerate}
\end{claim}
\begin{proof}
Let $h^{\mr}$ be any homotopy inverse of $f^{\mr}$, and let $f^{\mr}\ci h^{\mr}\un{\vp^{\mr}}{\sim}\id_{Y^{\mr}}$, $h^{\mr}\ci f^{\mr}\un{\psi^{\mr}}{\sim}\id_{X^{\mr}}$ be homotopies.

{\rm (1)} Modifying $h^{\mr}$ by a homotopy $(\vp^1,0,\ldots,0)$, we obtain a morphism $g^{\mr}\co Y^{\mr}\to X^{\mr}$ of the form $(\id_A,h^1+d_X^0\ci\vp^1,h^2,\ldots,h^{n+1})$.
Since $h^{\mr}\sim g^{\mr}$, this is also a homotopy inverse of $f^{\mr}$.

{\rm (2)} Dually to {\rm (1)}, $g^{\mr}=(h^0,\ldots,h^{n-1},h^n+\psi^{n+1}\ci d_Y^n,\id_C)$ gives a homotopy inverse of $f^{\mr}$ with the desired property.

{\rm (3)} $g^{\mr}=(\id_A,h^1+d_X^0\ci\vp^1,h^2,\ldots,h^{n-1},h^n+\psi^{n+1}\ci d_Y^n,\id_C)$ satisfies the desired properties. (If $n=1$, we put $g^{\mr}=(\id_A,h^1+d_X^0\ci\vp^1+\psi^2\ci d_Y^1,\id_C)$.)
\end{proof}

\begin{dfn}\label{DefAE}
Let $\C,\E,n$ be as before. Define a category $\AE=\AE^{n+2}_{(\C,\E)}$ as follows.
\begin{enumerate}
\item A pair $\Xd$ of $X^{\mr}\in\CC$ and $\del\in\E(X^{n+1},X^0)$ is called an {\it $\E$-attached complex of length $n+2$}, if it satisfies
\[ (d_X^0)\sas\del=0\quad\text{and}\quad(d_X^n)\uas\del=0. \]
We also denote it by
\[ X^0\ov{d_X^0}{\lra}X^1\ov{d_X^1}{\lra}\cdots\ov{d_X^n}{\lra} X^{n+1}\ov{\del}{\dra}. \]
When we emphasize the end-terms $X^0=A$ and $X^{n+1}=C$, we denote the pair by $_A\Xd_C$ or just by $_A\Xd$ or $\Xd_C$, depending on our purpose.
\item For such pairs $\Xd$ and $\Yr$, a {\it morphism} $f^{\mr}\co\Xd\to\Yr$ is defined to be a morphism $f^{\mr}\in\CC(X^{\mr},Y^{\mr})$ satisfying
\[ (f^0)\sas\del=(f^{n+1})\uas\rho. \]
We use the same composition and the identities as in $\CC$.
\end{enumerate}
\end{dfn}

\begin{prop}\label{PropAE1}
Let $f^{\mr}\co\Xd\to\Yr$ be any morphism in $\AE$.
\begin{enumerate}
\item If a morphism $f^{\prime\mr}\in\CC(X^{\mr},Y^{\mr})$ satisfies $f^{\mr}\sim f^{\prime\mr}$ in $\CC$, then $f^{\prime\mr}$ also belongs to $\AE(\Xd,\Yr)$.
Thus we may consider the same homotopy relation $\sim$ in $\AE$.
\item If $f^{\mr}$ has a homotopy inverse $g^{\mr}\co Y^{\mr}\to X^{\mr}$ in $\CC$, then $g^{\mr}$ belongs to $\AE(\Yr,\Xd)$.
\end{enumerate}
\end{prop}
\begin{proof}
{\rm (1)} Indeed if $f^{\mr}\un{\vp^{\mr}}{\sim}f^{\prime\mr}$, then we have $(f^{\prime 0})\sas\del=(f^0+\vp^1\ci d_X^0)\sas\del=(f^0)\sas\del$ and $(f^{\prime n+1})\uas\rho=(f^{n+1}+d_Y^n\ci\vp^{n+1})\uas\rho=(f^{n+1})\uas\rho$.

{\rm (2)} By assumption, there are homotopies $g^{\mr}\ci f^{\mr}\sim\id_{X^{\mr}}$ and $f^{\mr}\ci g^{\mr}\sim\id_{Y^{\mr}}$. As in the proof of {\rm (1)}, this implies
\[ (g^0\ci f^0)\sas\del=\del\quad\text{and}\quad (f^{n+1}\ci g^{n+1})\uas\rho=\rho. \]
Since $(f^0)\sas\del=(f^{n+1})\uas\rho$ by definition, it follows
\begin{eqnarray*}
(g^{n+1})\uas\del&=&(g^{n+1})\uas(g^0\ci f^0)\sas\del\ =\ (g^{n+1})\uas(g^0)\sas(f^{n+1})\uas\rho\\
&=&(g^0)\sas(f^{n+1}\ci g^{n+1})\uas\rho\ =\ (g^0)\uas\rho,
\end{eqnarray*}
which means $g^{\mr}\in\AE(\Yr,\Xd)$.
\end{proof}

\begin{dfn}\label{DefYoneda}
By Yoneda lemma, any extension $\del\in\E(C,A)$ induces natural transformations
\[ \del\ssh\colon\C(-,C)\Rightarrow\E(-,A)\ \ \text{and}\ \ \del\ush\colon\C(A,-)\Rightarrow\E(C,-). \]
For any $X\in\C$, these $(\del\ssh)_X$ and $\del\ush_X$ are given as follows.
\begin{enumerate}
\item $(\del\ssh)_X\colon\C(X,C)\to\E(X,A)\ ;\ f\mapsto f\uas\del$.
\item $\del\ush_X\colon\C(A,X)\to\E(C,X)\ ;\ g\mapsto g\sas\delta$.
\end{enumerate}
We abbreviately denote $(\del\ssh)_X(f)$ and $\del\ush_X(g)$ by $\del\ssh(f)$ and $\del\ush(g)$.
\end{dfn}

\begin{prop}\label{PropAEComplex}
Let $Q\in\C$ be any object. Then the $\mathit{Hom}$ functor
\[ \C(Q,-)\co\C\to\Ab \]
induces the following functor $\Yfr_Q\co\AE\to C_{\Ab}^{n+3}$. Here $\Ab$ denotes the category of abelian groups.
\begin{itemize}
\item[{\rm (i)}] An object $\Xd\in\AE$ is sent to the complex $\Yfr_Q(\Xd)$ defined as
\[ \C(Q,X^0)\ov{\C(Q,d_X^0)}{\lra}\cdots\ov{\C(Q,d_X^n)}{\lra}\C(Q,X^{n+1})\ov{\del\ssh}{\lra}\E(Q,X^0). \]
\item[{\rm (ii)}] A morphism $f^{\mr}\in\AE(\Xd,\Yr)$ is sent to the morphism of complexes
\[ \Yfr_Q(f^{\mr})=\big(\C(Q,f^0),\ldots,\C(Q,f^{n+1}),\E(Q,f^0)\big). \]
\end{itemize}
Similarly, $\C(-,Q)\co\C\op\to\Ab$ induces a functor $\AE\op\to C_{\Ab}^{n+3}$ which sends $\Xd$ to the complex $\C(X^{n+1},Q)\ov{\C(d_X^n,Q)}{\lra}\cdots\ov{\C(d_X^0,Q)}{\lra}\C(X^0,Q)\ov{\del\ush}{\lra}\E(X^{n+1},Q)$.
\end{prop}
\begin{proof}
This is straightforward.
\end{proof}

\begin{dfn}\label{DefEA}
An {\it $n$-exangle} is a pair $\Xd$ of $X^{\mr}\in\CC$ and $\del\in\E(X^{n+1},X^0)$ which satisfies the following conditions.
\begin{enumerate}
\item The following sequence of functors $\C\op\to\Ab$ is exact.
\begin{equation}\label{exnat1}
\C(-,X^0)\ov{\C(-,d_X^0)}{\ltc}\cdots\ov{\C(-,d_X^n)}{\ltc}\C(-,X^{n+1})\ov{\del\ssh}{\ltc}\E(-,X^0)
\end{equation}
\item The following sequence of functors $\C\to\Ab$ is exact.
\begin{equation}\label{exnat2}
\C(X^{n+1},-)\ov{\C(d_X^n,-)}{\ltc}\cdots\ov{\C(d_X^0,-)}{\ltc}\C(X^0,-)\ov{\del\ush}{\ltc}\E(X^{n+1},-)
\end{equation}
\end{enumerate}
In particular any $n$-exangle is an object in $\AE$.
A {\it morphism of $n$-exangles} simply means a morphism in $\AE$. Thus $n$-exangles form a full subcategory of $\AE$. 
\end{dfn}

\begin{rem}\label{RemEACoprod}
In $\AE$, a coproduct of objects $\Xd,\Yr$ is given by $\langle X^{\mr}\oplus Y^{\mr},\del\oplus\rho\rangle$, where $X^{\mr}\oplus Y^{\mr}$ is the direct sum in $\CC$ and $\del\oplus\rho$ is the one in Definition~\ref{DefSumExtension}. Remark that $\langle X^{\mr}\oplus Y^{\mr},\del\oplus\rho\rangle$ is an $n$-exangle if and only if both $\Xd,\Yr$ are $n$-exangles.
\end{rem}

\begin{claim}\label{ClaimSplit}
For any $n$-exangle $_A\Xd_C$, the following are equivalent.
\begin{enumerate}
\item $\del=0$.
\item There is $r\in\C(X^1,A)$ satisfying $r\ci d_X^0=\id_A$.
\item There is $s\in\C(C,X^n)$ satisfying $d_X^n\ci s=\id_C$.
\end{enumerate}
\end{claim}
\begin{proof}
The equivalence $(1)\Leftrightarrow(2)$ follows immediately from the exactness of
\[ \C(X^1,A)\ov{-\ci d_X^0}{\lra}\C(A,A)\ov{\del\ush}{\lra}\E(C,A). \]
Similarly for $(1)\Leftrightarrow(3)$.
\end{proof}

\begin{prop}\label{PropEAInv}
Let $\Xd,\Yr$ be any pair of objects in $\AE$.
Suppose that $f^{\mr}\in\AE(\Xd,\Yr)$ is a homotopy equivalence. Then $\Xd$ is an $n$-exangle if and only if $\Yr$ is.
\end{prop}
\begin{proof}
Let $g^{\mr}$ be a homotopy inverse of $f^{\mr}$, and let $g^{\mr}\ci f^{\mr}\un{\vp^{\mr}}{\sim}\id_{X^{\mr}}$, $f^{\mr}\ci g^{\mr}\un{\psi^{\mr}}{\sim}\id_{Y^{\mr}}$ be homotopies. Let $Q\in\C$ be any object.

By Proposition~\ref{PropAEComplex}, we obtain complexes $\Xbb^{\mr}=\Yfr_Q(\Xd)$, $\Ybb^{\mr}=\Yfr_Q(\Yr)$ and morphisms
\[ F^{\mr}=\Yfr_Q(f^{\mr})\co\Xbb^{\mr}\to\Ybb^{\mr},\quad%
G^{\mr}=\Yfr_Q(g^{\mr})\co\Ybb^{\mr}\to\Xbb^{\mr} \]
in $C_{\Ab}^{n+3}$. For the composition $G^{\mr}\ci F^{\mr}$
\[
\xy
(-54,8)*+{\C(Q,X^0)}="0";
(-26,8)*+{\C(Q,X^1)}="1";
(0,8)*+{\cdots}="2";
(26,8)*+{\C(Q,X^{n+1})}="3";
(54,8)*+{\E(Q,X^0)}="4";
(-54,-8)*+{\C(Q,X^0)}="10";
(-26,-8)*+{\C(Q,X^1)}="11";
(0,-8)*+{\cdots}="12";
(26,-8)*+{\C(Q,X^{n+1})}="13";
(54,-8)*+{\E(Q,X^0)}="14";
{\ar^{\C(Q,d_X^0)} "0";"1"};
{\ar^(0.6){\C(Q,d_X^1)} "1";"2"};
{\ar^(0.4){\C(Q,d_X^n)} "2";"3"};
{\ar^{\del\ssh} "3";"4"};
{\ar_{\C(Q,d_X^0)} "10";"11"};
{\ar_(0.6){\C(Q,d_X^1)} "11";"12"};
{\ar_(0.4){\C(Q,d_X^n)} "12";"13"};
{\ar_{\del\ssh} "13";"14"};
{\ar|*+{_{G^0\ci F^0}} "0";"10"};
{\ar|*+{_{G^1\ci F^1}} "1";"11"};
{\ar|*+{_{G^{n+1}\ci F^{n+1}}} "3";"13"};
{\ar|*+{_{G^{n+2}\ci F^{n+2}}} "4";"14"};
{\ar@{}|\circlearrowright "0";"11"};
{\ar@{}|\circlearrowright "1";"13"};
{\ar@{}|\circlearrowright "3";"14"};
\endxy
\]
the sequence of morphisms in $\Ab$
\[ \Phi^1=\C(Q,\vp^1),\ldots,\Phi^{n+1}=\C(Q,\vp^{n+1}) \]
satisfies
\[ 1-G^i\ci F^i=\C(Q,d_X^{i-1})\ci\Phi^i+\Phi^{i+1}\ci\C(Q,d_X^i)\quad(1\le i\le n) \]
and
\[ 1-G^{n+1}\ci F^{n+1}=\C(Q,d_X^n)\ci \Phi^{n+1}. \]
This shows that $G^{\mr}\ci F^{\mr}$ induces $H^i(G^{\mr}\ci F^{\mr})=\id$ on cohomologies for any $1\le i\le n+1$. In the same way, by using $\psi^{\mr}$, we can show $H^i(F^{\mr}\ci G^{\mr})=\id$ for $1\le i\le n+1$. Thus
\[ H^i(F^{\mr})\co H^i(\Xbb^{\mr})\ov{\cong}{\lra} H^i(\Ybb^{\mr})\quad (1\le i\le n+1) \]
are isomorphisms. In particular $\Xbb^{\mr}$ is exact if and only if $\Ybb^{\mr}$ is.
Similarly for the exactness of $(\ref{exnat2})$.
\end{proof}

\subsection{The categories $\CAC$ and $\KAC$}

We consider the complexes of length $n+2$ with fixed end-terms, as follows.
\begin{dfn}\label{DefSeq}
For any pair of objects $A,C\in\C$, define a subcategory $\Cbf^{n+2}_{(\C;A,C)}$ of $\CC$ as follows. We abbreviately denote $\Cbf^{n+2}_{(\C;A,C)}$ by $\CAC$, when $\C$ is clear from the context.
\begin{enumerate}
\item An object $X^{\mr}\in\CC$ belongs to $\CAC$ if it satisfies $X^0=A$ and $X^{n+1}=C$. We also write it as ${}_AX^{\mr}_C$ when we emphasize $A$ and $C$.
\item For any $X^{\mr},Y^{\mr}\in\CAC$, the morphism set is defined by
\[ \CAC(X^{\mr},Y^{\mr})=\{ f^{\mr}\in\CC(X^{\mr},Y^{\mr})\mid f^0=\id_A,\ f^{n+1}=\id_C\}. \]
\end{enumerate}

This category $\CAC$ is no longer (pre-)additive. However we can take the quotient $\CAC$ by the same homotopy relation $\sim$ as in $\CC$.
Namely, morphisms $f^{\mr},g^{\mr}\in\CC(X^{\mr},Y^{\mr})$ are {\it homotopic} if there is a sequence of morphisms $\vp^{\mr}=(\vp^1,\ldots,\vp^{n+1})$ satisfying
\begin{eqnarray}
0&=&\vp^1\ci d_X^0, \label{vp1}\\
g^i-f^i&=&d_Y^{i-1}\ci\vp^i+\vp^{i+1}\ci d_X^i\quad (1\le i\le n),\nonumber\\
0&=&d_Y^n\ci\vp^{n+1}.\label{vpn+1}
\end{eqnarray}
We use the same notation $f^{\mr}\sim g^{\mr}$ and $f^{\mr}\un{\vp^{\mr}}{\sim} g^{\mr}$ as before.
We denote the resulting category by $\KAC$, which is a subcategory of $\KC$.

For any morphism $f^{\mr}$ in $\CAC$, its image in $\KAC$ will be denoted by $\und{f^{\mr}}$. As the usual terminology, a morphism $f^{\mr}\in\CAC(X^{\mr},Y^{\mr})$ is called a {\it homotopy equivalence} if it induces an isomorphism $\und{f^{\mr}}$ in $\KAC$. Two objects $X^{\mr},Y^{\mr}\in\CAC$ are said to be {\it homotopically equivalent} if there is some homotopy equivalence $X^{\mr}\to Y^{\mr}$. We denote the homotopic equivalence class of ${}_AX^{\mr}_C$ by $[{}_AX^{\mr}_C]$ or simply by $[X^{\mr}]$.
\end{dfn}

\begin{rem}\label{RemHtpyEqDiff}
Let $X^{\mr},Y^{\mr}\in\CAC$ be any pair of objects. By Claim~\ref{ClaimDEA} {\rm (3)}, if a morphism $f^{\mr}\in\CAC(X^{\mr},Y^{\mr})$ gives a homotopy equivalence in $\CC$, then it is also a homotopy equivalence in $\CAC$.

However in general, a homotopy equivalence $g^{\mr}\in\CC(X^{\mr},Y^{\mr})$ does not necessarily give rise to a homotopy equivalence in $\CAC$, and thus there can be a difference between homotopic equivalences taken in $\CAC$ and in $\CC$. To distinguish, we use the notation $[X^{\mr}]$ exclusively for the homotopic equivalence class in $\CAC$.
\end{rem}

\begin{claim}\label{ClaimHtpyReduced}
Let $f^{\mr}\un{\vp^{\mr}}{\sim}g^{\mr}\co X^{\mr}\to Y^{\mr}$ be homotopic morphisms in $\CC$.
\begin{enumerate}
\item If $f^0=g^0$ and if $\C(X^2,A)\ov{-\ci d_X^1}{\lra}\C(X^1,A)\ov{-\ci d_X^0}{\lra}\C(X^0,A)$ is exact, then $\vp^{\mr}$ can be modified to satisfy $\vp^1=0$.
\item Dually, if $f^{n+1}=g^{n+1}$ and if $\C(C,Y^{n-1})\ov{d_Y^{n-1}\ci-}{\lra}\C(C,Y^n)\ov{d_Y^n\ci-}{\lra}\C(C,Y^{n+1})$ is exact, then $\vp^{\mr}$ can be modified to satisfy $\vp^{n+1}=0$.
\item If both assumptions of {\rm (1),(2)} are satisfied and if $n\ge 2$, then $\vp^{\mr}$ can be modified to satisfy $\vp^1=0$ and $\vp^{n+1}=0$.
\end{enumerate}
\end{claim}
\begin{proof}
We only show {\rm (1)}. By $\vp^1\ci d_X^0=g^0-f^0=0$ and the exactness of
\[ \C(X^2,Y^0)\ov{-\ci d_X^1}{\lra}\C(X^1,Y^0)\ov{-\ci d_X^0}{\lra}\C(X^0,Y^0), \]
there is $h\in\C(X^2,Y^0)$ which gives $h\ci d_X^1=\vp^1$. Then $(0,\vp^2+d_Y^0\ci h,\vp^3,\ldots,\vp^{n+1})$ gives the desired homotopy.
\end{proof}

Morphisms in $\CAC$ behave nicely with $n$-exangles. The following is obvious from the definition.
\begin{rem}\label{RemCAC}
Let $_A\del_C$ be any extension, and let $\Xd,\Yd$ be objects in $\AE$. Then any morphism $f^{\mr}\in\CAC(X^{\mr},Y^{\mr})$ gives a morphism $f^{\mr}\co\Xd\to\Yd$ in $\AE$.
\end{rem}

\begin{prop}\label{PropCAC1}
Let $_A\del_C$ be any extension, let $\Xd,\Yd$ be $n$-exangles, and let $f^{\mr}\in\CAC(X^{\mr},Y^{\mr})$ be any morphism.
If $\CAC(Y^{\mr},X^{\mr})\ne\emptyset$, then $f^{\mr}$ is a homotopy equivalence in $\CAC$.
\end{prop}
\begin{proof}
By the exactness of $\C(Y^n,X^n)\ov{-\ci d_X^n}{\lra}\C(Y^n,C)\ov{\del\ssh}{\lra}\E(Y^n,A)$ and $(d_Y^n)\uas\del=0$, there is $h\in\C(Y^n,X^n)$ which gives $d_X^n\ci h=d_Y^n$. By assumption, there is some $y^{\mr}\in\CAC(Y^{\mr},X^{\mr})$.

Put $\vp^1=0\in\C(X^1,X^0)$. By the exactness of
\[ \C(C,-)\ov{-\ci d_X^n}{\ltc}\C(X^n,-)\ov{-\ci d_X^{n-1}}{\ltc}\cdots\ov{-\ci d_X^0}{\ltc}\C(A,-)\ov{\del\ush}{\ltc}\E(C,-), \]
we obtain $\vp^i\in\C(X^i,X^{i-1})$ for $2\le i\le n+1$ satisfying
\[ \vp^{i+1}\ci d_X^i+d_X^{i-1}\ci\vp^i=\id-y^i\ci f^i\quad(1\le i\le n). \]
Then, since
\begin{eqnarray*}
d_X^n\ci\vp^{n+1}\ci d_Y^n&=&d_X^n\ci\vp^{n+1}\ci d_X^n\ci h\\
&=&d_X^n\ci(\id-y^n\ci f^n-d_X^{n-1}\ci\vp^n)\ci h\\
&=&(d_X^n-d_X^n\ci y^n\ci f^n)\ci h\ =\ 0,
\end{eqnarray*}
the sequence
\[ g^{\mr}=(y^0,y^1,\ldots,y^{n-1},y^n+\vp^{n+1}\ci d_Y^n,\id_C) \]
gives a morphism $g^{\mr}\in\CAC(Y^{\mr},X^{\mr})$. We can easily check that $g^{\mr}$ satisfies $g^{\mr}\ci f^{\mr}\sim\id$ for the homotopy $(\vp^1,\ldots,\vp^n,0)$. Thus $f^{\mr}$ has a left homotopy inverse $g^{\mr}$.

Applying the argument so far to $g^{\mr}$ instead of $f^{\mr}$, we see that $g^{\mr}$ also has a left homotopy inverse $f^{\prime\mr}$, which necessarily satisfies $\und{f^{\prime\mr}}=\und{f^{\mr}}$. This shows $\und{g^{\mr}}=(\und{f^{\mr}})\iv$.
\end{proof}

\subsection{Realization of extensions}\label{SecReal}

\begin{dfn}\label{DefReal}
Let $\sfr$ be a correspondence which associates a homotopic equivalence class $\sfr(\del)=[{}_AX^{\mr}_C]$ to each extension $\del={}_A\del_C$. Such $\sfr$ is called a {\it realization} of $\E$ if it satisfies the following condition for any $\sfr(\del)=[X^{\mr}]$ and any $\sfr(\rho)=[Y^{\mr}]$.
\begin{itemize}
\item[{\rm (R0)}] For any morphism of extensions $(a,c)\co\del\to\rho$, there exists a morphism $f^{\mr}\in\CC(X^{\mr},Y^{\mr})$ of the form $f^{\mr}=(a,f^1,\ldots,f^n,c)$. Such $f^{\mr}$ is called a {\it lift} of $(a,c)$.
\end{itemize}
In such a case, we abbreviately say that \lq\lq$X^{\mr}$ realizes $\del$" whenever they satisfy $\sfr(\del)=[X^{\mr}]$.

Moreover, a realization $\sfr$ of $\E$ is said to be {\it exact} if it satisfies the following conditions. 
\begin{itemize}
\item[{\rm (R1)}] For any $\sfr(\del)=[X^{\mr}]$, the pair $\Xd$ is an $n$-exangle. 
\item[{\rm (R2)}] For any $A\in\C$, the zero element ${}_A0_0=0\in\E(0,A)$ satisfies
\[ \sfr({}_A0_0)=[A\ov{\id_A}{\lra}A\to0\to\cdots\to0\to0]. \]
Dually, $\sfr({}_00_A)=[0\to0\to\cdots\to0\to A\ov{\id_A}{\lra}A]$ holds for any $A\in\C$.
\end{itemize}
By Proposition~\ref{PropEAInv} (and Remark~\ref{RemCAC}), the above condition {\rm (R1)} does not depend on representatives of the class $[X^{\mr}]$.
\end{dfn}

\begin{dfn}\label{Def223}
Let $\sfr$ be an exact realization of $\E$.
\begin{enumerate}
\item An $n$-exangle $\Xd$ is called a $\sfr$-{\it distinguished} $n$-exangle if it satisfies $\sfr(\del)=[X^{\mr}]$. We often simply say {\it distinguished $n$-exangle} when $\sfr$ is clear from the context.
\item An object $X^{\mr}\in\CC$ is called an {\it $\sfr$-conflation} or simply a {\it conflation} if it realizes some extension $\del\in\E(X^{n+1},X^0)$.
\item A morphism $f$ in $\C$ is called an {\it $\sfr$-inflation} or simply an {\it inflation} if it admits some conflation $X^{\mr}\in\CC$ satisfying $d_X^0=f$.
\item A morphism $g$ in $\C$ is called an {\it $\sfr$-deflation} or simply a {\it deflation} if it admits some conflation $X^{\mr}\in\CC$ satisfying $d_X^n=g$.
\end{enumerate}
\end{dfn}

\begin{lem}\label{LemDEA}
Let $_A\del_C$ be any extension, and let $\Xd,\Yd$ be $n$-exangles. If a morphism $f^{\mr}\in\AE(\Xd,\Yd)$ satisfies $f^{n+1}=\id_C$, then there is a morphism $f^{\prime\mr}$ which is homotopic to $f^{\mr}$ and belongs to $\CAC$.
\end{lem}
\begin{proof}
Since $f^{\mr}$ satisfies $(f^0)\sas\del=(f^{n+1})\uas\del=\del$, there exists $h\in\C(X^1,A)$ satisfying $h\ci d_X^0=\id-f^0$ by the exactness of
\[ \C(X^1,A)\ov{-\ci d_X^0}{\lra}\C(A,A)\ov{\del\ush}{\lra}\E(C,A). \]
If we modify $f^{\mr}$ by a homotopy $\vp^{\mr}=(h,0,\ldots,0)$, then the resulting morphism $f^{\prime\mr}$ satisfies the desired properties.
\end{proof}

\begin{prop}\label{PropDEAInv}
Let $\sfr$ be an exact realization of $\E$. Suppose that a morphism $f^{\mr}\in\AE(_A\Xd_C,{}_B\Yr_C)$ satisfies $f^{n+1}=\id_C$ and gives a homotopy equivalence in $\CC$. Then $\Xd$ is a distinguished $n$-exangle if and only if $\Yr$ is.
\end{prop}
\begin{proof}
By Claim~\ref{ClaimDEA}, there is a homotopy inverse $g^{\mr}\in\CC(Y^{\mr},X^{\mr})$ of $f^{\mr}$ satisfying $g^{n+1}=\id_C$, which gives a morphism $g^{\mr}\co\Yr\to\Xd$ by Proposition~\ref{PropAE1} {\rm (2)}.
Thus it suffices to show the \lq if' part, since the statement is symmetric in $\Xd$ and $\Yr$.

Assume that $\Yr$ is a distinguished $n$-exangle, and put $f^0=a, g^0=b$ for simplicity. By Proposition~\ref{PropEAInv}, the pair $\Xd$ is also an $n$-exangle.
Take $\sfr(\del)=[Z^{\mr}]$, to obtain a distinguished $n$-exangle $_A\Zd_C$. Since $\Yr$ is also a distinguished $n$-exangle, morphisms $(a,\id_C)\co\del\to\rho$ and $(b,\id_C)\co\rho\to\del$ have lifts $h^{\mr}\co\Zd\to\Yr$ and $\ell^{\mr}\co\Yr\to\Zd$.
Composing with $g^{\mr}$ and $f^{\mr}$, we obtain $g^{\mr}\ci h^{\mr}\in\AE(\Zd,\Xd)$ and $\ell^{\mr}\ci f^{\mr}\in\AE(\Xd,\Zd)$. Since $g^{n+1}\ci h^{n+1}=\id_C$, it is homotopic to a morphism $k^{\mr}\in\CAC(Z^{\mr},X^{\mr})$ by Lemma~\ref{LemDEA}. Similarly for $\ell^{\mr}\ci f^{\mr}$. Then by Proposition~\ref{PropCAC1}, we have $[X^{\mr}]=[Z^{\mr}]=\sfr(\del)$, which means that $\Xd$ is distinguished.
\end{proof}

\begin{cor}\label{CorDEAInv}
Let $\sfr$ be an exact realization of $\E$. For any distinguished $n$-exangle $_A\Xd_C$, i.e.
\[ A\ov{d_X^0}{\lra}X^1\ov{d_X^1}{\lra}\cdots\ov{d_X^{n-1}}{\lra}X^n\ov{d_X^n}{\lra}C\ov{\del}{\dra}, \]
the following holds.
\begin{enumerate}
\item For any isomorphisms $a\in\C(A,A\ppr)$ and $c\in\C(C\ppr,C)$,
\[ A\ppr\ov{d_X^0\ci a\iv}{\lra}X^1\ov{d_X^1}{\lra}\cdots\ov{d_X^{n-1}}{\lra}X^n\ov{c\iv\ci d_X^n}{\lra}C\ppr\ov{a\sas c\uas\del}{\dra} \]
is again a distinguished $n$-exangle.
\item If an object $\Yr\in\AE$ is isomorphic in $\AE$ to $\Xd$, then $\Yr$ is also a distinguished $n$-exangle.
\end{enumerate}
\end{cor}
\begin{proof}
{\rm (1)} We have the following sequence of isomorphisms $\Xd\ov{f^{\mr}}{\lra}\langle X^{\prime\mr},c\uas\del\rangle\ov{g^{\mr}}{\lra}\langle X^{\prime\prime\mr},a\sas c\uas\del\rangle$ in $\AE$.
\[
\xy
(-43,12)*+{A}="0";
(-28,12)*+{X^1}="1";
(-13,12)*+{X^2}="2";
(0,12)*+{\cdots}="3";
(14,12)*+{X^n}="4";
(29,12)*+{C}="5";
(41,12)*+{}="6";
(-43,0)*+{A}="10";
(-28,0)*+{X^1}="11";
(-13,0)*+{X^2}="12";
(0,0)*+{\cdots}="13";
(4,0)*+{}="13.5";
(14,0)*+{X^n}="14";
(29,0)*+{C\ppr}="15";
(41,0)*+{}="16";
(-43,-12)*+{A\ppr}="20";
(-28,-12)*+{X^1}="21";
(-13,-12)*+{X^2}="22";
(0,-12)*+{\cdots}="23";
(4,-12)*+{}="23.5";
(14,-12)*+{X^n}="24";
(29,-12)*+{C\ppr}="25";
(41,-12)*+{}="26";
{\ar^{d_X^0} "0";"1"};
{\ar^{d_X^1} "1";"2"};
{\ar^{d_X^2} "2";"3"};
{\ar^{d_X^{n-1}} "3";"4"};
{\ar^{d_X^n} "4";"5"};
{\ar@{-->}^{\del} "5";"6"};
{\ar@{=} "0";"10"};
{\ar@{=} "1";"11"};
{\ar@{=} "2";"12"};
{\ar@{=} "4";"14"};
{\ar^{c\iv} "5";"15"};
{\ar_{d_X^0} "10";"11"};
{\ar_{d_X^1} "11";"12"};
{\ar_{d_X^2} "12";"13"};
{\ar_{d_X^{n-1}} "13";"14"};
{\ar_{c\iv\ci d_X^n} "14";"15"};
{\ar@{-->}_{c\uas\del} "15";"16"};
{\ar_{a} "10";"20"};
{\ar@{=} "11";"21"};
{\ar@{=} "12";"22"};
{\ar@{=} "14";"24"};
{\ar@{=} "15";"25"};
{\ar_{d_X^0\ci a\iv} "20";"21"};
{\ar_{d_X^1} "21";"22"};
{\ar_{d_X^2} "22";"23"};
{\ar_{d_X^{n-1}} "23";"24"};
{\ar_{c\iv\ci d_X^n} "24";"25"};
{\ar@{-->}_(0.6){a\sas c\uas\del} "25";"26"};
{\ar@{}|\circlearrowright "0";"11"};
{\ar@{}|\circlearrowright "1";"12"};
{\ar@{}|\circlearrowright "2";"13.5"};
{\ar@{}|\circlearrowright "4";"15"};
{\ar@{}|\circlearrowright "10";"21"};
{\ar@{}|\circlearrowright "11";"22"};
{\ar@{}|\circlearrowright "12";"23.5"};
{\ar@{}|\circlearrowright "14";"25"};
\endxy
\]
Since $f^0=\id_A$, the middle row becomes a distinguished $n$-exangle by Proposition~\ref{PropDEAInv}. Then, since $g^{n+1}=\id_{C\ppr}$, the bottom row becomes a distinguished $n$-exangle by the dual of the same proposition.

{\rm (2)} Let $h^{\mr}=(h^0,h^1,\ldots,h^n)\co {}_A\Xd_C\to {}_B\Yr_D$ be an isomorphism. 
By {\rm (1)}, isomorphism $h^0\in\C(A,B)$ induces the following distinguished $n$-exangle.
\begin{equation}\label{DistBX}
B\ov{d_X^0\ci(h^0)\iv}{\lra}X^1\ov{d_X^1}{\lra}\cdots\to X^n\ov{d_X^n}{\lra}C\ov{(h^0)\sas\del}{\dra}
\end{equation}
Since $(\id_B,h^1,h^2,\ldots,h^n)$ gives an isomorphism from $(\ref{DistBX})$ to $\Yr$ in $\AE$, the dual of Proposition~\ref{PropDEAInv} shows that $\Yr$ becomes distinguished.
\end{proof}

\subsection{Definition of $n$-exangulated categories}

\begin{dfn}\label{DefMapCone}
For a morphism $f^{\mr}\in\CC(X^{\mr},Y^{\mr})$ satisfying $f^0=\id_A$ for some $A=X^0=Y^0$, its {\it mapping cone} $\Mf\in\CC$ is defined to be the complex
\[ X^1\ov{d_{M_f}^0}{\lra}X^2\oplus Y^1\ov{d_{M_f}^1}{\lra}X^3\oplus Y^2\ov{d_{M_f}^2}{\lra}\cdots\ov{d_{M_f}^{n-1}}{\lra}X^{n+1}\oplus Y^n\ov{d_{M_f}^n}{\lra}Y^{n+1} \]
where
\begin{eqnarray*}
d_{M_f}^0&=& \left[\begin{array}{c}-d_X^1\\ f^1\end{array}\right],\\
d_{M_f}^i&=& \left[\begin{array}{cc}-d_X^{i+1}&0\\ f^{i+1}&d_Y^i\end{array}\right]\quad(1\le i\le n-1),\\
d_{M_f}^n&=& \left[\begin{array}{cc}f^{n+1}&d_Y^n\end{array}\right].
\end{eqnarray*}

{\it Mapping cocone} is defined dually, for morphisms $h^{\mr}$ in $\CC$ satisfying $h^{n+1}=\id$.
\end{dfn}

\begin{prop}\label{PropMapCone}
Suppose that a diagram in $\CC$
\[
\xy
(-6,6)*+{X^{\mr}}="0";
(6,6)*+{Y^{\mr}}="2";
(-6,-6)*+{W^{\mr}}="4";
(6,-6)*+{Z^{\mr}}="6";
(0,-1)*+{\un{\vp^{\mr}}{\sim}}="10";
{\ar^{f^{\mr}} "0";"2"};
{\ar_{x^{\mr}} "0";"4"};
{\ar^{y^{\mr}} "2";"6"};
{\ar_{g^{\mr}} "4";"6"};
%
\endxy
\]
satisfies the following conditions.
\begin{itemize}
\item[{\rm (i)}] $x^{\mr}\in\CAC(X^{\mr},W^{\mr})$, with $X^0=W^0=A$ and $X^{n+1}=W^{n+1}=C$,
\item[{\rm (ii)}] $y^{\mr}\in\CAD(Y^{\mr},Z^{\mr})$, with $Y^0=Z^0=A$ and $Y^{n+1}=Z^{n+1}=D$,
\item[{\rm (iii)}] $f^0=g^0=\id_A$,
\item[{\rm (iv)}] $g^{\mr}\ci x^{\mr}\un{\vp^{\mr}}{\sim}y^{\mr}\ci f^{\mr}$ is a homotopy satisfying $\vp^1=0$.
\end{itemize}
Then the following holds for the mapping cones $\Mf$ and $\Mg$.
\begin{enumerate}
\item We have a morphism $F^{\mr}\in\CC(\Mf,\Mg)$ given by
\[ F^{\mr}=\Big(x^1,\left[\begin{array}{cc}x^2&0\\\vp^2&y^1\end{array}\right],\ldots,\left[\begin{array}{cc}x^{n+1}&0\\\vp^{n+1}&y^n\end{array}\right],\id_D\Big). \]
\item Assume that $X^{\mr},Y^{\mr},Z^{\mr},W^{\mr}$ satisfy the assumption of the exactness in Claim~\ref{ClaimHtpyReduced} {\rm (1)}. If $x^{\mr}$ and $y^{\mr}$ are homotopy equivalences in $\CAC$ and $\CAD$ respectively, then the above $F^{\mr}$ is a homotopy equivalence in $\CC$.
\end{enumerate}
\end{prop}
\begin{proof}
{\rm (1)} This is straightforward.

{\rm (2)} Let $w^{\mr}\in\CAC(W^{\mr},X^{\mr})$ and $z^{\mr}\in\CAD(Z^{\mr},Y^{\mr})$ be homotopy inverses of $x^{\mr}$ and $y^{\mr}$, with homotopies
\begin{eqnarray*}
x^{\mr}\ci w^{\mr}\un{\om^{\mr}}{\sim}\id_{W^{\mr}}&,&w^{\mr}\ci x^{\mr}\un{\xi^{\mr}}{\sim}\id_{X^{\mr}},\\
y^{\mr}\ci z^{\mr}\un{\zeta^{\mr}}{\sim}\id_{Z^{\mr}}&,&z^{\mr}\ci y^{\mr}\un{\eta^{\mr}}{\sim}\id_{Y^{\mr}}.
\end{eqnarray*}
As in Claim~\ref{ClaimHtpyReduced}, we may assume $\om^1=0,\xi^1=0,\eta^1=0,\zeta^1=0$. Then
\[ \psi^i=z^{i-1}\ci g^{i-1}\ci\om^i-z^{i-1}\ci\vp^i\ci w^i-\eta^i\ci f^i\ci w^i\quad(1\le i\le n+1) \]
gives a homotopy $f^{\mr}\ci w^{\mr}\un{\psi^{\mr}}{\sim}z^{\mr}\ci g^{\mr}$ satisfying $\psi^1=0$. Thus {\rm (1)} applied to
\[
\xy
(-6,6)*+{W^{\mr}}="0";
(6,6)*+{Z^{\mr}}="2";
(-6,-6)*+{X^{\mr}}="4";
(6,-6)*+{Y^{\mr}}="6";
(0,-1)*+{\un{\psi^{\mr}}{\sim}}="10";
{\ar^{g^{\mr}} "0";"2"};
{\ar_{w^{\mr}} "0";"4"};
{\ar^{z^{\mr}} "2";"6"};
{\ar_{f^{\mr}} "4";"6"};
%
\endxy
\]
gives a morphism $G^{\mr}\in\CC(\Mg,\Mf)$ defined in the same way as $F^{\mr}$. We can show that
\[ \Phi^{\mr}=\Big(\left[\begin{array}{cc}-\xi^2&0\end{array}\right],\left[\begin{array}{cc}-\xi^3&0\\ 0&\eta^2\end{array}\right],\ldots,\left[\begin{array}{cc}-\xi^{n+1}&0\\ 0&\eta^n\end{array}\right],\left[\begin{array}{c}0\\\eta^{n+1}\end{array}\right]\Big) \]
give a homotopy $G^{\mr}\ci F^{\mr}\un{\Phi^{\mr}}{\sim}I^{\mr}$ where $I^{\mr}\in C_{(X^1,D)}^{n+2}(\Mf,\Mf)$ is a morphism of the form
\[ I^{\mr}=\Big(\id,\left[\begin{array}{cc}1&0\\ a^2&1\end{array}\right],\ldots,\left[\begin{array}{cc}1&0\\ a^{n+1}&1\end{array}\right],\id\Big) \]
for some $a^i\in\C(X^i,Y^{i-1})$. Since $I^{\mr}$ is an isomorphism, this shows that $F^{\mr}$ has a left homotopy inverse.

Similarly, $\psi^{\mr}$ induces a homotopy $F^{\mr}\ci G^{\mr}\sim J^{\mr}$ to an isomorphism $J^{\mr}$, and $F^{\mr}$ also has a right homotopy inverse. Thus $F^{\mr}$ is a homotopy equivalence.
\end{proof}

\begin{prop}\label{PropMapCone_EA}
Let $f^{\mr}\co {}_A\Xd_C\to {}_A\Yr_D$ be a morphism in $\AE$, satisfying $f^0=\id_A$. Then $\langle\Mf,(d_X^0)\sas\rho\rangle$ also belongs to $\AE$.
\end{prop}
\begin{proof}
By the definition of $d_{M_f}^0$ and $d_{M_f}^{n+1}$, this follows from
\[ (d_X^1)\sas(d_X^0)\sas\rho=0,\quad (f^1)\sas(d_X^0)\sas\rho=(d_Y^0)\sas\rho=0 \]
and
\[ (f^{n+1})\uas(d_X^0)\sas\rho=(d_X^0)\sas\del=0,\quad (d_Y^n)\uas(d_X^0)\sas\rho=(d_X^0)\sas(d_Y^n)\uas\rho=0. \]
\end{proof}

\begin{cor}\label{CorMapConeHomotopic}
Let $f^{\mr},g^{\mr}\co {}_A\Xd_C\to {}_A\Yr_D$ be any pair of morphisms of $n$-exangles, satisfying $f^0=g^0=\id_A$. If $g^{\mr}\un{\vp^{\mr}}{\sim} f^{\mr}$ in $\CC$, then $\Mf\cong\Mg$ holds in $\Cbf^{n+2}_{(X^1,D)}$. In particular we have $[\Mf]=[\Mg]$.
\end{cor}
\begin{proof}
By Claim~\ref{ClaimHtpyReduced} {\rm (1)}, we may modify $\vp^{\mr}$ to satisfy $\vp^1=0$. Applying Proposition~\ref{PropMapCone} to
\[
\xy
(-6,6)*+{X^{\mr}}="0";
(6,6)*+{Y^{\mr}}="2";
(-6,-6)*+{X^{\mr}}="4";
(6,-6)*+{Y^{\mr}}="6";
(0,-1)*+{\un{\vp^{\mr}}{\sim}}="10";
{\ar^{f^{\mr}} "0";"2"};
{\ar@{=} "0";"4"};
{\ar@{=} "2";"6"};
{\ar_{g^{\mr}} "4";"6"};
%
\endxy,
\]
we obtain a homotopy equivalence
\[ F^{\mr}=\Big(\id_{X^1},\left[\begin{array}{cc}\id&0\\\vp^2&\id\end{array}\right],\ldots,\left[\begin{array}{cc}\id&0\\\vp^{n+1}&\id\end{array}\right],\id_D\Big)\co \Mf\to\Mg \]
in $\CC$, which is indeed an isomorphism.
\end{proof}

\begin{cor}\label{CorMapCone}
Let $f^{\mr}\co {}_A\Xd_C\to {}_A\Yr_D$ be a morphism of $n$-exangles, satisfying $f^0=\id_A$. If $w^{\mr}\in\CAC(W^{\mr},X^{\mr})$ and $y^{\mr}\in\CAD(Y^{\mr},Z^{\mr})$ are homotopy equivalences in $\CAC$ and $\CAD$ respectively, then the following holds for $g^{\mr}=y^{\mr}\ci f^{\mr}\ci w^{\mr}$.
\begin{enumerate}
\item If $\langle \Mf,(d_X^0)\sas\rho\rangle$ is an $n$-exangle, then so is $\langle \Mg,(d_W^0)\sas\rho\rangle$.
\item Moreover, if $\langle \Mf,(d_X^0)\sas\rho\rangle$ is distinguished, so is $\langle \Mg,(d_W^0)\sas\rho\rangle$.
\end{enumerate}
\end{cor}
\begin{proof}
By Proposition~\ref{PropEAInv} (and Remark~\ref{RemCAC}), the pairs $\langle W^{\mr},\del\rangle$ and $\langle Z^{\mr},\rho\rangle$ are $n$-exangles. Let $x^{\mr}\in\CAC(X^{\mr},W^{\mr})$ be a homotopy inverse of $w^{\mr}$, and take a homotopy $w^{\mr}\ci x^{\mr}\un{\xi^{\mr}}{\sim}\id_{X^{\mr}}$. If we define $\vp^{\mr}$ by
\[ \vp^i=y^{i-1}\ci f^{i-1}\ci\xi^i\quad(1\le i\le n+1), \]
this gives a homotopy $g^{\mr}\ci x^{\mr}\un{\vp^{\mr}}{\sim}y^{\mr}\ci f^{\mr}$. Since $\langle Z^{\mr},\rho\rangle$ is an $n$-exangle, we may assume $\vp^1=0$ by Claim~\ref{ClaimHtpyReduced}. Then by Proposition~\ref{PropMapCone}, we obtain a homotopy equivalence $F^{\mr}\in\CC(\Mf,\Mg)$ satisfying $F^0=x^1$ and $F^{n+1}=\id_D$. By $(x^1)\sas (d_X^0)\sas\rho=(d_W^0)\sas\rho$, this gives a morphism $F^{\mr}\in\AE(\langle \Mf,(d_X^0)\sas\rho\rangle,\langle \Mg,(d_W^0)\sas\rho\rangle)$.
Thus {\rm (1)} follows from Proposition~\ref{PropEAInv}, and {\rm (2)} follows from Proposition~\ref{PropDEAInv}.
\end{proof}

\begin{dfn}\label{DefEACat}
An {\it $n$-exangulated category} is a triplet $\CEs$ of additive category $\C$, biadditive functor $\E\co\C\op\ti\C\to\Ab$, and its exact realization $\sfr$, satisfying the following conditions.
\begin{itemize}
\item[{\rm (EA1)}] Let $A\ov{f}{\lra}B\ov{g}{\lra}C$ be any sequence of morphisms in $\C$. If both $f$ and $g$ are inflations, then so is $g\ci f$. Dually, if $f$ and $g$ are deflations then so is $g\ci f$.

\item[{\rm (EA2)}] For $\rho\in\E(D,A)$ and $c\in\C(C,D)$, let ${}_A\langle X^{\mr},c\uas\rho\rangle_C$ and ${}_A\Yr_D$ be distinguished $n$-exangles. Then $(\id_A,c)$ has a {\it good lift} $f^{\mr}$, in the sense that its mapping cone gives a distinguished $n$-exangle $\langle \Mf,(d_X^0)\sas\rho\rangle$.
\item[{\rm (EA2$\op$)}] Dual of {\rm (EA2)}.
\end{itemize}
\end{dfn}

\begin{rem}\label{RemEA2}
Concerning {\rm (EA2)}, the following holds. Similarly for {\rm (EA2$\op$)}.
\begin{enumerate}
\item By Corollary \ref{CorMapConeHomotopic}, if $g^{\mr}\sim f^{\mr}\co {}_A\Xd_C\to {}_A\Yr_D$ are lifts of $(\id_A,c)$, then $f^{\mr}$ is a good lift if and only if $g^{\mr}$ is.
\item By Corollary \ref{CorMapCone}, condition {\rm (EA2)} is independent from representatives of the classes $[X^{\mr}]$ and $[Y^{\mr}]$. 
\end{enumerate}
\end{rem}

\begin{dfn}\label{DefnExtClosed}
Let $\CEs$ be an $n$-exangulated category. A full subcategory $\Ccal\ppr\se\C$ is called extension-closed if for any $A,C\in\Ccal\ppr$ and any $\del\in\E(C,A)$, there exists a $\sfr$-distinguished $n$-exangle $\Xd$ which satisfies $X^i\in\Ccal\ppr$ for any $1\le i\le n$.
\end{dfn}

\begin{prop}\label{PropnExtClosed}
Let $\CEs$ be an $n$-exangulated category, and $\Ccal\ppr\se\C$ be an extension-closed subcategory.

Let $\E\ppr$ denote the restriction of $\E$ to $\Ccal^{\prime\mathrm{op}}\ti\Ccal\ppr$. For any $\del\in\E\ppr(C,A)$, take an $\sfr$-distinguished $n$-exangle $\Xd$ which satisfies $X^i\in\Ccal\ppr$ for any $1\le i\le n$, and put $\tfr(\del)=[X^{\mr}]$. Here, the homotopy equivalence class is taken in $\Cbf_{(\Ccal\ppr;A,C)}^{n+2}$, in the notation of Definition \ref{DefSeq}.
Then the following holds.
\begin{enumerate}
\item $\tfr$ gives an exact realization of $\E\ppr$, and $(\Ccal\ppr,\E\ppr,\tfr)$ satisfies {\rm (EA2), (EA2$\op$)}.
\item If $(\Ccal\ppr,\E\ppr,\tfr)$ moreover satisfies {\rm (EA1)}, then it is an $n$-exangulated category.
\end{enumerate}
\end{prop}
\begin{proof}
This follows immediately from the definition.
\end{proof}

\section{Fundamental properties}\label{section_Fund}
\subsection{Fundamental properties of $n$-exangulated category}
We summarize here some properties of $n$-exangulated category, which will be used in the proceeding sections. Let $\CEs$ be an $n$-exangulated category, throughout this section.
\begin{prop}\label{PropZeroEA}
Let $_A\del_C$ be an extension. Suppose that for any $Q\in\C$,
\[ \del\ssh\co\C(Q,C)\to\E(Q,A)\quad\text{and}\quad \del\ush\co\C(A,Q)\to\E(C,Q) \]
are monomorphic.
Then, the following holds for any $n$-exangle $\Xd$.
\begin{enumerate}
\item $d_X^0=0$ and $d_X^n=0$.
\item $X^{\mr}$ is homotopically equivalent in $\CAC$ to the object
\[ A\ov{0}{\lra}0\ov{0}{\lra}\cdots \ov{0}{\lra}0\ov{0}{\lra}C, \]
which will be denoted by $\O^{\mr}={}_A\O^{\mr}_C$ in the rest.
\end{enumerate}
In particular, such $\del$ should satisfy $\sfr(\del)=[\O^{\mr}]$.
\end{prop}
\begin{proof}
{\rm (1)} This immediately follows from $\del\ush(d_X^0)=0$ and $\del\ssh(d_X^n)=0$.

{\rm (2)} Remark that the assumption of the monomorphicity of $\del\ssh$ and $\del\ush$ is equivalent to that $\langle\O^{\mr},\del\rangle$ is an $n$-exangle. Since there are morphisms
\begin{eqnarray*}
&f^{\mr}=(\id_A,0,\ldots,0,\id_C)\co X^{\mr}\to\O^{\mr},&\\
&g^{\mr}=(\id_A,0,\ldots,0,\id_C)\co \O^{\mr}\to X^{\mr}&
\end{eqnarray*}
in $\CAC$ by {\rm (1)}, these are homotopy equivalences by Proposition~\ref{PropCAC1}.
\end{proof}

\begin{prop}\label{PropSummandOut}
Let $_A\Xd_C,{}_B\Yr_D$ be any pair of objects in $\AE$. Then the following are equivalent.
\begin{enumerate}
\item $\langle X^{\mr}\oplus Y^{\mr},\del\oplus\rho\rangle$ is a distinguished $n$-exangle.
\item Both $\Xd$ and $\Yr$ are distinguished $n$-exangles.
\end{enumerate}
\end{prop}
\begin{proof}
As in Remark~\ref{RemEACoprod}, $\langle X^{\mr}\oplus Y^{\mr},\del\oplus\rho\rangle$ is an $n$-exangle if and only if $\Xd$ and $\Yr$ are $n$-exangles. 

$(1)\Rightarrow(2)$. 
Put $\sfr(\del)=[Z^{\mr}]$, and let us show that $[X^{\mr}]=[Z^{\mr}]$ holds in $\CAC$. For simplicity, for any pair of objects $I,J\in\C$, denote the inclusion and projection to the 1st component by
\[ j_I=\begin{bmatrix}1\\0\end{bmatrix}\co I\to I\oplus J\quad\text{and}\quad p_I=[1\ 0]\co I\oplus J\to I, \]
respectively.
Since $(j_A,j_C)\co\del\to\del\oplus\rho$ and $(p_A,p_C)\co\del\oplus\rho\to\del$ are morphisms of extensions, they have lifts $f^{\mr}\in\CC(Z^{\mr},X^{\mr}\oplus Y^{\mr})$ and $g^{\mr}\in\CC(X^{\mr}\oplus Y^{\mr},Z^{\mr})$. If we compose them with
\[ p^{\mr}=(p_A,p_{X^1},\ldots,p_C)\in\CC(X^{\mr}\oplus Y^{\mr},X^{\mr}) \]
and
\[ j^{\mr}=(j_A,j_{X^1},\ldots,j_C)\in\CC(X^{\mr},X^{\mr}\oplus Y^{\mr}) \]
respectively, we obtain $p^{\mr}\ci f^{\mr}\in\CAC(Z^{\mr},X^{\mr})$ and $g^{\mr}\ci j^{\mr}\in\CAC(X^{\mr},Z^{\mr})$. Thus Proposition~\ref{PropCAC1} shows $[X^{\mr}]=[Z^{\mr}]$. Similarly for $\sfr(\rho)=[Y^{\mr}]$.

$(1)\Rightarrow(2)$. 
Put $\sfr(\del\oplus\rho)=[W^{\mr}]$, and let us show $[X^{\mr}\oplus Y^{\mr}]=[W^{\mr}]$. Let $x^{\mr}\in\CC(X^{\mr},W^{\mr})$ and $u^{\mr}\in\CC(W^{\mr},X^{\mr})$ be lifts of $(j_A,j_C)$ and $(p_A,p_C)$, respectively. Similarly, let $y^{\mr}\in\CC(Y^{\mr},W^{\mr})$ and $v^{\mr}\in\CC(W^{\mr},Y^{\mr})$ be lifts of $\big(\Big[\raise1ex\hbox{\leavevmode\vtop{\baselineskip-8ex \lineskip1ex \ialign{#\crcr{$\scriptstyle{0}$}\crcr{$\scriptstyle{1}$}\crcr}}}\Big],\Big[\raise1ex\hbox{\leavevmode\vtop{\baselineskip-8ex \lineskip1ex \ialign{#\crcr{$\scriptstyle{0}$}\crcr{$\scriptstyle{1}$}\crcr}}}\Big]\big)$ and $([0\ 1],[0\ 1])$. Then
\begin{eqnarray*}
&(\id,[x^1\ y^1],\ldots,[x^n\ y^n],\id)\co X^{\mr}\oplus Y^{\mr}\to W^{\mr},&\\
&\big(\id,\begin{bmatrix}u^1\\ v^1\end{bmatrix},\ldots,\begin{bmatrix}u^n\\ v^n\end{bmatrix},\id\big)\co W^{\mr}\to X^{\mr}\oplus Y^{\mr}&
\end{eqnarray*}
are morphisms in $\Cbf^{n+2}_{(A\oplus B,C\oplus D)}$. Proposition~\ref{PropCAC1} shows $[X^{\mr}\oplus Y^{\mr}]=[W^{\mr}]$.
\end{proof}

The following is an analog of \cite[Lemma 5]{Hu}.
\begin{cor}\label{CorSummandOut}
Suppose that
\begin{equation}\label{DEAXA1}
X^0\oplus A\ov{d}{\lra}X^1\oplus A\ov{[d_X^1\ w]}{\lra}X^2\ov{d_X^2}{\lra}\cdots\ov{d_X^n}{\lra}X^{n+1}\ov{\thh}{\dra}
\end{equation}
is a distinguished $n$-exangle, where $d$ is as follows.
\[ d=\left[\begin{array}{cc}x&u\\ v&\id\end{array}\right]\in\C(X^0\oplus A,X^1\oplus A) \]
Then for $d_X^0=x-u\ci v$ and $p=[1\ 0]\co X^0\oplus A\to X^0$,
\begin{equation}\label{DEAX}
X^0\ov{d_X^0}{\lra}X^1\ov{d_X^1}{\lra}X^2\ov{d_X^2}{\lra}\cdots\ov{d_X^n}{\lra}X^{n+1}\ov{p\sas\thh}{\dra}
\end{equation}
becomes a distinguished $n$-exangle.
\end{cor}
\begin{proof}
For $p$ and $q=[0\ 1]\co X^0\oplus A\to A$, put $\del=p\sas\thh$ and $\rho=q\sas\thh$. Then $\thh$ corresponds to $\begin{bmatrix}\del\\ \rho\end{bmatrix}$ through the natural isomorphism
\begin{equation}\label{EIso}
\E(X^{n+1},X^0\oplus A)\cong\E(X^{n+1},X^0)\oplus\E(X^{n+1},A),
\end{equation}
and the equality $d\sas\thh=0$ implies $v\sas\del+\rho=0$.
Thus
\[ \Big( a=\left[\begin{array}{cc}\id&0\\ v&\id\end{array}\right],\, b=\left[\begin{array}{cc}\id&-u\\ 0&\id\end{array}\right],\, \id,\id,\ldots,\id\Big) \]
gives an isomorphism in $\AE$ from $(\ref{DEAXA1})$ to
\begin{equation}\label{DEAXA2}
X^0\oplus A\ov{d_X^0\oplus\id_A}{\lra}X^1\oplus A\ov{[d_X^1\ 0]}{\lra}X^2\ov{d_X^2}{\lra}\cdots\ov{d_X^n}{\lra}X^{n+1}\ov{a\sas\thh}{\dra},
\end{equation}
with $a\sas\thh$ corresponding to $\begin{bmatrix}\del\\ 0\end{bmatrix}$ through $(\ref{EIso})$. Since $(\ref{DEAXA2})$ is isomorphic to a coproduct of $(\ref{DEAX})$ and
\[ A\ov{\id_A}{\lra}A\to0\to\cdots\to0\ov{0}{\dra} \]
in $\AE$, Corollary \ref{CorDEAInv} and Proposition~\ref{PropSummandOut} shows that $(\ref{DEAX})$ is also a distinguished $n$-exangle.
\end{proof}

\begin{lem}\label{LemOneExact}
For any distinguished $n$-exangle $_A\Xd_C$, the following holds.
\begin{enumerate}
\item $\C(-,C)\ov{\del\ssh}{\ltc}\E(-,A)\ov{(d_X^0)\sas}{\ltc}\E(-,X^1)$ is exact.
\item $\C(A,-)\ov{\del\ush}{\ltc}\E(C,-)\ov{(d_X^n)\uas}{\ltc}\E(X^n,-)$ is exact.
\end{enumerate}
\end{lem}
\begin{proof}
We only show {\rm (1)}, since {\rm (2)} can be shown dually. Let us show the exactness of
\begin{equation}\label{ToShowDEx}
\C(D,C)\ov{\del\ssh}{\lra}\E(D,A)\ov{(d_X^0)\sas}{\lra}\E(D,X^1)
\end{equation}
for any $D\in\C$. Suppose that $\thh\in\E(D,A)$ satisfies $(d_X^0)\sas\thh=0$. Put $\sfr(\thh)=[Y^{\mr}]$, to obtain a distinguished $n$-exangle $_A\langle Y^{\mr},\thh\rangle_D$. By Proposition~\ref{PropSummandOut}, coproduct $_{A\oplus A}\langle X^{\mr}\oplus Y^{\mr},\del\oplus\thh\rangle_{C\oplus D}$ is also a distinguished $n$-exangle.

Let $\nabla_A=[\id\ \id]\co A\oplus A\to A$ be the folding morphism. Put $\mu=(\nabla_A)\sas(\del\oplus\thh)$ and $\sfr(\mu)=[Z^{\mr}]$, to obtain a distinguished $n$-exangle $_A\langle Z^{\mr},\mu\rangle_{C\oplus D}$. If we write $d_Z^n=\begin{bmatrix}k\\ \ell\end{bmatrix}\co Z^n\to C\oplus D$, then $(d_Z^n)\uas\mu=0$ means
\begin{equation}\label{kellzero}
k\uas\del+\ell\uas\thh=0.
\end{equation}
Since $\begin{bmatrix}1\\ 0\end{bmatrix}\co C\to C\oplus D$ satisfies
$\begin{bmatrix}1\\ 0\end{bmatrix}\uas\!\!\mu=(\nabla_A)\sas\begin{bmatrix}1\\ 0\end{bmatrix}\uas\!\!(\del\oplus\thh)=\del$,
we have a morphism of extensions $\big(\id_A,\begin{bmatrix}1\\ 0\end{bmatrix}\big)\co\del\to\mu$. By {\rm (EA2)}, it has a good lift $f^{\mr}\co\Xd\to\langle Z^{\mr},\mu\rangle$, which gives a distinguished $n$-exangle $\langle M_f^{\mr},(d_X^0)\sas\mu\rangle$. By definition, the last two terms $M_f^n\ov{d_{M_f}^n}{\lra}M_f^{n+1}$ of $M_f^{\mr}$ is
\[ C\oplus Z^n\ov{\left[\begin{array}{cc}1&k\\0&\ell\end{array}\right]}{\lra}C\oplus D. \]
Remark that the assumption $(d_X^0)\sas\thh=0$ shows
\[ (d_X^0)\sas\mu=(d_X^0)\sas(\nabla_A)\sas(\del\oplus\thh)=[d_X^0\ d_X^0]\sas(\del\oplus\thh)=0. \]
Thus by Claim~\ref{ClaimSplit}, morphism $d_{M_f}^n$ has a section $s=\left[\begin{array}{cc}p&q\\ r&t\end{array}\right]\co C\oplus D\to C\oplus Z^n$, and the equality $d_{M_f}^n\ci s=\id$ implies in particular $q+k\ci t=0$ and $\ell\ci t=\id_D$.
Then $q\in\C(D,C)$ satisfies
\[ \del\ssh(q)=q\uas\del=-(k\ci t)\uas\del=t\uas(-k\uas\del)=t\uas\ell\uas\thh=\thh \]
by $(\ref{kellzero})$. This shows the exactness of $(\ref{ToShowDEx})$.
\end{proof}

The following is a consequence of {\rm (R0)} and {\rm (EA2)}.
\begin{prop}\label{PropShiftedReal}
Let $_A\Xd_C$ and $_B\Yr_D$ be distinguished $n$-exangles. Suppose that we are given a commutative square
\[
\xy
(-6,6)*+{X^0}="0";
(6,6)*+{X^1}="2";
(-6,-6)*+{Y^0}="4";
(6,-6)*+{Y^1}="6";
{\ar^{d_X^0} "0";"2"};
{\ar_{a} "0";"4"};
{\ar^{b} "2";"6"};
{\ar_{d_Y^0} "4";"6"};
{\ar@{}|\circlearrowright "0";"6"};
\endxy
\]
in $\C$. Then the following holds.
\begin{enumerate}
\item There is a morphism $f^{\mr}\co\Xd\to\Yr$ which satisfies $f^0=a$ and $f^1=b$.
\item If $X^0=Y^0=A$ and $a=\id_A$ for some $A\in\C$, then the above $f^{\mr}$ can be taken to give a distinguished $n$-exangle $\langle M_f^{\mr},(d_X^0)\sas\rho\rangle$.
\end{enumerate}
\end{prop}
\begin{proof}
By Lemma~\ref{LemOneExact}, $\C(C,D)\ov{\rho\ssh}{\lra}\E(C,B)\ov{(d_Y^0)\sas}{\lra}\E(C,Y^1)$ is exact. Thus by $(d_Y^0)\sas(a\sas\del)=b\sas(d_X^0)\sas\del=0$, there is $c\in\C(C,D)$ satisfying $\rho\ssh(c)=a\sas\del$. This gives a morphism of extensions $(a,c)\co\del\to\rho$.

By {\rm (R0)}, it has a lift $g^{\mr}\co\Xd\to\Yr$. Then by the exactness of
\begin{equation}\label{ExXY}
\C(X^2,Y^1)\ov{-\ci d_X^1}{\lra}\C(X^1,Y^1)\ov{-\ci d_X^0}{\lra}\C(X^0,Y^1)
\end{equation}
and $(b-g^1)\ci d_X^0=0$, there is $m\in\C(X^2,Y^1)$ which gives $m\ci d_X^1=b-g^1$.

Modifying $g^{\mr}$ by the homotopy $\vp^{\mr}=(0,m,0,\ldots,0)$, we obtain a morphism $f^{\mr}=(a,b,g^2+d_Y^2\ci m,g^3,\ldots,g^n,g^{n+1})$, which satisfies the desired condition.

{\rm (2)} The same construction as {\rm (1)} works, except for that we take a good lift in the second step. Indeed as above, there is $c\in\C(C,D)$ which gives a morphism $(\id_A,c)\co\del\to\rho$. By {\rm (EA2)}, it has a good lift $g^{\mr}\co\Xd\to\Yr$, which makes $\langle M_g^{\mr},(d_X^0)\sas\rho\rangle$ a distinguished $n$-exangle. Then the exactness of $(\ref{ExXY})$ gives a homotopy $\vp^{\mr}=(0,m,0,\ldots,0)$ from $g^{\mr}$ to a morphism $f^{\mr}$ satisfying $f^0=\id_A$ and $f^1=b$.
By Corollary \ref{CorMapConeHomotopic}, it follows that $\langle M_g^{\mr},(d_X^0)\sas\rho\rangle$ is also a distinguished $n$-exangle.
\end{proof}

\subsection{Relative theory}

In this subsection, we study relative theory for $n$-exangulated categories and show that $n$-exangulated structures can be inherited in any relative theory. This is just an $n$-exangulated analog of the argument for exact categories in \cite[Sections 1.2 and 1.3]{DRSSK}.
We also remark that subfunctors $\F\se\E$ for extriangulated categories, namely, in the case where $n=1$ (see Proposition \ref{Prop1EAandET}), are investigated in \cite{ZH}.

\begin{dfn}
Let $\C$ be a category, and let $\E\co\C\op\ti\C\to\Ab$ be a biadditive functor.
\begin{enumerate}
\item A functor $\F\co\C\op\ti\C\to\mathit{Set}$ is called a {\it subfunctor} of $\E$ if it satisfies the following conditions.
\begin{itemize}
\item $\F(C,A)$ is a subset of $\E(C,A)$, for any $A,C\in\C$.
\item $\F(c,a)=\E(c,a)|_{\F(C,A)}$ holds, for any $a\in\C(A,A\ppr)$ and $c\in\C(C\ppr,C)$.
\end{itemize}
In this case, we write as $\F\se\E$.
\item A subfunctor $\F\se\E$ is said to be an {\it additive subfunctor} if $\F(C,A)\se\E(C,A)$ is an abelian subgroup for any $A,C\in\C$. In this case, $\F\co\C\op\ti\C\to\Ab$ itself becomes a biadditive functor.
\end{enumerate}
\end{dfn}

\begin{dfn}
Let $\F\se\E$ be an additive subfunctor. For a realization $\sfr$ of $\E$, define $\sfr|_{\F}$ to be the restriction of $\sfr$ onto $\F$. Namely, it is defined by $\sfr|_{\F}(\del)=\sfr(\del)$ for any $\F$-extension $\del$.
\end{dfn}

\begin{claim}\label{ClaimPreExTr}
Let $\CEs$ be an $n$-exangulated category, and let $\F\se\E$ be an additive subfunctor. Then $\sfr|_{\F}$ is an exact realization of $\F$. Moreover, the triplet $\CFs$ satisfies condition {\rm (EA2)}.
\end{claim}
\begin{proof}
This immediately follows from the definitions of these conditions.
\end{proof}

Thus we may speak of $\sfr|_{\F}$-conflations (resp. $\sfr|_{\F}$-inflations, $\sfr|_{\F}$-deflations) and $\sfr|_{\F}$-distinguished $n$-exangles as in Definition \ref{Def223}.
The following condition on $\F\se\E$ gives a necessary and sufficient condition for $\CFs$ to be an $n$-exangulated category, as will be shown in Proposition \ref{PropClosed}.

\begin{dfn}\label{DefClosed}$($cf.\cite{DRSSK}$)$
Let $\F\se\E$ be a additive subfunctor.
\begin{enumerate}
\item $\F\se\E$ is {\it closed on the right} if
\[ \F(-,X^0)\ov{(d_X^0)\sas}{\ltc}\F(-,X^1)\ov{(d_X^1)\sas}{\ltc}\F(-,X^2) \]
is exact for any $\sfr|_{\F}$-conflation $X^{\mr}$.
\item $\F\se\E$ is {\it closed on the left} if
\[ \F(X^{n+1},-)\ov{(d_X^n)\uas}{\ltc}\F(X^n,-)\ov{(d_X^{n-1})\uas}{\ltc}\F(X^{n-1},-) \]
is exact for any $\sfr|_{\F}$-conflation $X^{\mr}$.
\end{enumerate}
\end{dfn}

\begin{prop}\label{PropComposAndClosed}
Let $\CEs$ be any $n$-exangulated category. If $\sfr|_{\F}$-inflations are closed by composition, then $\F\se\E$ is closed on the right. Dually, if $\sfr|_{\F}$-deflations are closed by composition, then $\F\se\E$ is closed on the left.
\end{prop}
\begin{proof}
We only show the first statement. Suppose that $\sfr|_{\F}$-inflations are closed by composition. Let ${}_A\Xd_C$ be any $\sfr|_{\F}$-distinguished $n$-exangle, and let us show the exactness of
\[ \F(F,X^0)\ov{(d_X^0)\sas}{\lra}\F(F,X^1)\ov{(d_X^1)\sas}{\lra}\F(F,X^2) \]
for any $F\in\C$. Since $(d_X^1)\sas\ci(d_X^0)\sas=0$ follows from $d_X^1\ci d_X^0=0$, it is enough to show $\Ker (d_X^1)\sas\se\Ima (d_X^0)\sas$.

Let $\thh\in\Ker (d_X^1)\sas$ be any element, and let ${}_B\langle Y^{\mr},\thh\rangle_F$ be an $\sfr|_{\F}$-distinguished $n$-exangle realizing it as follows, where we put $X^1=B$.
\[B\ov{d_Y^0}{\lra}Y^1\ov{d_Y^1}{\lra}Y^2\to\cdots\to Y^n\ov{d_Y^n}{\lra}F\ov{\thh}{\dra}. \]
Since $d_X^0$ and $d_Y^0$ are $\sfr|_{\F}$-inflations, their composition $d_Y^0\ci d_X^0$ becomes an $\sfr|_{\F}$-inflation by assumption. Thus there is some $\sfr|_{\F}$-distinguished $n$-exangle $_A\langle Z^{\mr},\tau\rangle_D$ which satisfies $Z^1=Y^1$ and $d_Z^0=d_Y^0\ci d_X^0$ as follows.
\[ A\ov{d_Y^0\ci d_X^0}{\lra}Y^1\ov{d_Z^1}{\lra}Z^2\ov{d_Z^2}{\lra}\cdots\to Z^n\ov{d_Z^n}{\lra}D\ov{\tau}{\dra}. \]
By Proposition \ref{PropShiftedReal} applied to the following commutative diagram,
\[
\xy
(-6,6)*+{A}="0";
(6,6)*+{X^1}="2";
(-6,-6)*+{A}="4";
(6,-6)*+{Y^1}="6";
{\ar^{d_X^0} "0";"2"};
{\ar@{=} "0";"4"};
{\ar^{d_Y^0} "2";"6"};
{\ar_{d_Z^0} "4";"6"};
{\ar@{}|\circlearrowright "0";"6"};
\endxy
\]
we find a morphism of $n$-exangles $f^{\mr}\co\Xd\to\langle Z^{\mr},\tau\rangle$ which satisfies
\[ f^0=\id_A,\quad f^1=d_Y^0 \]
and makes $_B\langle M_f^{\mr},(d_X^0)\sas\tau\rangle_D$ an $\sfr$-distinguished $n$-exangle.
Since $(d_X^0)\sas\tau\in\F(D,B)$, this is an $\sfr|_{\F}$-distinguished $n$-exangle.
Then by Proposition \ref{PropShiftedReal} applied to the diagram
\[
\xy
(-9,6)*+{B}="0";
(9,6)*+{X^2\oplus Y^1}="2";
(-9,-6)*+{B}="4";
(9,-6)*+{Y^1}="6";
{\ar^(0.4){d} "0";"2"};
{\ar@{=} "0";"4"};
{\ar^{[0\ 1]} "2";"6"};
{\ar_{d_Y^0} "4";"6"};
{\ar@{}|\circlearrowright "0";"6"};
\endxy
\]
where we put $d=\begin{bmatrix}-d_X^1\\ d_Y^0\end{bmatrix}\co B\to X^2\oplus Y^1$, we find a morphism of $n$-exangles $g^{\mr}\co\langle M_f^{\mr},(d_X^0)\sas\tau\rangle\to\langle Y^{\mr}\thh\rangle$ which satisfies 
\[ g^0=\id_B,\quad g^1=[0\ 1] \]
and makes $\langle M_g^{\mr},d\sas\thh\rangle$ an $\sfr|_{\F}$-distinguished $n$-exangle.
In particular it satisfies
\begin{equation}\label{Eq_theta_tau}
(g^{n+1})\uas\thh=(d_X^0)\sas\tau.
\end{equation}
By definition, this $n$-exangle $\langle M_g^{\mr},d\sas\thh\rangle$ is of the following form.
\[ X^2\oplus Y^1\to X^3\oplus Z^2\oplus Y^1\to X^4\oplus Z^3\oplus Y^2\to\cdots\to D\oplus Y^n\ov{[g^{n+1}\ d_Y^n]}{\lra}F\ov{d\sas\thh}{\dra} \]
Then $d\sas\thh=0$ follows from $(d_X^1)\sas\thh=0$ and $(d_Y^0)\sas\thh=0$, which means that $[g^{n+1}\ d_Y^n]$ has a section $\begin{bmatrix}s_1\\ s_2\end{bmatrix}\co F\to D\oplus Y^n$. If we put $\be=s_1\uas\tau\in\F(F,A)$, then the equalities $g^{n+1}\ci s_1+d_Y^n\ci s_2=\id_F$ and $(\ref{Eq_theta_tau})$ show
\[ \thh=s_1\uas(g^{n+1})\uas\thh+s_2\uas(d_Y^n)\uas\thh=s_1\uas(d_X^0)\sas\tau=(d_X^0)\sas\be, \]
and thus $\thh\in\Ima\big( \F(F,X^0)\ov{(d_X^0)\sas}{\lra}\F(F,X^1)\big)$ holds.
\end{proof}

\begin{cor}\label{CorEisClosed}
For any $n$-exangulated category $\CEs$, the sequences
\[ \E(-,X^0)\ov{(d_X^0)\sas}{\ltc}\E(-,X^1)\ov{(d_X^1)\sas}{\ltc}\E(-,X^2) \]
and
\[ \E(X^{n+1},-)\ov{(d_X^n)\uas}{\ltc}\E(X^n,-)\ov{(d_X^{n-1})\uas}{\ltc}\E(X^{n-1},-) \]
are exact for any $\sfr$-conflation $X^{\mr}$.
\end{cor}
\begin{proof}
This immediately follows from Proposition \ref{PropComposAndClosed} applied to $\F=\E$, as $\CEs$ satisfies condition {\rm (EA1)}.
\end{proof}

\begin{lem}\label{LemRightClosed}
Assume that an additive subfunctor $\F\se\E$ is closed on the right. Let $\thh\in\E(C,A)$ be any $\E$-extension. If it satisfies $x\sas\thh\in\F(C,B)$ for some $\sfr|_{\F}$-inflation $x\in\C(A,B)$, then $\thh\in\F(C,A)$ follows.
\end{lem}
\begin{proof}
Take an $\sfr|_{\F}$-distinguished $n$-exangle $_A\Xd$ satisfying $d_X^0=x$. We have the following commutative diagram, whose bottom row is exact by Lemma \ref{LemOneExact} and Corollary \ref{CorEisClosed}.
\[
\xy
(-11,7)*+{\F(C,A)}="2";
(11,7)*+{\F(C,B)}="4";
(33,7)*+{\F(C,X^2)}="6";
(-35,-7)*+{\C(C,X^{n+1})}="10";
(-23,4)*+{}="11";
(-11,-7)*+{\E(C,A)}="12";
(11,-7)*+{\E(C,B)}="14";
(33,-7)*+{\E(C,X^2)}="16";
{\ar^{\del\ssh} "10";"2"};
{\ar^{x\sas} "2";"4"};
{\ar^{(d_X^1)\sas} "4";"6"};
{\ar@{^(->} "2";"12"};
{\ar@{^(->} "4";"14"};
{\ar@{^(->} "6";"16"};
{\ar_{\del\ssh} "10";"12"};
{\ar_{x\sas} "12";"14"};
{\ar_{(d_X^1)\sas} "14";"16"};
{\ar@{}|{\circlearrowright} "11";"12"};
{\ar@{}|{\circlearrowright} "2";"14"};
{\ar@{}|{\circlearrowright} "4";"16"};
\endxy
\]
By assumption, the upper row is exact at $\F(C,B)$. Thus there exists some $\nu\in\F(C,A)$ satisfying $x\sas\nu=x\sas\thh$. Then by $x\sas(\thh-\nu)=0$, there exists $f\in\C(C,X^{n+1})$ which gives $\thh-\nu=\del\ssh (f)$. Thus it follows $\thh=\nu+\del\ssh f\in\F(C,A)$.
\end{proof}

\begin{lem}\label{LemClosed}
For any additive subfunctor $\F\se\E$, the following are equivalent.
\begin{enumerate}
\item $\F$ is closed on the right.
\item $\F$ is closed on the left.
\end{enumerate}
Thus in the following, we simply say $\F\se\E$ is {\it closed}, if one of them is satisfied.
\end{lem}
\begin{proof}
We only show $(1)\Rightarrow(2)$. Let $\Xd_C$ be any $\sfr|_{\F}$-distinguished $n$-exangle, and let us show the exactness of $\F(C,A)\ov{(d_X^n)\uas}{\ltc}\F(X^n,A)\ov{(d_X^{n-1})\uas}{\ltc}\F(X^{n-1},A)$ for any $A\in\C$.


Take any element $\thh\in\F(X^n,A)$ satisfying $(d_X^{n-1})\uas\thh=0$. By the exactness of
\[ \E(C,A)\ov{(d_X^n)\uas}{\ltc}\E(X^n,A)\ov{(d_X^{n-1})\uas}{\ltc}\E(X^{n-1},A), \]
there exists $\nu\in\E(C,A)$ satisfying $(d_X^n)\uas\nu=\thh$. It suffices to show $\nu\in\F(C,A)$.
Realize $\thh$ by an $\sfr|_{\F}$-distinguished $n$-exangle $_A\langle Y^{\mr},\thh\rangle_{X^n}$, and $\nu$ by an $\sfr$-distinguished $n$-exangle $_A\langle Z^{\mr},\nu\rangle_C$. Take a good lift $f^{\mr}$ of $(\id_A,d_X^n)\co\thh\to\nu$, to obtain $\sfr$-distinguished $n$-exangle $\langle M_f^{\mr},(d_Y^0)\sas\nu\rangle$ as follows.
\[ Y^1\to Y^2\oplus Z^1\to\cdots\to Y^n\oplus Z^{n-1}\to X^n\oplus Z^n\ov{[d_X^n\ d_Z^n]}{\lra}C\ov{(d_Y^0)\sas\nu}{\dra} \]
By the dual of Proposition \ref{PropShiftedReal} applied to the following diagram,
\[
\xy
(-12,7)*+{X^n}="0";
(12,7)*+{C}="2";
(-12,-7)*+{X^n\oplus Z^n}="4";
(12,-7)*+{C}="6";
{\ar^{d_X^n} "0";"2"};
{\ar_{\left[\bsm1\\0\esm\right]} "0";"4"};
{\ar@{=} "2";"6"};
{\ar_(0.6){[d_X^n\ d_Z^n]} "4";"6"};
{\ar@{}|\circlearrowright "0";"6"};
\endxy
\]
we obtain a morphism of $n$-exangles $g^{\mr}\co\Xd_C\to \langle M_f^{\mr},(d_Y^0)\sas\nu\rangle$
satisfying $g^{n+1}=\id_C$. In particular we have
$(d_Y^0)\sas\nu=(g^0)\sas\del\in\F(Y^1,C)$.
Since $d_Y^0$ is an $\sfr|_{\F}$-inflation, Lemma \ref{LemRightClosed} shows $\nu\in\F(C,A)$.
\end{proof}

\begin{prop}\label{PropClosed}
For any additive subfunctor $\F\se\E$, the following are equivalent.
\begin{enumerate}
\item $\CFs$ is $n$-exangulated.
\item $\sfr|_{\F}$-inflations are closed under composition.
\item $\sfr|_{\F}$-deflations are closed under composition.
\item $\F\se\E$ is closed.
\end{enumerate}
\end{prop}
\begin{proof}
{\rm (1)} holds if and only if both {\rm (2)} and {\rm (3)} hold, by Claim \ref{ClaimPreExTr}. Proposition \ref{PropComposAndClosed} means that {\rm (2)} or {\rm (3)} implies {\rm (4)}. Since {\rm (4)} is self-dual by Lemma \ref{LemClosed}, it remains to show $(4)\Rightarrow(2)$.

Let $_A\Xd_C, {}_B\Yr_F$ be any pair of $\sfr|_{\F}$-distinguished $n$-exanlges with $X^1=B$. Under the assumption of {\rm (4)}, let us show that $d_Y^0\ci d_X^0$ becomes an $\sfr|_{\F}$-inflation.
Since $d_Y^0\ci d_X^0$ is an $\sfr$-inflation, there is an $\sfr$-distinguished $n$-exangle $_A\langle Z^{\mr},\tau\rangle_D$ satisfying $d_Z^0=d_Y^0\ci d_X^0$. Similarly as in the proof of Proposition \ref{PropComposAndClosed}, we obtain by Proposition \ref{PropShiftedReal} a morphism of $n$-exangles $f^{\mr}\co\Xd\to\langle Z^{\mr},\tau\rangle$ which satisfies $f^0=\id_A,f^1=d_Y^0$ and makes $\langle M_f^{\mr},(d_X^0)\sas\tau\rangle$ an $\sfr$-distinguished $n$-exangle as follows.
\[ B\ov{\left[\bsm -d_X^1\\ d_Y^0\esm\right]}{\lra} X^2\oplus Y^1\to\cdots\to C\oplus Z^n\to D\ov{(d_X^0)\sas\tau}{\dra}. \]
Applying Proposition \ref{PropShiftedReal} to the following diagram,
\[
\xy
(-11,6)*+{B}="0";
(11,6)*+{X^2\oplus Y^1}="2";
(-11,-6)*+{B}="4";
(11,-6)*+{Y^1}="6";
{\ar^(0.4){\left[\bsm -d_X^1\\ d_Y^0\esm\right]} "0";"2"};
{\ar@{=} "0";"4"};
{\ar^{[0\ 1]} "2";"6"};
{\ar_{d_Y^0} "4";"6"};
{\ar@{}|\circlearrowright "0";"6"};
\endxy
\]
we obtain a morphism of $n$-exangles $g^{\mr}\co\langle M_f^{\mr},(d_X^0)\sas\tau\rangle\to\langle Y^{\mr},\rho\rangle$ satisfying $g^0=\id_B$ and $g^1=[0\ 1]$. In particular we have $(d_X^0)\sas\tau=(g^{n+1})\uas\rho\in\F(D,B)$. 
Since $d_X^0$ is an $\sfr|_{\F}$-inflation, Lemma \ref{LemRightClosed} shows $\tau\in\F(D,A)$.
\end{proof}

\begin{cor}\label{CorIntersection}
Let $\{\F_{\lam}\}_{\lam\in\Lambda}$ be a family of additive subfunctors of $\E$. If each $\F_{\lam}\se\E$ is closed, then so is their intersection $\un{\lam\in\Lambda}{\bigcap}\F_{\lam}\se\E$.
\end{cor}
\begin{proof}
It can be easily confirmed that the intersection satisfies condition {\rm (2)} in Proposition \ref{PropClosed}.
\end{proof}

\begin{dfn}\label{DefItoF}
Let $\Ical\se\C$ be a full subcategory. Define subfunctors $\E_{\Ical}$ and $\E^{\Ical}$ of $\E$ by
\begin{eqnarray*}
\E_{\Ical}(C,A)&=&\{\del\in\E(C,A)\mid (\del\ssh)_I=0\ \text{for any}\ I\in\Ical\},\\
\E^{\Ical}(C,A)&=&\{\del\in\E(C,A)\mid \del\ush_I=0\ \text{for any}\ I\in\Ical\}.
\end{eqnarray*}
\end{dfn}

\begin{prop}\label{PropItoF}
For any full subcategory $\Ical\se\C$, these $\E_{\Ical}$ and $\E^{\Ical}$ are closed subfunctors of $\E$.
\end{prop}
\begin{proof}
We only show for the statement for $\E_{\Ical}$. To show that $\E_{\Ical}$ is a subfunctor of $\E$, it suffices to show
\[ a\sas c\uas(\E_{\Ical}(C,A))\se\E_{\Ical}(C\ppr,A\ppr) \]
for any $a\in\C(A,A\ppr)$ and $c\in\C(C\ppr,C)$. Let $\del\in\E_{\Ical}(C,A)$ be any element. Then $a\sas c\uas\del\in\E(C\ppr,A\ppr)$ satisfies
\[ (a\sas c\uas\del)\ssh(f\ppr)=f^{\prime\ast}a\sas c\uas\del=a\sas(\del\ssh(c\ci f\ppr))=0 \]
for any $I\in\Ical$ and any $f\ppr\in\C(I,C\ppr)$. Thus $((a\sas c\uas\del)\ssh)_I=0$ holds for any $I\in\Ical$, which means $a\sas c\uas\del\in\E_{\Ical}(C\ppr,A\ppr)$.
Thus $\E_{\Ical}\se\E$ is a subfunctor. Moreover, since
\[ 0\ssh=0 \quad\text{and}\quad (\del-\del\ppr)\ssh=\del\ssh-\del\ppr\ssh \]
holds for $0\in\E(C,A)$ and any $\del,\del\ppr\in\E(C,A)$, we see that $\E_{\Ical}\se\E$ is additive.

For any $\sfr|_{\E_{\Ical}}$-distinguished $n$-exangle $\Xd$, let us show the exactness of
\[ \E_{\Ical}(-,X^0)\ov{(d_X^0)\sas}{\ltc}\E_{\Ical}(-,X^1)\ov{(d_X^1)\sas}{\ltc}\E_{\Ical}(-,X^2). \]
Remark that $\E(-,X^0)\ov{(d_X^0)\sas}{\ltc}\E(-,X^1)\ov{(d_X^1)\sas}{\ltc}\E(-,X^2)$ is exact. Thus for any $C\in\C$, if $\thh\in\E_{\Ical}(C,X^1)$ satisfies $(d_X^1)\sas\thh=0$, then there is $\nu\in\E(C,X^0)$ which gives $(d_X^0)\sas\nu=\thh$. It is enough to show $\nu\ssh(f)=0$ for any $I\in\Ical$ and any $f\in\C(I,C)$. Since $0\to\E(I,X^0)\ov{(d_X^0)\sas}{\lra}\E(I,X^1)$ is exact, this follows from the equation
\[ (d_X^0)\sas(\nu\ssh(f))=(d_X^0)\sas f\uas\nu=f\uas\thh=\thh\ssh(f)=0. \]
\end{proof}

\section{Typical cases}\label{section_Typical}
\subsection{Extriangulated categories}
In this subsection, we consider the case $n=1$.
Let $\C$ be an additive category, and let $\E\co\C\op\times\C\to\Ab$ be a biadditive functor.
\begin{lem}\label{Lem1EAandET}
For any $A,C\in\C$, let $X^{\mr},Y^{\mr}\in\CACt$ be any pair of objects. Assume that 
\begin{eqnarray}
&\C(C,-)\ov{\C(d_X^1,-)}{\Longrightarrow}\C(X^1,-)\ov{\C(d_X^0,-)}{\Longrightarrow}\E(A,-),&\label{ExXAC1}\\
&\C(-,A)\ov{\C(-,d_X^0)}{\Longrightarrow}\C(-,X^1)\ov{\C(-,d_X^1)}{\Longrightarrow}\E(-,C),&\label{ExXAC2}
\end{eqnarray}
are exact, and similarly for $Y^{\mr}$. Then for any morphism $f^{\mr}=(\id_A,f^1,\id_C)\in\CACt(X^{\mr},Y^{\mr})$, the following are equivalent.
\begin{enumerate}
\item $f^{\mr}$ is a homotopy equivalence in $\CACt$.
\item $f^{\mr}$ is an isomorphism in $\CACt$.
\item $f^1$ is an isomorphism in $\C$.
\end{enumerate}
Thus $X^{\mr},Y^{\mr}$ are homotopically equivalent in $\CACt$ if and only if they are equivalent in the sense of \cite[Definition 2.7]{NP}.
\end{lem}
\begin{proof}
$(2)\Leftrightarrow(3)$ is obvious. $(2)\Rightarrow(1)$ is also trivial.
Let us show that {\rm (1)} implies {\rm (3)}.
Suppose that $f^{\mr}$ has a homotopy inverse $g^{\mr}\in\CACt(Y^{\mr},X^{\mr})$. 

Let $g^{\mr}\ci f^{\mr}\un{\vp^{\mr}}{\sim}\id_{X^{\mr}}$ be a homotopy. By Claim~\ref{ClaimHtpyReduced}, we may assume that $\vp^{\mr}$ is of the form $\vp^{\mr}=(0,\vp^2)$. Then we have $\vp^2\ci d_X^1=\id-g^1\ci f^1$. Thus $x=g^1+\vp^2\ci d_Y^1$ satisfies
\[ x\ci f^1=g^1\ci f^1+\vp^2\ci d_X^1=\id, \]
and gives a left inverse of $f^1$ in $\C$. Similarly we can show that $f^1$ has a right inverse, which means it is an isomorphism.
\end{proof}

\begin{lem}\label{Lem1EAOcta}
Assume that $\CEs$ is a $1$-exangulated category, and suppose we are given distinguished $1$-exangles $A\ov{f}{\lra}B\ov{f\ppr}{\lra}D\ov{\del}{\dra}$
and $B\ov{g}{\lra}C\ov{g\ppr}{\lra}F\ov{\del\ppr}{\dra}$.
Remark that $h=g\ci f$ is an $\sfr$-inflation by {\rm (EA1)}, and thus there is also some distinguished $1$-exangle $A\ov{h}{\lra}C\ov{h\ppr}{\lra}E\ov{\del\pprr}{\dra}$.

Then, there exist $d\in\C(D,E)$ and $e\in\C(E,F)$ which satisfy the following conditions.
\begin{itemize}
\item[{\rm (i)}] $D\ov{d}{\lra}E\ov{e}{\lra}F\ov{f\ppr\sas\del\ppr}{\dra}$ is a distinguished $1$-exangle.
\item[{\rm (ii)}] $d\uas\del\pprr=\del$.
\item[{\rm (iii)}] $f\sas\del\pprr=e\uas\del\ppr$.
\end{itemize}
\end{lem}
\begin{proof}
This is an analog of \cite[3.5]{Hu}. By Proposition~\ref{PropShiftedReal} {\rm (2)}, there is $d\in\C(D,E)$ which satisfies $d\ci f\ppr=h\ppr\ci g$, $d\uas\del\pprr=\del\ppr$
and makes
\[ B\ov{u}{\lra}D\oplus C\ov{[d\ h\ppr]}{\lra}E\ov{f\sas\del\pprr}{\dra} \]
a distinguished $1$-exangle for $u=\begin{bmatrix}-f\ppr\\ g\end{bmatrix}$. Again by the same proposition applied to the following,
\[
\xy
(-24,6)*+{B}="0";
(-7,6)*+{D\oplus C}="2";
(8,6)*+{E}="4";
(20,6)*+{}="6";
(-24,-6)*+{B}="10";
(-7,-6)*+{C}="12";
(8,-6)*+{F}="14";
(20,-6)*+{}="16";
{\ar^(0.4){u} "0";"2"};
{\ar^(0.62){[d\ h\ppr]} "2";"4"};
{\ar@{-->}^{f\sas\del\pprr} "4";"6"};
{\ar@{=} "0";"10"};
{\ar^{[0\ 1]} "2";"12"};
{\ar_{g} "10";"12"};
{\ar_{g\ppr} "12";"14"};
{\ar@{-->}_{\del\ppr} "14";"16"};
{\ar@{}|\circlearrowright "0";"12"};
\endxy
\]
we obtain $e\in\C(E,F)$ which satisfies $e\ci [d\ h\ppr]=g\ppr\ci [0\ 1]$,
$e\uas\del\ppr=f\sas\del\pprr$
and makes
\begin{equation}\label{DCEC}
D\oplus C\ov{v}{\lra}E\oplus C\ov{[e\ g\ppr]}{\lra}F\ov{u\sas\del\ppr}{\dra}
\end{equation}
a distinguished $1$-exangle for $v=\left[\begin{array}{cc}-d&-h\ppr\\0&1\end{array}\right]$.
Thus Corollary \ref{CorSummandOut} shows that
\[ D\ov{-d}{\lra}E\ov{e}{\lra}F\ov{-f\ppr\sas\del\ppr}{\dra} \]
is a distinguished $1$-exangle. This is isomorphic to $D\ov{d}{\lra}E\ov{e}{\lra}F\ov{f\ppr\sas\del\ppr}{\dra}$, and thus Corollary \ref{CorDEAInv} can be applied.
\end{proof}

\begin{prop}\label{Prop1EAandET}
Let $\C$ and $\E$ be as before. Then, a triplet $\CEs$ is a $1$-exangulated category if and only if it is an extriangulated category.
\end{prop}
\begin{proof}
First, suppose $\CEs$ is a $1$-exangulated category, and show it is an extriangulated category. By duality, let us just confirm conditions {\rm (ET1),(ET2),(ET3),(ET4)} in \cite[Definition 2.12]{NP}.

{\rm (ET1)} is already assumed. By Lemma~\ref{Lem1EAandET}, the homotopy equivalence class $\sfr(\del)=[X^{\mr}]$ is equal to the equivalence class of $X^{\mr}$ in the sense of \cite[Definition 2.7]{NP} for any extension $\del$. Thus {\rm (ET2)} follows from {\rm (R0),(R2)} and Proposition~\ref{PropSummandOut}. {\rm (ET3)} is shown in Proposition~\ref{PropShiftedReal}. Lemma~\ref{Lem1EAOcta} shows {\rm (ET4)}.

Conversely, suppose $\CEs$ is an extriangulated category. By \cite[Proposition 3.3]{NP}, sequences $(\ref{ExXAC1})$ and $(\ref{ExXAC2})$ are exact for any $X^{\mr}$ realizing an extension $\del$. Thus the equivalence class of $X^{\mr}$ in the sense of \cite{NP} is equal to the homotopy equivalence class of $X^{\mr}$ in $\CACt$, by Lemma~\ref{Lem1EAOcta}.  Similarly as above, let us just confirm conditions {\rm (R0),(R1),(R2)} and {\rm (EA1),(EA2)}.

{\rm (R0),(R2)} follow from {\rm (ET2)}. {\rm (R1)} is shown in \cite[Proposition 3.3]{NP}. {\rm (EA1)} follows from {\rm (ET4)}, as stated in \cite[Remark 2.16]{NP}. {\rm (EA2)} follows from the dual of \cite[Proposition 1.20]{LN}.
\end{proof}

\subsection{$(n+2)$-angulated categories}

In this subsection, we consider the case where the additive category $\C$ is equipped with an automorphism $\Sig\co\C\ov{\cong}{\lra}\C$. Then $\Sig$ gives a biadditive functor $\E_{\Sig}=\C(-,\Sig-)\co\C\op\ti\C\to\Ab$, defined by the following.
\begin{itemize}
\item[{\rm (i)}] For any $A,C\in\C$, $\E_{\Sig}(C,A)=\C(C,\Sig A)$.
\item[{\rm (ii)}] For any $a\in\C(A,A\ppr)$ and $c\in\C(C\ppr,C)$, the map $\E_{\Sig}(c,a)\co\C(C,\Sig A)\to\C(C\ppr,\Sig A\ppr)$ sends $\del\in\C(C,\Sig A)$ to $c\uas a\sas\del=(\Sig a)\ci\del\ci c$.
\end{itemize}
The aim of this subsection is to show the equivalence of the following {\rm (I)} and {\rm (II)}.
\begin{itemize}
\item[{\rm (I)}] To give a class of $(n+2)$-$\Sigma$-sequences $\square$ which makes $(\C,\Sig,\square)$ an $(n+2)$-angulated category in the sense of \cite{GKO}.
\item[{\rm (II)}]  To give an exact realization $\sfr$ of $\E_{\Sig}$ which makes $(\C,\E_{\Sig},\sfr)$ an $n$-exangulated category.
\end{itemize}

First let us show that {\rm (I)} implies {\rm (II)}. 
Let $(\C,\Sig,\square)$ be an $(n+2)$-angulated category as in \cite{GKO}. We assume that $\Sig$ is an automorphism as above.
For each $\del\in\E_{\Sig}(C,A)$, complete it into an $(n+2)$-angle
\[ A\ov{d_X^0}{\lra}X^1\ov{d_X^1}{\lra}X^2\ov{d_X^2}{\lra}\cdots\ov{d_X^{n-1}}{\lra}X^n\ov{d_X^n}{\lra}C\ov{\del}{\lra}\Sig A \]
by {\rm (F1)\,(c)} and {\rm (F2)} in \cite{GKO}. Then define $\sfr_{\square}(\del)=[X^{\mr}]$ by using $X^{\mr}\in\CAC$ given by
\[ X^0\ov{d_X^0}{\lra}X^1\ov{d_X^1}{\lra}X^2\ov{d_X^2}{\lra}\cdots\ov{d_X^{n-1}}{\lra}X^n\ov{d_X^n}{\lra}X^{n+1}\quad(X^0=A,\, X^{n+1}=C). \]

\begin{lem}\label{LemWellDefReal}
For each ${}_A\del_C$, the above $\sfr_{\square}(\del)=[X^{\mr}]$ is well-defined.
\end{lem}
\begin{proof}
Let $A\ov{d_Y^0}{\lra}Y^1\ov{d_Y^1}{\lra}Y^2\ov{d_Y^2}{\lra}\cdots\ov{d_Y^{n-1}}{\lra}Y^n\ov{d_Y^n}{\lra}C\ov{\del}{\lra}\Sig A$ be another choice of $(n+2)$-angle, and let
$Y^{\mr}$ be the corresponding object in $\CAC$ given by
\[ Y^0\ov{d_Y^0}{\lra}Y^1\ov{d_Y^1}{\lra}Y^2\ov{d_Y^2}{\lra}\cdots\ov{d_Y^{n-1}}{\lra}Y^n\ov{d_Y^n}{\lra}Y^{n+1}\quad(Y^0=A,\, Y^{n+1}=C). \]
Let us show $[X^{\mr}]=[Y^{\mr}]$. By {\rm (F2),(F3)} in \cite{GKO}, there is a morphism $f^{\mr}\in\CAC(X^{\mr},Y^{\mr})$ as follows.
\[
\xy
(-34,6)*+{A}="-2";
(-20,6)*+{X^1}="0";
(-6,6)*+{X^2}="2";
(7,6)*+{\cdots}="4";
(20,6)*+{X^n}="6";
(34,6)*+{C}="8";
(48,6)*+{}="9";
(-34,-6)*+{A}="-12";
(-20,-6)*+{Y^1}="10";
(-6,-6)*+{Y^2}="12";
(7,-6)*+{\cdots}="14";
(20,-6)*+{Y^n}="16";
(34,-6)*+{C}="18";
(48,-6)*+{}="19";
{\ar^{d_X^0} "-2";"0"};
{\ar^{d_X^1} "0";"2"};
{\ar^{d_X^2} "2";"4"};
{\ar^{d_X^{n-1}} "4";"6"};
{\ar^{d_X^n} "6";"8"};
%
{\ar_{d_Y^0} "-12";"10"};
{\ar_{d_Y^1} "10";"12"};
{\ar_{d_Y^2} "12";"14"};
{\ar_{d_Y^{n-1}} "14";"16"};
{\ar_{d_Y^n} "16";"18"};
%
{\ar@{=} "-2";"-12"};
{\ar_{f^1} "0";"10"};
{\ar_{f^2} "2";"12"};
{\ar_{f^n} "6";"16"};
{\ar@{=} "8";"18"};
{\ar@{}|\circlearrowright "-2";"10"};
{\ar@{}|\circlearrowright "0";"12"};
{\ar@{}|\circlearrowright "2";"14"};
{\ar@{}|\circlearrowright "6";"18"};
\endxy
\]
Similarly, there is $g^{\mr}\in\CAC(Y^{\mr},X^{\mr})$.
Remark that by \cite[Proposition 2.5]{GKO} and its dual, sequences $(\ref{exnat1})$ and $(\ref{exnat2})$ are exact for $X^{\mr}$ and $Y^{\mr}$, which means $\Xd$ and $\Yd$ are $n$-exangles. Thus $f^{\mr}$ is an homotopy equivalence by Proposition~\ref{PropCAC1}.
\end{proof}

\begin{prop}\label{PropAtoEA}
With the above definition, $(\C,\E_{\Sig},\sfr_{\square})$ becomes an $n$-exangulated category.
\end{prop}
\begin{proof}
Let us confirm the conditions.
{\rm (R0)} follows from {\rm (F2)} and {\rm (F3)}.
{\rm (R1)} follows from \cite[Proposition 2.5]{GKO} and its dual.
{\rm (R2)} follows from {\rm (F1)(b)} and {\rm (F2)}.
{\rm (EA1)} becomes trivial, since any morphism is both inflation and deflation by {\rm (F2)}.

Let us show {\rm (EA2)}. The following argument has been given in (the dual of) \cite[Lemma 4.1]{BT}.
Let $c\in\C(C,D)$ be a morphism, and let $_A\langle X^{\mr},\del=c\uas\rho\rangle_C,{}_A\langle Y^{\mr},\rho\rangle_D$ be $n$-exangles. By the definition of $\sfr_{\square}$ and Remark~\ref{RemEA2} {\rm (2)}, we may assume that they correspond to $(n+2)$-angles
\begin{eqnarray*}
&A\ov{d_X^0}{\lra}X^1\ov{d_X^1}{\lra}\cdots\ov{d_X^{n-1}}{\lra}X^n\ov{d_X^n}{\lra}C\ov{\del}{\lra}\Sig A,&\\
&A\ov{d_Y^0}{\lra}Y^1\ov{d_Y^1}{\lra}\cdots\ov{d_Y^{n-1}}{\lra}Y^n\ov{d_Y^n}{\lra}D\ov{\rho}{\lra}\Sig A.&
\end{eqnarray*}
By {\rm (F2)}, we can \lq rotate' them to obtain $(n+2)$-angles
\begin{eqnarray*}
&\Sig\iv C\ov{(-1)^n\Sig\iv\del}{\lra}A\ov{d_X^0}{\lra}X^1\ov{d_X^1}{\lra}\cdots\ov{d_X^{n-1}}{\lra}X^n\ov{d_X^n}{\lra}C,&\\
&\Sig\iv D\ov{(-1)^n\Sig\iv\rho}{\lra}A\ov{d_Y^0}{\lra}Y^1\ov{d_Y^1}{\lra}\cdots\ov{d_Y^{n-1}}{\lra}Y^n\ov{d_Y^n}{\lra}D.&
\end{eqnarray*}
By {\rm (F4)}, we obtain a morphism of $(n+2)$-$\Sig$-sequences
\[
\xy
(-38,6)*+{\Sig\iv C}="-2";
(-20,6)*+{A}="0";
(-6,6)*+{X^1}="2";
(7,6)*+{\cdots}="4";
(20,6)*+{X^n}="6";
(34,6)*+{C}="8";
(48,6)*+{}="9";
(-38,-6)*+{\Sig\iv D}="-12";
(-20,-6)*+{A}="10";
(-6,-6)*+{Y^1}="12";
(7,-6)*+{\cdots}="14";
(20,-6)*+{Y^n}="16";
(34,-6)*+{D}="18";
(48,-6)*+{}="19";
{\ar^(0.6){(-1)^n\Sig\iv\del} "-2";"0"};
{\ar^{d_X^0} "0";"2"};
{\ar^{d_X^1} "2";"4"};
{\ar^{d_X^{n-1}} "4";"6"};
{\ar^{d_X^n} "6";"8"};
%
{\ar_(0.6){(-1)^n\Sig\iv\rho} "-12";"10"};
{\ar_{d_Y^0} "10";"12"};
{\ar_{d_Y^1} "12";"14"};
{\ar_{d_Y^{n-1}} "14";"16"};
{\ar_{d_Y^n} "16";"18"};
%
{\ar_{\Sig\iv c} "-2";"-12"};
{\ar@{=} "0";"10"};
{\ar_{f^1} "2";"12"};
{\ar_{f^n} "6";"16"};
{\ar^{c} "8";"18"};
{\ar@{}|\circlearrowright "-2";"10"};
{\ar@{}|\circlearrowright "0";"12"};
{\ar@{}|\circlearrowright "2";"14"};
{\ar@{}|\circlearrowright "6";"18"};
\endxy
\]
which gives an $(n+2)$-angle
\begin{equation}\label{ADEA1}
A\oplus\Sig\iv D\ov{d^0}{\lra}X^1\oplus A\ov{d^1}{\lra}
X^2\oplus Y^1\ov{d^2}{\lra}\cdots%
\ov{d^{n-1}}{\lra}X^n\oplus Y^{n-1}\ov{d^n}{\lra}C\oplus Y^n%
\ov{d^{n+1}}{\lra}\Sig A\oplus D
\end{equation}
where
\begin{eqnarray*}
&d^0=\left[\begin{array}{cc}-d_X^0&0\\1&(-1)^n\Sig\iv\rho\end{array}\right],\quad%
d^i=\left[\begin{array}{cc}-d_X^i&0\\ f^i&d_Y^{i-1}\end{array}\right]\quad(1\le i\le n),&\\
&d^{n+1}=\left[\begin{array}{cc}(-1)^{n+1}\del&0\\ c&d_Y^n\end{array}\right].&
\end{eqnarray*}
Then the sequence of isomorphisms in $\C$
\[ \Big(\left[\begin{array}{cc}1&(-1)^n\Sig\iv\rho\\0&1\end{array}\right],\left[\begin{array}{cc}0&1\\1&d_X^0\end{array}\right],\id,\id,\ldots,\id,\left[\begin{array}{cc}1&(-1)^n\rho\\0&1\end{array}\right] \Big) \]
gives an isomorphism of $(n+2)$-sequences from $(\ref{ADEA1})$ to
\begin{equation}\label{ADEA2}
A\oplus\Sig\iv D\ov{e^0}{\lra}A\oplus X^1\ov{e^1}{\lra}
X^2\oplus Y^1\ov{d^2}{\lra}\cdots%
\ov{d^{n-1}}{\lra}X^n\oplus Y^{n-1}\ov{d^n}{\lra}C\oplus Y^n%
\ov{e^{n+1}}{\lra}\Sig A\oplus D,
\end{equation}
with
\[ e^0=\left[\begin{array}{cc}1&0\\0&(-1)^nd_X^0\ci\Sig\iv\rho\end{array}\right],\quad%
e^1=\left[\begin{array}{cc}0&-d_X^1\\0&f^1\end{array}\right],\quad%
e^{n+1}=\left[\begin{array}{cc}0&0\\ c&d_Y^n\end{array}\right]
 \]
and the same $d^2,\ldots,d^n$.
Thus $(\ref{ADEA2})$ belongs to $\square$. Since this is equal to the direct sum of
\begin{eqnarray}
&A\ov{\id_A}{\lra}A\to0\to\cdots\to0\to\Sig A\quad\text{and}&\nonumber\\
&\Sig\iv D\ov{q^0}{\lra}X^1\ov{q^1}{\lra}
X^2\oplus Y^1\ov{d^2}{\lra}\cdots%
\ov{d^{n-1}}{\lra}X^n\oplus Y^{n-1}\ov{d^n}{\lra}C\oplus Y^n%
\ov{q^{n+1}}{\lra}D\label{ADEA3}&
\end{eqnarray}
with
\[ q^0=(-1)^nd_X^0\ci\Sig\iv\rho,\quad q^1=\begin{bmatrix}-d_X^1\\ f^1\end{bmatrix},\quad q^{n+1}=[c\ d_Y^n], \]
we see that $(\ref{ADEA3})$ also belongs to $\square$ by {\rm (F1)(a)}. Rotating it by {\rm (F2)}, we obtain an $(n+2)$-angle
\[ X^1\ov{q^1}{\lra}X^2\oplus Y^1\ov{d^2}{\lra}\cdots%
\ov{d^{n-1}}{\lra}X^n\oplus Y^{n-1}\ov{d^n}{\lra}C\oplus Y^n%
\ov{q^{n+1}}{\lra}D\ov{(\Sig d_X^0)\ci\rho}{\lra}\Sig X^1. \]
By the definition of $\sfr_{\square}$, this shows that $f^{\mr}=(\id_A,f^1,\ldots,f^n,c)\co\Xd\to\Yr$ gives a distinguished $n$-exangle
\[ X^1\ov{q^1}{\lra}X^2\oplus Y^1\ov{d^2}{\lra}\cdots%
\ov{d^{n-1}}{\lra}X^n\oplus Y^{n-1}\ov{d^n}{\lra}C\oplus Y^n%
\ov{q^{n+1}}{\lra}D\ov{(d_X^0)\sas\rho}{\dra}, \]
that is what we wanted to show.
\end{proof}

\bigskip

Conversely, let us show that {\rm (II)} implies {\rm (I)}. Suppose we are given an exact realization of $\E_{\Sig}$ which makes $(\C,\E_{\Sig},\sfr)$ an $n$-exangulated category. Remark that any object in $\AE$
\[ X^0\ov{d_X^0}{\lra}X^1\ov{d_X^1}{\lra}\cdots \ov{d_X^n}{\lra}X^{n+1}\ov{\del}{\dra} \]
can be naturally regarded as an {\it $(n+2)$-$\Sigma$-sequence}
\begin{equation}\label{XSS}
X^0\ov{d_X^0}{\lra}X^1\ov{d_X^1}{\lra}\cdots \ov{d_X^n}{\lra}X^{n+1}\ov{\del}{\lra}\Sig X^0
\end{equation}
in the sense of \cite[Definition 2.1]{GKO}.
\begin{rem}\label{RemAEtoSS}
The above correspondence gives a fully faithful functor from $\AE$ to the category of $(n+2)$-$\Sig$-sequences. In this way, we may identify $\AE$ with the full subcategory of the category of $(n+2)$-$\Sig$-sequences, consisting of $(\ref{XSS})$ satisfying $d_X^{i+1}\ci d_X^i=0\ (0\le i\le n-1)$, $\del\ci d_X^n=0$ and $(\Sig d_X^0)\ci\del=0$. This subcategory is closed by isomorphisms, and by taking finite direct sums and summands.
\end{rem}

\begin{lem}\label{LemEAtoA}
For any $A\in\C$, let us denote $\id_{\Sig A}\in\C(\Sig A,\Sig A)$ by $\iota={}_A\iota_{\Sig A}\in\E_{\Sig}(\Sig A,A)$ when we regard it as an extension. For this extension, we have $\sfr(\iota)=[\O^{\mr}]$. Namely,
\[ A\ov{0}{\lra}0\ov{0}{\lra}\cdots\ov{0}{\lra}0\ov{0}{\lra}\Sig A\ov{\iota}{\dra} \]
is a distinguished $n$-exangle.
\end{lem}
\begin{proof}
This immediately follows from Proposition~\ref{PropZeroEA} applied to $\del=\iota$.
\end{proof}

\begin{prop}\label{PropEAtoA}
Define $\square_{\sfr}$ to be the class of $(n+2)$-$\Sig$-sequences obtained as $(\ref{XSS})$ from distinguished $n$-exangles. Then $(\C,\E_{\Sig},\square_{\sfr})$ becomes an $(n+2)$-angulated category.
\end{prop}
\begin{proof}
By Corollary \ref{CorDEAInv} and Remark~\ref{RemAEtoSS}, the class $\square_{\sfr}$ is closed by isomorphisms of $(n+2)$-$\Sig$-sequences. Thus we do not have to take any isomorphism closure. Let us confirm conditions {\rm (F1)},\ldots,{\rm (F4)} in \cite{GKO}.
{\rm (F1)(a)} follows from Proposition~\ref{PropSummandOut} and Remark~\ref{RemAEtoSS}. {\rm (F1)(b)} follows from {\rm (R2)}.

{\rm (F2)} Let $_A\Xd_C$ be any distinguished $n$-exangle. It suffices to show that we can rotate $\Xd$ in both directions to obtain distinguished $n$-exangles. As in Lemma~\ref{LemEAtoA}, the pair $_A\langle\O^{\mr},\iota\rangle_{\Sig A}$ is also a distinguished $n$-exangle. By {\rm (EA2)}, the morphism $(\id_A,\del)\co\del\to\iota$ has a good lift $f^{\mr}\co\Xd\to\langle\O^{\mr},\iota\rangle$, which should be as follows without any other possibility.
\[
\xy
(-43,12)*+{A}="0";
(-28,12)*+{X^1}="1";
(-13,12)*+{X^2}="2";
(0,12)*+{\cdots}="3";
(14,12)*+{X^n}="4";
(29,12)*+{C}="5";
(41,12)*+{}="6";
(-43,0)*+{A}="10";
(-28,0)*+{0}="11";
(-13,0)*+{0}="12";
(0,0)*+{\cdots}="13";
(4,0)*+{}="13.5";
(14,0)*+{0}="14";
(29,0)*+{\Sig A}="15";
(41,0)*+{}="16";
{\ar^{d_X^0} "0";"1"};
{\ar^{d_X^1} "1";"2"};
{\ar^{d_X^2} "2";"3"};
{\ar^{d_X^{n-1}} "3";"4"};
{\ar^{d_X^n} "4";"5"};
{\ar@{-->}^{\del} "5";"6"};
{\ar@{=} "0";"10"};
{\ar^{0} "1";"11"};
{\ar^{0} "2";"12"};
{\ar^{0} "4";"14"};
{\ar^{\del} "5";"15"};
{\ar_{0} "10";"11"};
{\ar_{0} "11";"12"};
{\ar_{0} "12";"13"};
{\ar_{0} "13";"14"};
{\ar_{0} "14";"15"};
{\ar@{-->}_{\iota} "15";"16"};
{\ar@{}|\circlearrowright "0";"11"};
{\ar@{}|\circlearrowright "1";"12"};
{\ar@{}|\circlearrowright "2";"13.5"};
{\ar@{}|\circlearrowright "4";"15"};
\endxy
\]
Its mapping cone induces a distinguished $n$-exangle
\begin{equation}\label{F2Dist1}
X^1\ov{-d_X^1}{\lra}X^2\ov{-d_X^2}{\lra}\cdots\to X^{n-1}\ov{-d_X^{n-1}}{\lra}X^n\ov{-d_X^n}{\lra}C\ov{\del}{\lra}\Sig A\ov{(d_X^0)\sas\iota}{\dra}.
\end{equation}
Remark that we have $(d_X^0)\sas\iota=(\Sig d_X^0)\ci\id_{\Sig A}=\Sig d_X^0$ by definition. Since
\[ \big((-1)^n,(-1)^{n-1},\ldots,1,-1,\id_C,\id_{\Sig A}\big) \]
gives an isomorphism from $(\ref{F2Dist1})$ to
\begin{equation}\label{F2Dist2}
X^1\ov{d_X^1}{\lra}X^2\ov{d_X^2}{\lra}\cdots\to X^{n-1}\ov{d_X^{n-1}}{\lra}X^n\ov{d_X^n}{\lra}C\ov{\del}{\lra}\Sig A\ov{(-1)^n\Sig d_X^0}{\dra},
\end{equation}
this $(\ref{F2Dist2})$ becomes a distinguished $n$-exangle by Corollary \ref{CorDEAInv}. Rotation to the opposite direction can be performed in a dual manner.

{\rm (F1)(c)} Any morphism $f\in\C(A,B)$ can be regarded as an extension $f\in\E_{\Sig}(A,\Sig\iv B)$, and then there exists some distinguished $n$-exangle
\[ \Sig\iv B\ov{d_X^0}{\lra}X^1\ov{d_X^1}{\lra}\cdots\to X^n\ov{d_X^n}{\lra}A\ov{f}{\dra}. \]
Applying {\rm (F2)} repeatedly, we obtain a distinguished $n$-exangle of the form
\[ A\ov{f}{\lra}B\to\Sig X^1\to \cdots\to \Sig X^n\dra. \]

{\rm (F3)} This follows from {\rm (F2)} and {\rm (R0)}, or from Proposition~\ref{PropShiftedReal}.

{\rm (F4)} The same argument on the axioms of triangulated category \cite{Hu} works. Suppose that we are given distinguished $n$-exangles $\Xd,\Yr$ and a commutative square
\begin{equation}\label{F4Sq}
\xy
(-6,6)*+{X^0}="0";
(6,6)*+{X^1}="2";
(-6,-6)*+{Y^0}="4";
(6,-6)*+{Y^1}="6";
{\ar^{d_X^0} "0";"2"};
{\ar_{f^0} "0";"4"};
{\ar^{f^1} "2";"6"};
{\ar_{d_Y^0} "4";"6"};
{\ar@{}|\circlearrowright "0";"6"};
\endxy
\end{equation}
in $\C$. 
Let us construct a morphism $f^{\mr}=(f^0,f^1,f^2,\ldots,f^{n+1})\co\Xd\to\Yr$ to fulfill the requirement of {\rm (F4)}.
Remark that
\begin{eqnarray*}
&Y^0\ov{\id_{Y^1}}{\lra}Y^0\to0\to\cdots\to0\ov{0}{\dra},&\\
&X^0\to0\to\cdots\to0\to\Sig X^0\ov{\iota}{\dra}&
\end{eqnarray*}
are distinguished $n$-exangles by {\rm (R2)} and Lemma~\ref{LemEAtoA}. Taking coproducts with $\Xd$ and $\Yr$, we obtain distinguished $n$-exangles
\begin{eqnarray}
&X^0\oplus Y^0\ov{d_X^0\oplus \id}{\lra}X^1\oplus Y^0\ov{[d_X^1\ 0]}{\lra}X^2\ov{d_X^2}{\lra}\cdots\ov{d_X^{n-1}}{\lra}X^n\ov{d_X^n}{\lra}X^{n+1}\ov{\mu}{\dra},&\label{XAdded}\\
&X^0\oplus Y^0\ov{[0\ d_Y^0]}{\lra}Y^1\ov{d_Y^1}{\lra}Y^2\ov{d_Y^2}{\lra}\cdots\ov{d_Y^{n-1}}{\lra}Y^n%
\ov{\Big[\raise1ex\hbox{\leavevmode\vtop{\baselineskip-8ex \lineskip1ex \ialign{#\crcr{$\scriptstyle{\ 0}$}\crcr{$\scriptstyle{d_Y^n}$}\crcr}}}\Big]}{\lra}%
\Sig X^0\oplus Y^{n+1}\ov{\nu}{\dra}&\label{YAdded}
\end{eqnarray}
by Proposition~\ref{PropSummandOut}, where $\mu\in\E_{\Sig}(X^{n+1},X^0\oplus Y^0)$ and $\nu\in\E_{\Sig}(\Sig X^0\oplus Y^{n+1},X^0\oplus Y^0)$ correspond to
\[ \left[\begin{array}{c}\del\\0\end{array}\right]\in\C(X^{n+1},\Sig X^0\oplus\Sig Y^0), \]
\[ \left[\begin{array}{cc}\id_{\Sig X^0}&0\\0&\rho\end{array}\right]\in\C(\Sig X^0\oplus Y^{n+1},\Sig X^0\oplus\Sig Y^0) \]
respectively, through the natural isomorphism
\begin{equation}\label{TrivIso}
\Sig (X^0\oplus Y^0)\cong \Sig X^0\oplus\Sig Y^0.
\end{equation}
Using the automorphism
\[ a=\left[\begin{array}{cc}-\id&0\\ f^0&\id\end{array}\right]\co X^0\oplus Y^0\ov{\cong}{\lra}X^0\oplus Y^0, \]
we can modify $(\ref{YAdded})$ to obtain a distinguished $n$-exangle
\begin{equation}\label{YModified}
X^0\oplus Y^0\ov{[d_Y^0\ci f^0\ d_Y^0]}{\lra}Y^1\ov{d_Y^1}{\lra}Y^2\ov{d_Y^2}{\lra}\cdots\ov{d_Y^{n-1}}{\lra}Y^n\ov{\Big[\raise1ex\hbox{\leavevmode\vtop{\baselineskip-8ex \lineskip1ex \ialign{#\crcr{$\scriptstyle{\ 0}$}\crcr{$\scriptstyle{d_Y^n}$}\crcr}}}\Big]}{\lra}\Sig X^0\oplus Y^{n+1}\ov{a\sas\nu}{\dra}
\end{equation}
by Corollary \ref{CorDEAInv} {\rm (1)}, where $a\sas\nu$ corresponds to
\[ \left[\begin{array}{cc}-1&0\\ \Sig f^0&\rho\end{array}\right]\in\C(\Sig X^0\oplus Y^{n+1},\Sig X^0\oplus \Sig Y^0) \]
through $(\ref{TrivIso})$.
By Proposition~\ref{PropShiftedReal} {\rm (2)} applied to the following commutative square,
\[
\xy
(-12,6)*+{X^0\oplus Y^0}="0";
(12,6)*+{X^1\oplus Y^0}="2";
(-12,-6)*+{X^0\oplus Y^0}="4";
(12,-6)*+{Y^1}="6";
{\ar^{d_X^0\oplus\id} "0";"2"};
{\ar@{=} "0";"4"};
{\ar^{[f^1\ d_Y^0]} "2";"6"};
{\ar_(0.58){[d_Y^0\ci f^0\ d_Y^0]} "4";"6"};
{\ar@{}|\circlearrowright "0";"6"};
\endxy
\]
we obtain a morphism
\[ g^{\mr}=\Big(\id,[f^1\ d_Y^0],f^2,\ldots,f^n,\begin{bmatrix}x\\ y\end{bmatrix}\Big) \]
from $(\ref{XAdded})$ to $(\ref{YModified})$, which makes $\langle M_g^{\mr},(d_X^0\oplus\id)\sas a\sas\nu\rangle$ a distinguished $n$-exangle. If we put $f^{n+1}=y$, then the equalities
\[ \begin{bmatrix}x\\ y\end{bmatrix}\ci d_X^n=\begin{bmatrix}0\\ d_Y^n\end{bmatrix}\ci f^n\quad\text{and}\quad \begin{bmatrix}x\\ y\end{bmatrix}\uas\!\! a\sas\nu=\mu \]
imply $x=-\del$ and $f^{n+1}\ci d_X^n=d_Y^n\ci f^n$, and
\[ (\Sig f^0)\ci\del=\rho\ci f^{n+1}, \]
which means $(f^0)\sas\del=(f^{n+1})\uas\rho$. Thus $f^{\mr}=(f^0,f^1,f^2,\ldots,f^{n+1})\co\Xd\to\Yr$ is a morphism. Moreover, the obtained distinguished $n$-exangle $\langle M_g^{\mr},(d_X^0\oplus\id)\sas a\sas\nu\rangle$ is of the form
\[
\xy
(-56,0)*+{X^1\oplus Y^0}="0";
(-33,0)*+{X^2\oplus Y^1}="2";
(-14,0)*+{\cdots}="4";
(6,0)*+{X^{n+1}\oplus Y^n}="6";
(36,0)*+{\Sig X^0\oplus Y^{n+1}}="8";
(56,0)*+{}="10";
{\ar^{d_{M_g}^0} "0";"2"};
{\ar^(0.6){d_{M_g}^1} "2";"4"};
{\ar^(0.38){d_{M_g}^{n-1}} "4";"6"};
{\ar^(0.48){d_{M_g}^n} "6";"8"};
{\ar@{-->}^(0.7){\tau} "8";"10"};
\endxy
\]
where
\[ d_{M_g}^i=\left[\begin{array}{cc}-d_X^{i+1}&0\\ f^{i+1}&d_Y^i\end{array}\right]\ \ (0\le i\le n-1),\quad d_{M_g}^n=\left[\begin{array}{cc}-\del&0\\ f^{n+1}&d_Y^n\end{array}\right], \]
and $\tau=(d_X^0\oplus\id)\sas a\sas\nu$ corresponds to
\[ \left[\begin{array}{cc}-\Sig d_X^0&0\\ \Sig f^0&\rho\end{array}\right]\in\C(\Sig X^0\oplus Y^{n+1},\Sig X^1\oplus\Sig Y^0) \]
through the isomorphism $\Sig(X^1\oplus Y^0)\cong\Sig X^1\oplus\Sig Y^0$.
This yields the desired $(n+2)$-angle.
\end{proof}

\subsection{$n$-exact categories}\label{subsection_n-exact}

In this subsection, we will see the relation with the notion of an {\it $n$-exact category} introduced in \cite{J}. We briefly recall its definition and related notions from \cite{J}.
\begin{dfn}\label{DefNExSeq}(cf. \cite[Definitions 2.2 and 2.4]{J})
Let $\C$ be an additive category, and let $X^{\mr}\in\CC$ be any object.
\begin{enumerate}
\item $X^{\mr}$ is called an {\it $n$-kernel sequence} if the following sequence of functors $\C\op\to\Ab$ is exact.
\[ 0\,\tc\,\C(-,X^0)\ov{\C(-,d_X^0)}{\ltc}\C(-,X^1)\ov{\C(-,d_X^1)}{\ltc}\cdots\ov{\C(-,d_X^n)}{\ltc}\C(-,X^{n+1}) \]
In particular $d_X^0$ is a monomorphism in $\C$.
\item $X^{\mr}$ is called an {\it $n$-cokernel sequence} if the following sequence of functors $\C\to\Ab$ is exact.
\[ 0\,\tc\,\C(X^{n+1},-)\ov{\C(d_X^n,-)}{\ltc}\C(X^n,-)\ov{\C(d_X^{n-1},-)}{\ltc}\cdots\ov{\C(d_X^0,-)}{\ltc}\C(X^0,-) \]
In particular $d_X^n$ is an epimorphism in $\C$.
\item $X^{\mr}$ is called an {\it $n$-exact sequence} if it is both $n$-kernel and $n$-cokernel sequence.
\end{enumerate}
Remark that $n$-kernel (respectively, $n$-cokernel, or $n$-exact) sequences are closed by homotopic equivalences in $\CC$.
\end{dfn}

The following can be shown easily.
\begin{prop}\label{PropNExSeq}
Let $\C$ be an additive category, and let $X^{\mr},Y^{\mr}\in\CC$ be any pair of $n$-exact sequences.
\begin{enumerate}
\item Let $k\in\{0,\ldots,n\}$ be any integer. For any commutative square
\[
\xy
(-7,6)*+{X^k}="0";
(7,6)*+{X^{k+1}}="2";
(-7,-6)*+{Y^k}="4";
(7,-6)*+{Y^{k+1}}="6";
{\ar^(0.46){d_X^k} "0";"2"};
{\ar_{a} "0";"4"};
{\ar^{b} "2";"6"};
{\ar_(0.46){d_Y^k} "4";"6"};
{\ar@{}|\circlearrowright "0";"6"};
\endxy
\]
in $\C$, there exists $f^{\mr}\in\CC(X^{\mr},Y^{\mr})$ satisfying $f^k=a$ and $f^{k+1}=b$. Moreover, such $f^{\mr}$ is unique up to homotopy. Especially, if both $a,b$ are isomorphisms, then $f^{\mr}$ becomes a homotopy equivalence in $\CC$.
\item Let $a\in\C(X^0,Y^0), c\in\C(X^{n+1},Y^{n+1})$ be any pair of morphisms. If there exists $f^{\mr}\in\CC(X^{\mr},Y^{\mr})$ satisfying $f^0=a$ and $f^{n+1}=c$, then such $f^{\mr}$ is unique up to homotopy in $\CC$.
\end{enumerate}
\end{prop}
\begin{proof}
This is straightforward. (See \cite[Proposition 2.7]{J} for {\rm (1)}, and \cite[Comparison Lemma 2.1]{J} for {\rm (2)}.)
\end{proof}

In particular, the following holds in $\CAC$.
\begin{cor}\label{CorNExSeq}
Let $\C$ be an additive category, let $A,C\in\C$ be any pair of objects. For any pair of $n$-exact sequences $X^{\mr},Y^{\mr}\in\CAC$, we have
\[ |\KAC(X^{\mr},Y^{\mr})|\le1. \]
Thus if $\CAC(X^{\mr},Y^{\mr})\ne\emptyset$ and $\CAC(Y^{\mr},X^{\mr})\ne\emptyset$, then $X^{\mr}$ and $Y^{\mr}$ are homotopically equivalent in $\CAC$.
\end{cor}
\begin{proof}
This is an immediate consequence of Proposition~\ref{PropNExSeq} {\rm (2)}.
\end{proof}

\begin{dfn}\label{DefLambda}
Let $\C$ be an additive category, and let $A,C\in\C$ be any pair of objects. Denote the class of all homotopic equivalence classes of $n$-exact sequences in $\CAC$ by $\LAC$. This is a subclass of $\Ob(\KAC)/\cong$.

For $[X^{\mr}],[Y^{\mr}]\in\LAC$, we write $[X^{\mr}]\le [Y^{\mr}]$ if $\CAC(X^{\mr},Y^{\mr})\ne\emptyset$. By Corollary \ref{CorNExSeq}, this relation makes $\LAC$ a poset (provided it forms a set).
\end{dfn}

\begin{cor}\label{CorIsolated}
For any $n$-exact sequence $X^{\mr}\in\CAC$, the following are equivalent.
\begin{enumerate}
\item $[X^{\mr}]$ is isolated, in the sense that
\[ [X^{\mr}]\le [X^{\prime\mr}]\ \ \text{or}\ \ [X^{\prime\mr}]\le [X^{\mr}]\ \ \Rightarrow\ \ [X^{\mr}]=[X^{\prime\mr}]. \]
holds in $\LAC$.
\item $X^{\mr}$ satisfies the following {\rm (I1)} and {\rm (I2)} for any $n$-exact sequence $Y^{\mr}\in\CC$.
\begin{itemize}
\item[{\rm (I1)}] If there is $f^{\mr}\in\CC(X^{\mr},Y^{\mr})$ in which $f^0=a$ and $f^{n+1}=c$ are isomorphisms, then $f^{\mr}$ is a homotopy equivalence in $\CC$.
\item[{\rm (I2)}] Dually, if there is $g^{\mr}\in\CC(Y^{\mr},X^{\mr})$ in which $g^0$ and $g^{n+1}$ are isomorphisms, then $g^{\mr}$ is a homotopy equivalence in $\CC$.
\end{itemize}
\end{enumerate}
\end{cor}
\begin{proof}
Assume that $[X^{\mr}]$ is isolated in $\LAC$, and let us show {\rm (I1)}. Suppose that $Y^{\mr}\in\CC(X^{\mr},Y^{\mr})$ is an $n$-exact sequence, and let $f^{\mr}\in\CC(X^{\mr},Y^{\mr})$ be a morphism in which $f^0=a, f^{n+1}=c$ are isomorphisms. Then
\begin{equation}\label{exAYC}
A\ov{d_Y^0\ci a}{\lra}Y^1\ov{d_Y^1}{\lra}Y^2\to\cdots\ov{d_Y^{n-1}}{\lra}Y^n\ov{c\iv\ci d_Y^n}{\lra}C
\end{equation}
is an $n$-exact sequence in $\CAC$, with an isomorphism $(a,\id,\ldots,\id,c)$ to $Y^{\mr}$ in $\CC$. Since $(\id,f^1,\ldots,f^n,\id)$ gives a morphism from $X^{\mr}$ to $(\ref{exAYC})$ in $\CAC$, it becomes a homotopy equivalence by {\rm (1)}. As their composition, $f^{\mr}$ gives a homotopy equivalence in $\CC$. Similarly for {\rm (I2)}.

Conversely, assume that $X^{\mr}$ satisfies {\rm (I1)}, and suppose $[X^{\mr}]\le [X^{\prime\mr}]$ holds for some $n$-exact sequence $X^{\prime\mr}\in\CAC$. Then there is a morphism $f^{\mr}\in\CAC(X^{\mr},X^{\prime\mr})$, which becomes a homotopy equivalence in $\CAC$ by {\rm (I1)} and Remark~\ref{RemHtpyEqDiff}. This means $[X^{\mr}]=[X^{\prime\mr}]$. Similarly, {\rm (I2)} shows $[X^{\prime\mr}]\le [X^{\mr}]\Rightarrow [X^{\mr}]=[X^{\prime\mr}]$.
\end{proof}

\begin{dfn}\label{DefNPullBack}(\cite[(dual of) Definition 2.11]{J})
Let $Y^{\mr}\in\CC$ be any object. A commutative diagram in $\C$
\begin{equation}\label{Diag_NPB}
\xy
(-20,6)*+{X^1}="0";
(-6,6)*+{X^2}="2";
(7,6)*+{\cdots}="4";
(20,6)*+{X^n}="6";
(34,6)*+{X^{n+1}}="8";
(48,6)*+{}="9";
(-20,-6)*+{Y^1}="10";
(-6,-6)*+{Y^2}="12";
(7,-6)*+{\cdots}="14";
(20,-6)*+{Y^n}="16";
(34,-6)*+{Y^{n+1}}="18";
(48,-6)*+{}="19";
{\ar^{d_X^1} "0";"2"};
{\ar^{d_X^2} "2";"4"};
{\ar^{d_X^{n-1}} "4";"6"};
{\ar^{d_X^n} "6";"8"};
%
{\ar_{d_Y^1} "10";"12"};
{\ar_{d_Y^2} "12";"14"};
{\ar_{d_Y^{n-1}} "14";"16"};
{\ar_{d_Y^n} "16";"18"};
%
{\ar_{f^1} "0";"10"};
{\ar_{f^2} "2";"12"};
{\ar_{f^n} "6";"16"};
{\ar^{f^{n+1}} "8";"18"};
{\ar@{}|\circlearrowright "0";"12"};
{\ar@{}|\circlearrowright "2";"14"};
{\ar@{}|\circlearrowright "6";"18"};
\endxy
\end{equation}
is called an {\it $n$-pullback diagram} if
\begin{equation}\label{XMfY}
X^1\ov{d^0}{\lra}X^2\oplus Y^1\ov{d^1}{\lra}X^3\oplus Y^2\ov{d^2}{\lra}\cdots\ov{d^{n-1}}{\lra}X^{n+1}\oplus Y^n\ov{d^n}{\lra}Y^{n+1}
\end{equation}
is an $n$-kernel sequence, where $d^i$ are defined by
\begin{eqnarray*}
d^0&=& \left[\begin{array}{c}-d_X^1\\ f^1\end{array}\right],\\
d^i&=& \left[\begin{array}{cc}-d_X^{i+1}&0\\ f^{i+1}&d_Y^i\end{array}\right]\quad(1\le i\le n-1),\\
d^n&=& \left[\begin{array}{cc}f^{n+1}&d_Y^n\end{array}\right].
\end{eqnarray*}
An {\it $n$-pushout diagram} is defined dually.
\end{dfn}

\begin{rem}\label{RemNPullBack}
Let $(\ref{Diag_NPB})$ be an $n$-pullback diagram. If we put $X^0=Y^0$, then the exactness of
\[ 0\to\C(X^0,X^1)\ov{\C(X^0,d^0)}{\lra}\C(X^0,X^2\oplus Y^1)\ov{\C(X^0,d^1)}{\lra}\C(X^0,X^3\oplus Y^2) \]
gives a unique morphism $d_X^0\in\C(X^0,X^1)$ satisfying $f^1\ci d_X^0=d_Y^0$ and $d_X^1\ci d_X^0=0$. Then the sequence
\[ X^0\ov{d_X^0}{\lra}X^1\ov{d_X^1}{\lra}\cdots\ov{d_X^n}{\lra}X^{n+1} \]
gives an object $X^{\mr}\in\CC$, and $f^{\mr}=(\id,f^1,\ldots,f^{n+1})\in\CC(X^{\mr},Y^{\mr})$ becomes a morphism. Sequence $(\ref{XMfY})$ is nothing but the mapping cone $M^{\mr}_f$ (in Definition~\ref{DefMapCone}) of this morphism $f^{\mr}$.
\end{rem}

\begin{dfn}\label{Defcast}
Let $Y^{\mr}\in\CAC$ be any $n$-exact sequence, and denote its homotopic equivalence class in $\CAC$ by $[Y^{\mr}]$, as before. Let $c\in\C(C\ppr,C)$ be any morphism.

If there exists an $n$-exact sequence $X^{\mr}\in\CACp$ equipped with a morphism
\[ f^{\mr}=(\id_A,f^1,\ldots,f^n,f^{n+1}=c)\in\CC(X^{\mr},Y^{\mr}) \]
which makes $(\ref{Diag_NPB})$ an $n$-pullback diagram, then we define $c\uas [Y^{\mr}]$ to be
\begin{equation}\label{cYX}
c\uas [Y^{\mr}]=[X^{\mr}].
\end{equation}
Dually, for a morphism $a\in\C(A,A\ppr)$, the class $a\sas [Y^{\mr}]$ is defined by using an $n$-pushout diagram when it exists.
Well-definedness of this definition will be shown in Proposition~\ref{PropcastWellDef}.
\end{dfn}

\begin{lem}\label{LemNPBJasso}
Let $f^{\mr}=(\id_A,f^1,\ldots,f^n,f^{n+1}=c)\in\CC({}_AX^{\mr}_{C\ppr},{}_AY^{\mr}_C)$ be any morphism, which makes $(\ref{Diag_NPB})$ an $n$-pullback diagram. Then for any morphism
\[ g^{\mr}=(a,g^1,\ldots,g^n,c)\in\CC({}_{A\ppr}Z^{\mr}_{C\ppr},{}_AY^{\mr}_C), \]
there exists a morphism $h^{\mr}=(a,h^1,\ldots,h^n,\id)\in\CC(Z^{\mr},X^{\mr})$ and a homotopy $\vp^{\mr}=(0,\vp^2,\vp^3,\ldots,\vp^n,0)$ which gives $g^{\mr}\un{\vp^{\mr}}{\sim}f^{\mr}\ci h^{\mr}$. Moreover, such $h^{\mr}$ is unique up to homotopy.
\end{lem}
\begin{proof}
This is shown in a straightforward way, only using the fact that $M^{\mr}_f$ is an $n$-kernel sequence (cf. dual of \cite[Proposition 2.13]{J} and Remark~\ref{RemNPullBack}).
\end{proof}

\begin{prop}\label{PropcastWellDef}
For any $n$-exact sequence ${}_AY^{\mr}_C\in\CAC$ and any $c\in\C(C\ppr,C)$, the class $c\uas [Y^{\mr}]$ in $(\ref{cYX})$ is unique if it exists, only depending on the homotopic equivalence class $[Y^{\mr}]$.
\end{prop}
\begin{proof}
Suppose that there are $X^{\mr}\in\CACp$ and $f^{\mr}\in\CC(X^{\mr},Y^{\mr})$ with the required properties in Definition~\ref{Defcast} to give $c\uas [Y^{\mr}]=[X^{\mr}]$. Lemma~\ref{LemNPBJasso} shows that $[X^{\mr}]$ is unique for $Y^{\mr}$. Let us show that it only depends on the class $[Y^{\mr}]$.
Assume that $Y^{\mr}$ and $Y^{\prime\mr}$ are homotopically equivalent in $\CAC$. Then $Y^{\prime\mr}$ is also an $n$-exact sequence. Take a homotopy equivalence $y^{\mr}\in\CAC(Y^{\mr},Y^{\prime\mr})$ in $\CAC$, and put $f^{\prime\mr}=y^{\mr}\circ f^{\mr}$. Proposition~\ref{PropMapCone} applied to
\[
\xy
(-7,6)*+{X^{\mr}}="0";
(7,6)*+{Y^{\mr}}="2";
(-7,-6)*+{X^{\mr}}="4";
(7,-6)*+{Y^{\prime\mr}}="6";
{\ar^{f^{\mr}} "0";"2"};
{\ar@{=} "0";"4"};
{\ar^{y^{\mr}} "2";"6"};
{\ar_{f^{\prime\mr}} "4";"6"};
{\ar@{}|\circlearrowright "0";"6"};
\endxy
\]
gives a homotopy equivalence between $M^{\mr}_f$ and $M^{\mr}_{f\ppr}$. In particular $M^{\mr}_{f\ppr}$ also becomes an $n$-kernel sequence. Thus $X^{\mr}\in\CACp$ and $f^{\prime\mr}\in\CC(X^{\mr},Y^{\prime\mr})$ satisfy the required properties to give $c\uas [Y^{\prime\mr}]=[X^{\mr}]$. This shows $c\uas [Y^{\prime\mr}]=[X^{\mr}]=c\uas [Y^{\mr}]$.
\end{proof}

The following is the definition of an $n$-exact category in \cite{J}. Later we will rephrase it in Definition~\ref{DefModifNExCat} (see Proposition~\ref{PropJassoEquiv}).

\begin{dfn}\label{DefNExCat}(\cite[Definition 4.2]{J})
Let $\C$ be an additive category, and let $\Xcal$ be a class of $n$-exact sequences in $\C$. The pair $\CX$ is called an {\it $n$-exact category} if it satisfies the following closedness {\rm (EC)} and conditions {\rm (E0),(E1),\ldots,(E2$\op$)}.

In the following, a morphism $a\in\C(A,B)$ is called an {\it $\Xcal$-admissible monomorphism} (respectively, an {\it $\Xcal$-admissible epimorphism}) if there is some $X^{\mr}\in\Xcal$ of the form $A\ov{a}{\lra}B\to X^2\to\cdots\to X^{n+1}$ $($resp. $X^0\to\cdots\to X^{n-1}\to A\ov{a}{\lra}B)$.
\begin{enumerate}
\item[{\rm (EC)}] The following holds for any morphism $f^{\mr}\in\CC(X^{\mr},Y^{\mr})$ between $n$-exact sequences $X^{\mr},Y^{\mr}$.
\begin{itemize}
\item[{\rm (i)}] If $f^k$ and $f^{k+1}$ are isomorphisms for some $k\in\{0,\ldots,n\}$, then $X^{\mr}\in\Xcal$ holds if and only if $Y^{\mr}\in\Xcal$.
\item[{\rm (ii)}] If $f^0$ and $f^{n+1}$ are isomorphisms, then $X^{\mr}\in\Xcal$ holds if and only if $Y^{\mr}\in\Xcal$.
\end{itemize}
\item[{\rm (E0)}] The sequence ${}_0\O^{\mr}_0\in\CC$ (see Proposition~\ref{PropZeroEA})
\[ 0\to0\to\cdots\to0 \]
belongs to $\Xcal$.
\item[{\rm (E1)}] $\Xcal$-admissible monomorphisms are closed by composition.
\item[{\rm (E1$\op$)}] Dually, $\Xcal$-admissible epimorphisms are closed by composition.
\item[{\rm (E2)}] For any $X^{\mr}\in\Xcal$, any $Y^0\in\C$ and any $f^0\in\C(X^0,Y^0)$, there is an $n$-pushout diagram in $\C$ as follows.
\[
\xy
(-20,6)*+{X^0}="0";
(-6,6)*+{X^1}="2";
(7,6)*+{X^2}="4";
(20,6)*+{\cdots}="6";
(34,6)*+{X^n}="8";
(48,6)*+{}="9";
(-20,-6)*+{Y^0}="10";
(-6,-6)*+{Y^1}="12";
(7,-6)*+{Y^2}="14";
(20,-6)*+{\cdots}="16";
(34,-6)*+{Y^n}="18";
(48,-6)*+{}="19";
{\ar^{d_X^0} "0";"2"};
{\ar^{d_X^1} "2";"4"};
{\ar^{d_X^2} "4";"6"};
{\ar^{d_X^{n-1}} "6";"8"};
%
{\ar_{d_Y^0} "10";"12"};
{\ar_{d_Y^1} "12";"14"};
{\ar_{d_Y^2} "14";"16"};
{\ar_{d_Y^{n-1}} "16";"18"};
%
{\ar_{f^0} "0";"10"};
{\ar^{f^1} "2";"12"};
{\ar^{f^2} "4";"14"};
{\ar^{f^n} "8";"18"};
{\ar@{}|\circlearrowright "0";"12"};
{\ar@{}|\circlearrowright "2";"14"};
{\ar@{}|\circlearrowright "6";"18"};
\endxy
\]
\item[{\rm (E2$\op$)}] Dually, for any $Y^{\mr}\in\Xcal$, any $X^{n+1}\in\C$ and any $f^{n+1}\in\C(X^{n+1},Y^{n+1})$, there is an $n$-pullback diagram in $\C$ as in $(\ref{Diag_NPB})$.
\end{enumerate}
\end{dfn}

The following is a consequence of being an $n$-exact category, shown in \cite{J}.
\begin{fact}\label{FactJasso}(\cite[dual of Proposition 4.8 $\mathrm{(iv)}\Rightarrow \mathrm{(ii)}$]{J}.)
Let $\CX$ be an $n$-exact category. For any ${}_AX^{\mr},{}_AY^{\mr}\in\Xcal$, if $f^{\mr}\in\CC(X^{\mr},Y^{\mr})$ satisfies $f^0=\id_A$, then $\Mf\in\Xcal$ holds. Dually for $f^{\mr}$ satisfying $f^{n+1}=\id$.
\end{fact}

In order to rephrase the definition of an $n$-exact category, let us consider the following conditions.
\begin{dfn}\label{DefModifNExCat}
Let $\C$ be an additive category, and let $\Xcal$ be a class of $n$-exact sequences in $\C$. Define conditions {\rm (EC$\ppr$),(E2$\ppr$),(E2$^{\prime\mathrm{op}}$)} and {\rm (EI)} as follows.
\begin{itemize}
\item[{\rm (EC$\ppr$)}] For any $A,C\in\C$,
\[ \{X\in\Xcal\mid X^0=A,X^{n+1}=C \}\se\Ob(\CAC) \]
is closed by homotopic equivalences in $\CAC$.
\item[{\rm (E2$\ppr$)}] The dual of the following {\rm (E2$^{\prime\mathrm{op}}$)}.
\item[$\ \ ${\rm (E2$^{\prime\mathrm{op}}$)}]
\begin{itemize}
\item[{\rm (i)}] For any $c\in\C(C\ppr,C)$ and any ${}_AY^{\mr}_C\in\Xcal$, there exists ${}_AX^{\mr}_{C\ppr}\in\Xcal$ equipped with a morphism $f^{\mr}\in\CC(X^{\mr},Y^{\mr})$ satisfying $f^0=\id_A$.
\item[{\rm (ii)}] For any ${}_AX^{\mr}_{C\ppr},{}_AY^{\mr}_C\in\Xcal$ and any $f^{\mr}\in\CC(X^{\mr},Y^{\mr})$ satisfying $f^0=\id_A$, we have $\Mf\in\Xcal$.
\end{itemize}
\item[{\rm (EI)}] $[X^{\mr}]\in\LAC$ is isolated, for any ${}_AX^{\mr}_C\in\Xcal$.
\end{itemize}
\end{dfn}

In Proposition~\ref{PropJassoEquiv}, we will see that conditions {\rm (EC),(E2),(E2$\op$)} in Definition~\ref{DefNExCat} can be replaced by the above conditions. First let us show the following.
\begin{lem}\label{LemConverseJE}
Let $\C$ be an additive category, and let $\Xcal$ be a class of $n$-exact sequences in $\C$. If $\CX$ satisfies {\rm (EC$\ppr$),(E2$\ppr$),(E2$^{\prime\mathrm{op}}$)}, then the following holds.
\begin{enumerate}
\item Let $A,C\in\C$ be any pair of objects. If ${}_AX^{\mr}_C,{}_AY^{\mr}_C\in\Xcal$, then any $f^{\mr}\in\CAC(X^{\mr},Y^{\mr})$ is a homotopy equivalence in $\CAC$.
\item  For any $c\in\C(C\ppr,C)$, the class $\Xcal$ is closed by $c\uas$. Namely, for any ${}_AY^{\mr}_C\in\Xcal$, there exists ${}_AX^{\mr}_{C\ppr}\in\Xcal$ which gives $c\uas [Y^{\mr}]=[X^{\mr}]$. Dually, $\Xcal$ is closed by $a\sas$ for any morphism $a$ in $\C$.
\item For any ${}_AX^{\mr}_C\in\Xcal$, any $a\in\C(A,A\ppr)$ and $c\in\C(C\ppr,C)$, we have $a\sas(c\uas [X^{\mr}])=c\uas(a\sas [X^{\mr}])$ in $\Lambda^{n+2}_{(A\ppr,C\ppr)}$.
\item $\Xcal$ is closed by homotopic equivalences in $\CC$.
\end{enumerate}
\end{lem}
\begin{proof}
{\rm (1)} By {\rm (E2$^{\prime\mathrm{op}}$)}, we have $\Mf\in\Xcal$. In particular $\Mf$ is $n$-exact. Since $d_{M_f}^n\co C\oplus Y^n\to C$ has a section $\begin{bmatrix}1\\0\end{bmatrix}\co C\to C\oplus Y^n$, we can construct a homotopy
\[ 0\un{\vp^{\mr}}{\sim}\id_{\Mf}\co\Mf\to\Mf \]
satisfying $\vp^{n+1}=\begin{bmatrix}1\\0\end{bmatrix}$ (cf. dual of \cite[Proposition 2.6]{J}.) If we write $\vp^k$ as
\[ \vp^1=[p^2\ q^1]\co X^2\oplus Y^1\to X^1, \]
\[ \vp^k=\left[\begin{array}{cc}p^{k+1}&q^k\\ r^{k+1}&s^k\end{array}\right]\co X^{k+1}\oplus Y^k\to X^k\oplus Y^{k-1}\quad(1\le k\le n), \]
then by definition they satisfy
\[ \vp^1\ci d_{M_f}^0=\id\ \ \text{and}\ \ d_{M_f}^{k-1}\ci\vp^k+\vp^{k+1}\ci d_{M_f}^k=\id\quad (1\le k\le n). \]
In particular we have
\begin{eqnarray*}
&d_{M_f}^0\ci q^1+\begin{bmatrix}q^2\\ s^2\end{bmatrix}\ci d_Y^1=\begin{bmatrix}0\\1\end{bmatrix},&\\
&q^{k+1}\ci d_Y^k=d_X^k\ci q^k\quad(1\le k\le n-1),&\\
&d_Y^n=d_X^n\ci q^n.&
\end{eqnarray*}
Then the monomorphicity of $d_{M_f}^0$ and the equality
\[ d_{M_f}^0\ci q^1 \ci d_Y^0=\begin{bmatrix}0\\1\end{bmatrix}\ci d_Y^0-\begin{bmatrix}q^2\\ s^2\end{bmatrix}\ci d_Y^1\ci d_Y^0=d_{M_f}^0\ci d_X^0  \]
shows $q^1\ci d_Y^0=d_X^0$. Thus $q^{\mr}=(\id,q^1,\ldots,q^n,\id)$ gives a morphism $q^{\mr}\in\CAC(Y^{\mr},X^{\mr})$. Corollary \ref{CorNExSeq} shows that $f^{\mr}$ is a homotopy equivalence in $\CAC$.

{\rm (2)} This follows immediately from {\rm (E2$^{\prime\mathrm{op}}$)} and the definition of $c\uas [Y^{\mr}]$. Dually for the closedness by $a\sas$.

{\rm (3)} By {\rm (2)}, there are ${}_{A\ppr}Y^{\mr}_{C\ppr},{}_{A\ppr}Z^{\mr}_{C\ppr}\in\Xcal$ which give
\[ a\sas(c\uas [X^{\mr}])=[Y^{\mr}]\quad\text{and}\quad c\uas(a\sas [X^{\mr}])=[Z^{\mr}]. \]
By Lemma~\ref{LemNPBJasso} and its dual, we find a morphism $f^{\mr}\in\Cbf^{n+2}_{(A\ppr,C\ppr)}(Y^{\mr},Z^{\mr})$. By {\rm (1)} this becomes a homotopy equivalence in $\Cbf^{n+2}_{(A\ppr,C\ppr)}$, and thus $[Y^{\mr}]=[Z^{\mr}]$ holds.

{\rm (4)} Let $f^{\mr}=(a,f^1,\ldots,f^n,c)\in\CC({}_AX^{\mr}_C,{}_BY^{\mr}_D)$ be a homotopy equivalence in $\CC$, with a homotopy inverse $g^{\mr}=(b,g^1,\ldots,g^n,d)$. Assume $X^{\mr}\in\Xcal$, and let us show $Y^{\mr}\in\Xcal$. Existence of a homotopy equivalence implies that $Y^{\mr}$ is also an $n$-exact sequence. By {\rm (2)}, there are ${}_BU^{\mr}_C,{}_AV^{\mr}_D\in\Xcal$ which give
\[ a\sas[X^{\mr}]=[U^{\mr}]\quad\text{and}\quad d\uas [X^{\mr}]=[V^{\mr}]. \]
By {\rm (2)} and {\rm (3)}, there is ${}_BZ^{\mr}_D\in\Xcal$ which gives $[Z^{\mr}]=d\uas [U^{\mr}]=a\sas [V^{\mr}]$.
Remark that there are morphisms $u^{\mr}\in\CC(X^{\mr},U^{\mr})$ and $v^{\mr}\in\CC(V^{\mr},X^{\mr})$ satisfying $u^0=a,u^{n+1}=\id_C$ and $v^0=\id_A,v^{n+1}=d$. Applying Lemma~\ref{LemNPBJasso} to $u^{\mr}\ci g^{\mr}\in\CC(Y^{\mr},U^{\mr})$, we obtain some $y^{\mr}\in\CC(Y^{\mr},Z^{\mr})$ satisfying $y^0=a\ci b$ and $y^{n+1}=\id_D$. Since $f^{\mr}\ci g^{\mr}\sim\id_{Y^{\mr}}$ by assumption, there is $\vp^1\in\C(Y^1,B)$ which gives $\vp^1\ci d_Y^0=\id_B-a\ci b$. Modifying $y^{\mr}$, we obtain a morphism
\[ (\id_B,y^1+d_Z^0\ci\vp^1,y^2,\ldots,y^n,\id_D)\in\Cbf^{n+2}_{(B,D)}(Y^{\mr},Z^{\mr}). \]
Similarly, the dual of Lemma~\ref{LemNPBJasso} applied to $f^{\mr}\ci v^{\mr}$ gives $z^{\mr}\in\CC(Z^{\mr},Y^{\mr})$ satisfying $z^0=\id_B$ and $z^{n+1}=c\ci d$, and thus we obtain
\[ (\id_B,z^1,\ldots,z^{n-1},z^n+\vp^{n+1}\ci d_Z^n,\id_D)\in\Cbf^{n+2}_{(B,D)}(Z^{\mr},Y^{\mr}). \]
By Corollary \ref{CorNExSeq}, it follows $[Y^{\mr}]=[Z^{\mr}]$. Thus {\rm (EC$\ppr$)} shows $Y^{\mr}\in\Xcal$.
\end{proof}

\begin{prop}\label{PropJassoEquiv}
Let $\C$ be an additive category, and let $\Xcal$ be a class of $n$-exact sequences in $\C$. Assume that $\CX$ satisfies {\rm (E0),(E1),(E1$\op$)}. Then the following are equivalent.
\begin{enumerate}
\item $\CX$ is an $n$-exact category.
\item $\CX$ satisfies conditions {\rm (EC$\ppr$),(E2$\ppr$),(E2$^{\prime\mathrm{op}}$)} and {\rm (EI)} in Definition~\ref{DefModifNExCat}.
\end{enumerate}
\end{prop}
\begin{proof}
$(1)\Rightarrow(2)$.
{\rm (EC$\ppr$)} is a particular case of {\rm (EC)(ii)}. {\rm (E2$^{\prime\mathrm{op}}$)(i)} follows from {\rm (E2$\op$)} and Remark~\ref{RemNPullBack}. {\rm (E2$^{\prime\mathrm{op}}$)(ii)} is given in Fact \ref{FactJasso}. Dually for {\rm (E2$\ppr$)}.
Thus we can apply Lemma~\ref{LemConverseJE} to $\CX$. Then {\rm (EI)} follows from {\rm (EC)(ii)} and Remark~\ref{RemHtpyEqDiff}.

$(2)\Rightarrow(1)$.
Let $\CX$ be as in {\rm (2)}. Let us confirm conditions in Definition~\ref{DefNExCat}. {\rm (EC)(i)} follows from Proposition~\ref{PropNExSeq} {\rm (1)} and Lemma~\ref{LemConverseJE} {\rm (4)}. To show {\rm (EC)(ii)}, let $f^{\mr}\in\CC(X^{\mr},Y^{\mr})$ be a morphism between $n$-exact sequences $X^{\mr},Y^{\mr}$ in which $f^0$ and $f^{n+1}$ are isomorphisms. If one of $X^{\mr},Y^{\mr}$ belongs to $\Xcal$, then {\rm (EI)} implies that $f^{\mr}$ is a homotopy equivalence in $\CC$, by Corollary \ref{CorIsolated}. Thus the other also belongs to $\Xcal$ by Lemma~\ref{LemConverseJE} {\rm (4)}. {\rm (E2$\op$)} follows from {\rm (E2$^{\prime\mathrm{op}}$)}. Dually for {\rm (E2)}.
\end{proof}

We proceed to show that each $n$-exact category is $n$-exangulated. By the equivalence shown in Proposition~\ref{PropJassoEquiv}, we may use conditions {\rm (EC$\ppr$)},{\rm (E2$\ppr$)},{\rm (E2$^{\prime\mathrm{op}}$)} and consequently Lemma~\ref{LemConverseJE}. Indeed, we can avoid using {\rm (EI)} (see Remark~\ref{RemAvoid}). We begin by defining the bifunctor $\E$ similarly to the usual Yoneda extension functor. The procedure follows very closely the classical case of exact categories (see e.g. \cite{FS}).
\begin{dfn}\label{YonedaExt}
Let $(\mathscr{C}, \mathcal{X})$ be an $n$-exact category. For $A,C \in\C$, let $\E(C,A)$ be the subclass of $\LAC$ consisting of all $[X^{\mr}]$ such that $X^{\mr}\in\Xcal$. This is well-defined by {\rm (EC$\ppr$)}. From now on we assume that $\E(C,A)$ is a set for all $A,C \in \mathscr{C}$. We consider the assignment $(C,A) \mapsto \E(C,A)$ as a functor $$\E\co\C\op\ti\C\to\Sets$$ by defining $\E(c,a)[X^{\mr}] = a\sas(c\uas [X^{\mr}])$ for all  $(c,a)\in \C(C\ppr,C) \times \C(A,A\ppr)$ and ${}_AX^{\mr}_C\in\Xcal$. That $\E$ is well-defined is shown in Lemma~\ref{YonedaExtWell}.
\end{dfn}
\begin{rem}\label{YonedaOnMorphisms}
To compute the functor $\E$, it is useful to note that
\[ c\uas [X^{\mr}]=[Y^{\mr}] \]
holds for ${}_AX^{\mr}, {}_AY^{\mr}\in\Xcal$ if and only if there is $f^{\mr}\in\CC(Y^{\mr},X^{\mr})$ such that $f^0=\id_A$ and $f^{n+1}=c$. This follows from {\rm (E2$^{\prime\mathrm{op}}$)} (or alternatively from \cite[Proposition 4.8]{J}). Dually,
\[ a\sas [X^{\mr}]=[Y^{\mr}] \]
holds for $X^{\mr}_C,Y^{\mr}_C\in\Xcal$ if and only if there is $f^{\mr}\in\CC(X^{\mr},Y^{\mr})$ such that $f^0=a$ and $f^{n+1}=\id_C$.
\end{rem}

\begin{lem}\label{YonedaExtWell}
The functor $$\E\co\C\op\ti\C\to\Sets$$ in Definition~\ref{YonedaExt} is well-defined.
\end{lem}
\begin{proof}
First note that $\E(c,a)\co\E(C,A)\to\E(C\ppr,A\ppr)$ is a well-defined map by Lemma~\ref{LemConverseJE} {\rm (2)}.

Considering the identity morphism on ${}_AX^{\mr}_C\in\Xcal$ and Remark~\ref{YonedaOnMorphisms} we find $(\id_C)\uas [X^{\mr}]=[X^{\mr}]$ and $(\id_A)\sas [X^{\mr}]=[X^{\mr}]$ and so $\E(\id_C,\id_A)=\id_{\E(C,A)}$.

Now let ${}_AX^{\mr},{}_AY^{\mr},{}_AZ^{\mr}\in\Xcal$ and suppose that $c\uas [X^{\mr}]=[Y^{\mr}]$ and $d\uas [Y^{\mr}]=[Z^{\mr}]$. By Remark~\ref{YonedaOnMorphisms} there are $f^{\mr}\co Y^{\mr} \to X^{\mr}$ and $g^{\mr}\co Z^{\mr}\to Y^{\mr}$, with $f^0=\id_A$, $f^{n+1}=c$, $g^0=\id_A$, $g^{n+1}=d$. By considering $f^{\mr}\ci g^{\mr}\co Z^{\mr}\to X^{\mr}$ we find that $(cd)\uas=d\uas c\uas$. Similarly $(ab)\sas=a\sas b\sas$. By Lemma~\ref{LemConverseJE} {\rm (3)} it follows that
\[ \E((d,b)\circ(c,a))=\E(cd,ba)=(cd)\uas(ba)\sas=d\uas c\uas b\sas a \sas=d\uas b\sas c\uas a\sas=\E(d,b) \circ \E(c,a). \]
\end{proof}

Next we want to endow $\E(C,A)$ with the structure of an abelian group. As for exact categories this is done using the Baer sum.
\begin{rem}\label{RemAddX}
In an $n$-exact category $(\C,\Xcal)$,
\[ X^{\mr},Y^{\mr}\in\Xcal\ \Rightarrow\ X^{\mr}\oplus Y^{\mr}\in\Xcal \]
holds. This has been shown in \cite[Proposition 4.6]{J}. We also remark that if $(\C,\Xcal)$ satisfies {\rm (EC$\ppr$),(E0),(E1),(E1$\op$),(E2$\ppr$),(E2$^{\prime\mathrm{op}}$)}, then the same proof as in \cite[Lemma 4.5, Proposition 4.6]{J} works, because of Lemma~\ref{LemConverseJE}.
\end{rem}

\begin{dfn}\label{BaerSum}
Let ${}_AX^{\mr}_C, {}_AY^{\mr}_C\in\Xcal$. As in Remark~\ref{RemAddX}, the direct sum $X^{\mr}\oplus Y^{\mr}\in\Xcal$. Moreover, $[X^{\mr}\oplus Y^{\mr}]$ only depends on $[X^{\mr}]$ and $[Y^{\mr}]$ so we may define
\[ [X^{\mr}]\oplus [Y^{\mr}]=[X^{\mr}\oplus Y^{\mr}]. \]
Finally define the Baer sum of $[X^{\mr}]$ and $[Y^{\mr}]$ to be
\[ [X^{\mr}] + [Y^{\mr}] = (\Delta_C)\uas (\nabla_A)\sas([X^{\mr}] \oplus [Y^{\mr}]) \in \E(C,A). \]
\end{dfn}
\begin{rem}\label{BaerAssociative}
By Lemma~\ref{LemConverseJE} {\rm (3)}, we also have $$[X^{\mr}]+[Y^{\mr}]=(\nabla_A)\sas(\Delta_C)\uas ([X^{\mr}]\oplus [Y^{\mr}]).$$ Using this together with $X^{\mr}\oplus(Y^{\mr}\oplus Z^{\mr})=(X^{\mr}\oplus Y^{\mr})\oplus Z^{\mr}$ one easily checks that
\[ ([X^{\mr}] + [Y^{\mr}])+[Z^{\mr}] = [X^{\mr}] + ([Y^{\mr}]+[Z^{\mr}]). \]
\end{rem}

To show that $\E(C,A)$ with the Baer sum is an abelian group, we will use the following result.
\begin{lem}\label{Factorization}
Let ${}_AX^{\mr}_C, {}_{A\ppr}Y^{\mr}_{C\ppr}\in\Xcal$ and $f^{\mr}\co X^{\mr}\to Y^{\mr}$ with $f^0=a$, $f^{n+1}=c$. Then
\[ a\sas[X^{\mr}]=c\uas[Y^{\mr}]. \]
\end{lem}
\begin{proof}
By Lemma~\ref{LemNPBJasso} and Remark~\ref{YonedaOnMorphisms}, there are ${}_{A\ppr}Z^{\mr}_C\in\Xcal$ and morphisms $g^{\mr}\co X^{\mr}\to Z^{\mr}$, $h^{\mr}\co Z^{\mr}\to Y^{\mr}$ satisfying $g^0=a$, $g^{n+1}=\id_C$, $h^0=\id_{A\ppr}$ and $h^{n+1}=c$ (see also \cite[Proposition 4.9]{J}). Hence $a\sas[X^{\mr}]=[Z^{\mr}]=c\uas[Y^{\mr}]$.
\end{proof}

The following lemma is analogous to \cite[Proposition 6.10]{FS} and has a similar proof, which we include for the sake of completeness.
\begin{lem}\label{Biadditive}
Let ${}_AX^{\mr}_C,{}_{A}Y^{\mr}_C\in\Xcal$ and $a,b\in\C(A,A\ppr)$. Then the following statements hold.
\begin{enumerate}
\item $\left[\begin{array}{cc}a&0 \\0&b\end{array}\right]\sas([X^{\mr}]\oplus [Y^{\mr}])=a\sas [X^{\mr}]\oplus b\sas[Y^{\mr}]$
\item $(a+b)\sas[X^{\mr}]=a\sas [X^{\mr}]+b\sas [X^{\mr}]$
\item $a\sas([X^{\mr}]+[Y^{\mr}])=a\sas[X^{\mr}]+a\sas [Y^{\mr}]$
\end{enumerate}
\end{lem}
\begin{proof}
{\rm (1)} Write $a\sas [X^{\mr}]=[Z^{\mr}]$ and $b\sas [Y^{\mr}]=[W^{\mr}]$. Then there are $f^{\mr}\in\CC(X^{\mr},Z^{\mr})$ and $g^{\mr}\in\CC(Y^{\mr},W^{\mr})$ such that $f^0 = a$, $f^{n+1}=\id_C$, $g^0=b$ and $g^{n+1}=\id_C$ by Remark~\ref{YonedaOnMorphisms}. Considering $h^{\mr}\in\CC(X^{\mr}\oplus Y^{\mr},Z^{\mr}\oplus W^{\mr})$ given by
\[ h^k=\left[\begin{array}{cc}f^k&0 \\0&g^k\end{array}\right] \]
we find that 
\[ \left[\begin{array}{cc}a&0 \\0&b\end{array}\right]\sas([X^{\mr}]\oplus [Y^{\mr}])=[Z^{\mr}]\oplus [W^{\mr}]=a\sas [X^{\mr}]\oplus b\sas [Y^{\mr}]. \]

{\rm (2)} Consider $\Delta^{\mr}\in\CC(X^{\mr},X^{\mr} \oplus X^{\mr})$ defined by $\Delta^k=\Delta_{X^k}$. By Lemma~\ref{Factorization}, we get $(\Delta_A)\sas [X^{\mr}]=(\Delta_C)\uas([X^{\mr}]\oplus [X^{\mr}])$. Now by {\rm (1)}
\[\begin{split}
(a+b)\sas[X^{\mr}] &= 
\left(\nabla_{A\ppr}
\left[\begin{array}{cc}a&0 \\0&b\end{array}\right]
\Delta_A\right)
\sas[X^{\mr}]
= (\nabla_{A\ppr})\sas
\left[\begin{array}{cc}a&0 \\0&b\end{array}\right]\sas
(\Delta_A)\sas[X^{\mr}] \\ &= 
(\nabla_{A\ppr})\sas
\left[\begin{array}{cc}a&0 \\0&b\end{array}\right]\sas
(\Delta_C)\uas([X^{\mr}] \oplus [X^{\mr}]) \\ &=
(\nabla_{A\ppr})\sas (\Delta_C)\uas
\left[\begin{array}{cc}a&0 \\0&b\end{array}\right]\sas
([X^{\mr}] \oplus [X^{\mr}]) \\ &=
(\nabla_{A\ppr})\sas (\Delta_C)\uas
(a\sas [X^{\mr}]\oplus b\sas[X^{\mr}]) \\ &=
a\sas [X^{\mr}] + b\sas [X^{\mr}].
\end{split}\]

{\rm (3)} Using {\rm (1)} we compute
\[\begin{split}
a\sas([X^{\mr}] + [Y^{\mr}]) &= a\sas (\nabla_{A})\sas (\Delta_C)\uas ([X^{\mr}]\oplus [Y^{\mr}]) \\&= 
(\nabla_{A\ppr})\sas\left[\begin{array}{cc}a&0 \\0&a\end{array}\right]\sas(\Delta_C)\uas ([X^{\mr}]\oplus [Y^{\mr}])\\ &=
(\nabla_{A\ppr})\sas(\Delta_C)\uas\left[\begin{array}{cc}a&0 \\0&a\end{array}\right] \sas([X^{\mr}]\oplus [Y^{\mr}]) \\ &= 
(\nabla_{A\ppr})\sas (\Delta_C)\uas (a\sas [X^{\mr}]\oplus a\sas[Y^{\mr}]) \\ &=
a\sas [X^{\mr}] + a\sas [Y^{\mr}].
\end{split}\]
\end{proof}

\begin{prop}\label{YonedaBiadd}
For all $C,A\in\C$ the Baer sum defines the structure of an abelian group on $\E(C,A)$. This enhances the functor $\E$ in Definition~\ref{YonedaExt} to a biadditive functor
$$\E\co\C\op\ti\C\to\Ab.$$
\end{prop}
\begin{proof}
Since this is well-known for $n=1$, we assume $n\ge 2$.
Let ${}_AX^{\mr}_{C}, {}_AY^{\mr}_{C}\in\Xcal$ and consider the canonical isomorphism $t\co X^{\mr}\oplus Y^{\mr}\to Y^{\mr}\oplus X^{\mr}$. By Lemma~\ref{Factorization} we get 
\[ (t^0)\sas ([X^{\mr}] \oplus [Y^{\mr}]) = (t^{n+1})\uas ([Y^{\mr}] \oplus [X^{\mr}]) \]
Now
\[\begin{split}
[X^{\mr}]+[Y^{\mr}] &= (\Delta_C )\uas(\nabla_A)\sas ([X^{\mr}] \oplus [Y^{\mr}]) = (\Delta_C )\uas(\nabla_A t^0)\sas ([X^{\mr}] \oplus [Y^{\mr}]) \\ &=
(\Delta_C )\uas(\nabla_A)\sas(t^0)\sas ([X^{\mr}] \oplus [Y^{\mr}]) = (\Delta_C )\uas(\nabla_A)\sas(t^{n+1})\uas ([Y^{\mr}] \oplus [X^{\mr}]) \\&= 
(\Delta_C )\uas(t^{n+1})\uas(\nabla_A)\sas ([Y^{\mr}] \oplus [X^{\mr}]) \\&= 
(t^{n+1} \Delta_C )\uas(\nabla_A)\sas ([Y^{\mr}] \oplus [X^{\mr}]) \\ &= 
(\Delta_C )\uas(\nabla_A)\sas ([Y^{\mr}] \oplus [X^{\mr}]) = [Y^{\mr}] + [X^{\mr}].
\end{split}\]
Together with Remark~\ref{BaerAssociative} this shows that $\E(C,A)$ is an abelian semigroup.

Let ${}_AX^{\mr}_C\in \Xcal$ and $N^{\mr}$ be the complex
\begin{equation}\label{neutralEl}
A \ov{\id_{A}}{\lra}A\to0\to\cdots\to0\to C \ov{\id_{C}}{\lra}C
\end{equation}
It follows from {\rm (E0)} and Lemma~\ref{LemConverseJE} {\rm (4)} (or alternatively from \cite[Remark 4.7]{J}) that $N^{\mr}\in \Xcal$. By considering $f^{\mr}\in\CC(X^{\mr},N^{\mr})$ defined by $f^{n+1}=\id_C$, $f^{n}= d_X^n$ and $f^k = 0$ for $k <n$, we find that $0 \sas [X^{\mr}]=[N^{\mr}]$. By Lemma~\ref{Biadditive}, we have $[N^{\mr}]+[X^{\mr}]=0 \sas [X^{\mr}]+1 \sas [X^{\mr}] = [X^{\mr}]$, and so $[N^{\mr}]$ is the neutral element in $\E(C,A)$. Similarly, $(-1)\sas [X^{\mr}]$ is the inverse of $[X^{\mr}]$. Hence $\E(C,A)$ is an abelian group. From Lemma~\ref{Biadditive} and its dual it follows that 
$$\E\co\C\op\ti\C\to\Ab$$
is well-defined and biadditive.
\end{proof}

\begin{rem}\label{RemForTheCasen=1}
As in the above proof, the element $0\in\E(C,A)$ is given by the sequence (\ref{neutralEl}) if $n\ge 2$. If $n=1$, it is given by $A\ov{\Big[\raise1ex\hbox{\leavevmode\vtop{\baselineskip-8ex \lineskip1ex \ialign{#\crcr{$\scriptstyle{1}$}\crcr{$\scriptstyle{0}$}\crcr}}}\Big]}{\lra}A\oplus C\ov{[0\ 1]}{\lra}C$.
\end{rem}

\begin{prop}\label{nExactIsnExangulated}
Let $(\mathscr{C}, \mathcal{X})$ be an $n$-exact category such that $\E(C,A)$ is a set for all $A,C  \in \mathscr{C}$. For all $\del \in \E(C,A)$, set $\sfr(\del) = [X^{\mr}]$, where $\del  = [X^{\mr}]$. Then $(\mathscr{C},\E,\sfr)$ is $n$-exangulated.
\end{prop}
\begin{proof}
Similarly as before, we only deal with the case $n\ge 2$. As for the case $n=1$, a similar proof to the one below works, if we take Remark~\ref{RemForTheCasen=1} into account. The case $n=1$ also follows from Proposition~\ref{Prop1EAandET} and \cite[Example 2.13]{NP}.

By Proposition~\ref{YonedaBiadd} we know that $\E\co\C\op\ti\C\to\Ab$ is a biadditive functor. So it remains to check the conditions  {\rm (R0)},{\rm (R1)},{\rm (R2)} and also {\rm (EA1)},{\rm (EA2)},{\rm (EA2)$\op$}.

{\rm (R0)} Let $(a,c):{}_A\delta_{C} \to {}_{A\ppr}\rho_{C\ppr}$ be a morphism of extensions where $\delta = [X^{\mr}]$ and $\rho  =[Y^{\mr}]$. Then $a \sas [X^{\mr}] = c \uas [Y^{\mr}] = [Z^{\mr}]$ for some ${}_{A\ppr}Z^{\mr}_{C} \in \Xcal$ and so there are morphisms $f^{\mr} \co X^{\mr} \to Z^{\mr}$, $g^{\mr} \co Z^{\mr} \to Y^{\mr}$ satisfying $f^0=a$, $f^{n+1}=\id_C$, $g^0=\id_{A\ppr}$ and $g^{n+1}=c$. The composition $g^{\mr} \circ f^{\mr} \co X^{\mr} \to Y^{\mr}$ is a lift of $(a,c)$.

{\rm (R1)} Let $X^{\mr} \in \Xcal$ and $\delta = [X^{\mr}]$. We need to check that $\langle X^{\mr}, \delta \rangle$ is an $n$-exangle. Since $X^{\mr}$ is $n$-exact it is enough to check that $$\C(Y,X^n)\ov{\C(Y,d_X^n)}{\lra}\C(Y,X^{n+1})\ov{\del\ssh}{\lra}\E(Y,X^0)$$ and $$\C(X^{1},Y)\ov{\C(d_X^0,Y)}{\lra}\C(X^0,Y)\ov{\del\ush}{\lra}\E(X^{n+1},Y)$$ are exact for all $Y \in \C$. We only check the first case as the second is similar.

Let $f \co Y \to X^{n+1}$. Then $\del\ssh(f)  = f\uas [X^{\mr}]$ is zero if and only if there is a commutative diagram of the form
\[
\xy
(-43,12)*+{X^0}="0";
(-28,12)*+{X^0}="1";
(-13,12)*+{0}="2";
(0,12)*+{\cdots}="3";
(14,12)*+{0}="4";
(29,12)*+{Y}="5";
(43,12)*+{Y}="6";
(-43,0)*+{X^0}="10";
(-28,0)*+{X^1}="11";
(-13,0)*+{X^2}="12";
(0,0)*+{\cdots}="13";
(14,0)*+{X^{n-1}}="14";
(29,0)*+{X^n}="15";
(43,0)*+{X^{n+1}}="16";
{\ar@{=} "0";"1"};
{\ar^{} "1";"2"};
{\ar^{} "2";"3"};
{\ar^{} "3";"4"};
{\ar^{} "4";"5"};
{\ar@{=} "5";"6"};
{\ar@{=} "0";"10"};
{\ar^{} "1";"11"};
{\ar^{} "2";"12"};
{\ar^{} "4";"14"};
{\ar^{} "5";"15"};
{\ar^{f} "6";"16"};
{\ar_{d_X^0} "10";"11"};
{\ar_{d_X^1} "11";"12"};
{\ar_{} "12";"13"};
{\ar_{} "13";"14"};
{\ar_{d_X^{n-1}} "14";"15"};
{\ar_{d_X^{n}} "15";"16"};
{\ar@{}|\circlearrowright "0";"11"};
{\ar@{}|\circlearrowright "1";"12"};
{\ar@{}|\circlearrowright "4";"15"};
{\ar@{}|\circlearrowright "5";"16"};
\endxy
\]
Evidently, this is equivalent to $f = d_X^n\ci g$ for some $g \co Y \to X^n$, i.e., $f$ is in the image of $\C(Y,d_X^n)$. 

{\rm (R2)} immediately follows from the description of $0 \in \E(0,A)$ and $0 \in \E(A,0)$.

{\rm (EA1)} follows from {\rm (E1)} and {\rm (E1$\op$)}.

{\rm (EA2)} Let ${}_AY^{\mr}_{C}  \in \Xcal$ and $c \in \C(C',C)$. Let ${}_AX^{\mr}_{C\ppr} \in \Xcal$ such that $c\uas [Y^{\mr}] = [X^{\mr}]$. Then there is $f^{\mr}\in\CC(X^{\mr},Y^{\mr})$ such that by $f^{0}=\id_A$, $f^{n+1}= c$. We claim that this is a good lift of $(\id_A,c)$. First of all $M^{\mr}_f \in \Xcal$ by {\rm (E2$^{\prime\mathrm{op}}$)}. Next, existence of the morphism $g^{\mr}\in\CC(Y^{\mr},M^{\mr}_f)$ given by $g^0 = d^0_X$, $g^{n+1} = \id_C$ and 
\[
g^k = \left[\begin{array}{cc}0 \\\id_{Y^k}\end{array}\right]
\]
for all other $k$ shows $(d^0_X)\sas [Y^{\mr}] = [M^{\mr}_f]$. The dual case {\rm (EA2$\op$)} is similar. 
\end{proof}
\begin{rem}\label{RemAvoid}
Every $n$-exangulated category $(\mathscr{C},\E,\sfr)$ coming from an $n$-exact category $(\mathscr{C}, \mathcal{X})$ as in Proposition~\ref{nExactIsnExangulated} satisfies the condition that all inflations are monomorphisms and all deflations are epimorphisms. In fact, the arguments so far show that if $(\C,\Xcal)$ satisfies conditions {\rm (EC$\ppr$),(E0),(E1),(E1$\op$),(E2$\ppr$),(E2$^{\prime\mathrm{op}}$)}, then it gives an $n$-exangulated category of this type. Next we will show the converse of this (Proposition~\ref{PropNEAtoNEx}).
\end{rem}

\begin{lem}\label{LemNEAtoNEx}
Let $\CEs$ be an $n$-exangulated category. Assume that any $\sfr$-inflation is monomorphic, and any $\sfr$-deflation is epimorphic in $\C$. Note that this is equivalent to assuming that any $\sfr$-conflation is $n$-exact.
If we denote the class of all $\sfr$-conflations by $\Xcal$, then we have the following.
\begin{enumerate}
\item For any $n$-exangle ${}_A\Yd_C$ and any $c\in\C(C\ppr,C)$, if we put $\sfr(c\uas\del)=[X^{\mr}]$, then any lift $f^{\mr}\in\CC(X^{\mr},Y^{\mr})$ of $(\id_A,c)\co\del\to c\uas\del$ satisfies $M_f^{\mr}\in\Xcal$. In particular,
\[
\xy
(-20,6)*+{X^1}="0";
(-6,6)*+{X^2}="2";
(7,6)*+{\cdots}="4";
(20,6)*+{X^n}="6";
(34,6)*+{X^{n+1}}="8";
(48,6)*+{}="9";
(-20,-6)*+{Y^1}="10";
(-6,-6)*+{Y^2}="12";
(7,-6)*+{\cdots}="14";
(20,-6)*+{Y^n}="16";
(34,-6)*+{Y^{n+1}}="18";
(48,-6)*+{}="19";
{\ar^{d_X^1} "0";"2"};
{\ar^{d_X^2} "2";"4"};
{\ar^{d_X^{n-1}} "4";"6"};
{\ar^{d_X^n} "6";"8"};
%
{\ar_{d_Y^1} "10";"12"};
{\ar_{d_Y^2} "12";"14"};
{\ar_{d_Y^{n-1}} "14";"16"};
{\ar_{d_Y^n} "16";"18"};
%
{\ar_{f^1} "0";"10"};
{\ar_{f^2} "2";"12"};
{\ar_{f^n} "6";"16"};
{\ar^{f^{n+1}} "8";"18"};
{\ar@{}|\circlearrowright "0";"12"};
{\ar@{}|\circlearrowright "2";"14"};
{\ar@{}|\circlearrowright "6";"18"};
\endxy
\]
becomes an $n$-pullback diagram in $\C$.
\item For any pair of distinguished $n$-exangles ${}_A\Xd_C,{}_B\Yr_D$, we have 
\[ \CC(X^{\mr},Y^{\mr})=\AE(\Xd,\Yr). \]
\item If $\del,\del\ppr\in\E(C,A)$ satisfies $\sfr(\del)=\sfr(\del\ppr)$, then $\del=\del\ppr$ holds. Thus for any $A,C\in\C$, the realization $\sfr$ gives the following bijective correspondence.
\[ \E(C,A)\ov{\text{bij.}}{\lra}\frac{\{ X^{\mr}\in\Xcal\mid X^0=A,X^{n+1}=C\}}{(\text{homotopic equivalence in}\ \CAC)}. \]
\end{enumerate}
\end{lem}
\begin{proof}
{\rm (1)} By {\rm (EA2)}, there is a good lift $g^{\mr}$ of $(\id_A,c)$, which makes $\langle M^{\mr}_{g},(d_X^0)\sas\del\rangle$ a distinguished $n$-exangle by definition. Since $f^{\mr}\sim g^{\mr}$ holds by Proposition~\ref{PropNExSeq} {\rm (2)}, it follows that $f^{\mr}$ is also a good lift by Remark~\ref{RemEA2} {\rm (1)}, and thus $\Mf\in\Xcal$.

{\rm (2)} It suffices to show $\CC(X^{\mr},Y^{\mr})\se\AE(\Xd,\Yr)$. Let $f^{\mr}\in\CC(X^{\mr},Y^{\mr})$ be any morphism. By Proposition~\ref{PropShiftedReal}, there is some $g^{\mr}\in\AE(\Xd,\Yr)$ satisfying $g^0=f^0$ and $g^1=f^1$. Since $X^{\mr}$ and $Y^{\mr}$ are $n$-exact sequences, Proposition~\ref{PropNExSeq} {\rm (1)} shows $f^{\mr}\sim g^{\mr}$. Thus Proposition~\ref{PropAE1} shows $f^{\mr}\in\AE(\Xd,\Yr)$.

{\rm (3)} This immediately follows from {\rm (2)}. Indeed, $\id_{X^{\mr}}\in\CC(X^{\mr},X^{\mr})$ should give a morphism $\id_{X^{\mr}}\in\AE(\Xd,\langle X^{\mr},\del\ppr\rangle)$, which in particular satisfies $\del=\del\ppr$.
\end{proof}

\begin{prop}\label{PropNEAtoNEx}
Let $\CEs$ be an $n$-exangulated category, in which any $\sfr$-conflation is monomorphic and any $\sfr$-deflation is epimorphic. Let $\Xcal$ be the class of all $\sfr$-conflations, as in Lemma~\ref{LemNEAtoNEx}. Then, the following holds.
\begin{enumerate}
\item The pair $\CX$ satisfies conditions {\rm (EC$\ppr$),(E0),(E1),(E1$\op$),(E2$\ppr$),(E2$^{\prime\mathrm{op}}$)}.
\item Moreover, if $\CEs$ satisfies the following conditions {\rm (a),(b)} for any pair of morphisms $A\ov{a}{\lra}B\ov{b}{\lra}C$ in $\C$, then $\CX$ also satisfies {\rm (EI)}, and thus becomes an $n$-exact category in the sense of \cite{J} by Proposition~\ref{PropJassoEquiv}.
\begin{itemize}
\item[{\rm (a)}] If $b\ci a$ is an $\sfr$-inflation, then so is $a$.
\item[{\rm (b)}] If $b\ci a$ is an $\sfr$-deflation, then so is $b$.
\end{itemize}
\end{enumerate}
\end{prop}
\begin{proof}
{\rm (1)} {\rm (EC$\ppr$)} is obvious from the definition of $\Xcal$. {\rm (E0)} follows from {\rm (R2)}. {\rm (E1)} and {\rm (E1$\op$)} follow from {\rm (EA1)}. {\rm (E2$^{\prime\mathrm{op}}$)(i)} follows from the functoriality of $\E$ and {\rm (R0)}. {\rm (E2$^{\prime\mathrm{op}}$)(ii)} follows from Lemma~\ref{LemNEAtoNEx} {\rm (1),(2)}. Dually for {\rm (E2$\ppr$)}.

{\rm (2)} By Corollary \ref{CorIsolated}, it suffices to show that any $X^{\mr}\in\Xcal$ satisfies {\rm (I1)} and {\rm (I2)} in Corollary \ref{CorIsolated}. Since {\rm (I2)} is dual to {\rm (I1)}, we only show that {\rm (b)} implies {\rm (I1)}. For ${}_AX^{\mr}_C\in\Xcal$, let $f^{\mr}\in\CC(X^{\mr},Y^{\mr})$ be any morphism to an $n$-exact sequence $Y^{\mr}$, in which $f^0$ and $f^{n+1}$ are isomorphisms in $\C$. Modifying $X^{\mr}$ using isomorphisms $f^0$ and $f^{n+1}$ by Corollary \ref{CorDEAInv}, we may assume $f^0=\id_A$ and $f^{n+1}=\id_C$ from the beginning.

By the equality $d_Y^n\ci f^n=d_X^n$, condition {\rm (b)} implies that $d_Y^n$ is an $\sfr$-deflation. Thus there is $Z^{\mr}\in\Xcal$ of the form
\[ Z^0\ov{d_Z^0}{\lra}Z^1\ov{d_Z^1}{\lra}\cdots\ov{d_Z^{n-2}}{\lra}Z^{n-1}\ov{d_Z^{n-1}}{\lra}Y^n\ov{d_Y^n}{\lra}Y^{n+1}. \]
Since both $Y^{\mr}$ and $Z^{\mr}$ are $n$-exact sequences, the commutative square in $\C$
\[
\xy
(-20,6)*+{Z^0}="0";
(-6,6)*+{Z^1}="2";
(7,6)*+{\cdots}="4";
(20,6)*+{Z^{n-1}}="6";
(34,6)*+{Y^n}="8";
(48,6)*+{Y^{n+1}}="9";
(-20,-6)*+{Y^0}="10";
(-6,-6)*+{Y^1}="12";
(7,-6)*+{\cdots}="14";
(20,-6)*+{Y^{n-1}}="16";
(34,-6)*+{Y^n}="18";
(48,-6)*+{Y^{n+1}}="19";
{\ar^{d_Z^0} "0";"2"};
{\ar^{d_Z^1} "2";"4"};
{\ar^(0.4){d_Z^{n-2}} "4";"6"};
{\ar^{d_Z^{n-1}} "6";"8"};
{\ar^(0.46){d_Y^n} "8";"9"};
{\ar_{d_Y^0} "10";"12"};
{\ar_{d_Y^1} "12";"14"};
{\ar_(0.4){d_Y^{n-2}} "14";"16"};
{\ar_{d_Y^{n-1}} "16";"18"};
{\ar_(0.46){d_Y^n} "18";"19"};
{\ar@{=} "8";"18"};
{\ar@{=} "9";"19"};
{\ar@{}|\circlearrowright "8";"19"};
\endxy
\]
can be completed into a homotopy equivalence $Z^{\mr}\to Y^{\mr}$ by Proposition~\ref{PropNExSeq} {\rm (1)}. Thus $Z^{\mr}\in\Xcal$ implies $Y^{\mr}\in\Xcal$ by Lemma~\ref{LemConverseJE} {\rm (4)}.
\end{proof}

\begin{rem}
Let $\CEs$ be an $n$-exangulated category, and let $\F\se\E$ be a closed subfunctor. Trivially, if any $\sfr$-inflation is monomorphic (respectively, if any $\sfr$-deflation is epimorphic), then so is any $\sfr|_{\F}$-inflation (resp. $\sfr|_{\F}$-deflation).

Let $\Xcal_{\sfr}$ and $\Xcal_{\sfr|_{\F}}$ be the classes of all $\sfr$-conflations and $\sfr|_{\F}$-conflations respectively, as in Lemma \ref{LemNEAtoNEx}. If $(\C,\Xcal_{\sfr})$ moreover satisfies condition {\rm (EI)}, then so does $(\C,\Xcal_{\sfr|_{\F}})$. By the arguments so far, this means that any relative theory for an $n$-exact category induces an $n$-exact category.
\end{rem}

\begin{rem}\label{RemRelativeNor}
Let $(\C,\Sig,\square)$ be an $(n+2)$-angulated category, and regard it as an $n$-exangulated category through Proposition \ref{PropAtoEA}. Then by Proposition \ref{PropClosed}, any closed subfunctor $\F\se\E_{\Sig}$ gives an $n$-exangulated category $(\C,\F,\sfr_{\square}|_{\F})$, which is not $n$-exact unless $\F=0$. Indeed, if $d_X^0$ is monomorphic in $\C$ for $\del\in\F(C,A)$ with $\sfr_{\square}(\del)=[X^{\mr}]$, then $d_X^0$ should be a split monomorphism, which implies $\del=0$.

It is not $(n+2)$-angulated either, in general. Especially for the closed subfunctor $\F=\E_{\Sig}^{\Ical}$ associated with a full subcategory $\Ical\se\C$ as in Definition \ref{DefItoF}, the resulting $n$-exangulated category is not $(n+2)$-angulated unless $\Ical=0$. In fact if $\Ical\ne0$, any object $0\ne I\in\Ical$ satisfies $\E_{\Sig}^{\Ical}(C,I)=0$ for any $C\in \C$, which cannot happen in an $(n+2)$-angulated category. Similarly for $(\E_{\Sig})_{\Ical}$.
\end{rem}

\section{$n$-cluster tilting subcategories as $n$-exangulated categories}\label{section_nCT}
Throughout this section, let $\CEs$ be a $1$-exangulated category, or equivalently, an extriangulated category.
Assume that it has enough projectives and injectives in the sense of \cite[Definition 1.12]{LN}, and let $\Pcal\se\C$ (respectively, $\Ical\se\C$) denote the full subcategory of projectives (resp. injectives). We denote the ideal quotients $\C/\Pcal$ and $\C/\Ical$ by $\und{\C}$ and $\ovl{\C}$, respectively. For a morphism $f$ in $\C$, its images in $\und{\C}$ and $\ovl{\C}$ are denoted by $\und{f}$ and $\ovl{f}$. We simply call a $1$-exangle in $\CEs$ an $\sfr$-{\it triangle}, as in \cite{LN}.

Remark that $\E\co\C\op\ti\C\to\Ab$ induces $\E\co\und{\C}\op\ti\ovl{\C}\to\Ab$, which we denote by the same symbol. From this reason, for any morphism $\ovl{a}\in\ovl{\C}(A,A\ppr)$, we also use the notation $\ovl{a}\sas$ for the homomorphisms $a\sas\co\E(C,A)\to\E(C,A\ppr)$. Similarly for $\und{a}\uas$.

\bigskip

In the rest, we assume the following.
\begin{assumption}\label{Assump}
Each object $A\in\C$ is assigned the following data {\rm (i)} and {\rm (ii)}.
\begin{itemize}
\item[{\rm (i)}] A pair $(\Sig A,\iota^A)$ of an object $\Sig A\in\C$ and an extension $\iota^A\in\E(\Sig A,A)$, for which $\sfr(\iota^A)=[A\to I\to\Sig A]$ satisfies $I\in\Ical$.
\item[{\rm (ii)}] A pair $(\Om A,\om_A)$ of an object $\Om A\in\C$ and an extension $\om_A\in\E(A,\Om A)$, for which $\sfr(\om_A)=[\Om A\to P\to A]$ satisfies $P\in\Pcal$.
\end{itemize}
\end{assumption}

\subsection{Higher extensions in an extriangulated category}

Under Assumption~\ref{Assump}, we have the following.

\begin{prop}\label{PropSigmaOmega}
Let $\CEs$ be as above. The following holds.
\begin{enumerate}
\item An additive functor $\Sig\co\ovl{\C}\to\ovl{\C}$ is given by the following.
\begin{itemize}
\item[-] For each $A\in\C$, associate $\Sig A\in\C$.
\item[-] For any $\ovl{a}\in\ovl{\C}(A,A\ppr)$, the morphism $\Sig\ovl{a}$ is defined by $\Sig\ovl{a}=\ovl{s}$, where $s\in\C(\Sig A,\Sig A\ppr)$ is a morphism satisfying $\ovl{a}\sas\iota^A=s\uas\iota^{A\ppr}$. Such $s$ always exists, and this $\Sig\ovl{a}$ does note depend on the choice of $s$.
\end{itemize}
\item Similarly, functor $\Om\co\und{\C}\to\und{\C}$ is given by the following.
\begin{itemize}
\item[-] For each $C\in\C$, associate $\Om C\in\C$.
\item[-] For any $\und{c}\in\und{\C}(C,C\ppr)$, the morphism $\Om\und{c}$ is defined by $\Om\und{c}=\und{t}$, where $t\in\C(\Om C,\Om C\ppr)$ is a morphism satisfying $\und{c}\uas\om_{C\ppr}=t\sas\om_C$. Similarly as in {\rm (1)}, this $\Omega\und{c}$ is uniquely determined by $\und{c}$.
\end{itemize}
\end{enumerate}
Up to natural isomorphisms, these functors are unique independently of the choice of data assigned in Assumption~\ref{Assump}.
\end{prop}
\begin{proof}
This is straightforward, shown in the same way as in \cite{Ha} or \cite{IY}.
\end{proof}

\begin{rem}\label{RemBar}
For any $A\in\C$, the $\sfr$-triangle $A\to I\to \Sig A\ov{\iota^A}{\dra}$ gives an exact sequence $\C(-,I)\to\C(-,\Sig A)\ov{\iota^A\ssh}{\lra}\E(-,A)\to0$.
Thus for any $\del\in\E(C,A)$, there exists $d\in\C(C,\Sig A)$ satisfying $d\uas\iota^A=\del$. Moreover $\ovl{d}\in\ovl{\C}(C,\Sig A)$ is unique, independently of a choice of such $d$.

If we denote this $\ovl{d}$ by $\ovl{\del}$, this gives a well-defined homomorphism
\begin{equation}\label{defbar}
\ovl{(\ )}\co\E(C,A)\to\ovl{\C}(C,\Sig A)\ ;\ \del\mapsto\ovl{\del}.
\end{equation}
\end{rem}

\begin{dfn}\label{DefHighExt}
Let $i\ge 1$ be any integer. Define a biadditive functor $\E^i\co\C\op\ti\C\to\Ab$ to be the composition of
\[ \C\op\ti\C\to\und{\C}\op\ti\ovl{\C}\overset{\Id\ti\Sig^{i-1}}{\lra}\und{\C}\op\ti\ovl{\C}\ov{\E}{\lra}\Ab, \]
where $\Sig^{i-1}$ is the $(i-1)$-times iteration of the endofunctor $\Sig$.

Dually, define $\Ep^{i}\co\C\op\ti\C\to\Ab$ to be the composition of
\[ \C\op\ti\C\to\und{\C}\op\ti\ovl{\C}\ov{\Om^{i-1}\ti\Id}{\lra}\und{\C}\op\ti\ovl{\C}\ov{\E}{\lra}\Ab. \]
By definition, for any $a\in\C(A,A\ppr)$ and $c\in\C(C,C\ppr)$ we have
\[ \E^i(c,a)=(\Sig^{i-1}\ovl{a})\sas c\uas\co\E(C\ppr,\Sig^{i-1}A)\to\E(C,\Sig^{i-1}A\ppr), \]
\[ \Ep^{i}(c,a)=(\Om^{i-1}\und{c})\uas a\sas\co\E(\Om^{i-1}C\ppr,A)\to\E(\Om^{i-1}C,A\ppr). \]
Also remark that $\E^1=\Ep^{1}=\E$ holds.
\end{dfn}

\begin{rem}
We will give a natural isomorphism $\Ep^{i}\ov{\cong}{\Longrightarrow}\E^i$ in the following (Definition~\ref{Defvp}), which ensures the validity of dual arguments for $\E^i$.
\end{rem}

\begin{dfn}\label{DefBar}
Let $i\ge 1$ be any integer. If we apply $(\ref{defbar})$ for $C$ and $\Sig^{i-1}A$, we obtain a homomorphism
\[ \ovl{(\ )}\co\E^i(C,A)\to\ovl{\C}(C,\Sig^i A)\ ;\ \del\mapsto \ovl{\del}. \]
Dually, a homomorphism
\[ \und{(\ )}\co\Ep^{i}(C,A)\to\und{\C}(\Om^iC,A) \]
is defined by $\und{\tau}=\und{e}$ for any $\tau\in\Ep^{i}(C,A)$, where $e\in\C(\Om^iC,A)$ is a morphism satisfying $e\sas(\om_{\Om^{i-1}C})=\tau$.
\end{dfn}

The above homomorphisms are natural on $\C\op\ti\C$, as follows.
\begin{prop}\label{PropFlat}
Let $i\ge 1$ be any integer.
For any $a\in\C(A,A\ppr)$ and $c\in\C(C,C\ppr)$,
\[
\xy
(-12,6)*+{\E^i(C\ppr,A)}="0";
(12,6)*+{\ovl{\C}(C\ppr,\Sig^i A)}="2";
(-12,-6)*+{\E^i(C,A\ppr)}="4";
(12,-6)*+{\ovl{\C}(C,\Sig^i A\ppr)}="6";
{\ar^(0.46){\ovl{(\ )}} "0";"2"};
{\ar_{\E^i(c,a)} "0";"4"};
{\ar^{(\Sig^i\ovl{a})\ci-\ci\ovl{c}} "2";"6"};
{\ar_(0.46){\ovl{(\ )}} "4";"6"};
{\ar@{}|\circlearrowright "0";"6"};
\endxy
\]
is commutative. Dually for $\und{(\ )}$.
\end{prop}
\begin{proof}
Replacing $\Sig^{i-1}A^{(\prime)}$ by $A^{(\prime)}$, we may assume $i=1$ from the beginning. Let $\del\in\E(C\ppr,A)$ be any element. By definition, $\Sig\ovl{a}=\ovl{s}$ and $\ovl{\del}=\ovl{d}$ are given by morphisms $s$ and $d$ satisfying
\[ a\sas\iota^A=s\uas\iota^{A\ppr}\quad\text{and}\quad d\uas\iota^A=\del \]
respectively. Thus we have $(s\ci d\ci c)\uas\iota^{A\ppr}=a\sas c\uas\del$, which means $\ovl{a\sas c\uas\del}=\ovl{s\ci d\ci c}=(\Sig\ovl{a})\ci\ovl{\del}\ci\ovl{c}$.
\end{proof}

\begin{rem}
Definition~\ref{DefBar} does not mean that {\bf any} representative $d$ of $\ovl{\delta}$ satisfies $d\uas\iota^A=\del$.
In fact, the surjective homomorphism $\iota^A\ssh\co\C(C,\Sig A)\to \E(C,A)$ only induces a well-defined homomorphism $\und{\C}(C,\Sig A)\to\E(C,A)$, but not $\ovl{\C}(C,\Sig A)\to\E(C,A)$ in general. Dually for $\und{(\ )}$.
\end{rem}
However, for a pair of objects of the form $(\Om C,\Sig A)$, the following holds.
\begin{prop}\label{PropfTFAE}
Let $A,C\in\C$ be any pair of objects. For any $f\in\C(\Om C,\Sig A)$, the following are equivalent.
\begin{enumerate}
\item $\ovl{f}=0$.
\item $\om_C\ush(f)=0$.
\item $\und{f}=0$.
\item $\iota^A\ssh(f)=0$.
\end{enumerate}
In particular, we have $\ovl{\C}(\Om C,\Sig A)=\und{\C}(\Om C,\Sig A)$. Consequently
\[ \ovl{(\ )}\co\E(\Om C,A)\ov{\cong}{\lra}\ovl{\C}(\Om C,\Sig A)\quad\text{and}\quad
\und{(\ )}\co\E(C,\Sig A)\ov{\cong}{\lra}\und{\C}(\Om C,\Sig A) \]
are isomorphisms, whose inverses are induced from $\iota^A\ssh$ and $\om_C\ush$ respectively.
\end{prop}
\begin{proof}
$(1)\Rightarrow(2)$ is obvious, since we have $\om_C\ush(f)=\ovl{f}\sas\om_C$.
By assumption, $\sfr(\om_C)=[\Om C\ov{p}{\lra}P\to C]$ satisfies $P\in\Pcal$. If $\om_C\ush(f)=0$ then $f$ factors through $P$, and thus $\und{f}=0$. This shows $(2)\Rightarrow (3)$. Dually for $(3)\Rightarrow (4)$, and $(4)\Rightarrow (1)$. Thus we have the equivalence of the four conditions.
The latter part is straightforward.
\end{proof}

\begin{dfn}\label{Defvp}
By Proposition~\ref{PropfTFAE}, we have a sequence of isomorphisms
\begin{equation}\label{SeqForvp}
\E(\Om C,A)\un{\cong}{\ov{\ovl{(\ )}}{\lra}}\ovl{\C}(\Om C,\Sig A)=\und{\C}(\Om C,\Sig A)\un{\cong}{\ov{\und{(\ )}}{\lla}}\E(C,\Sig A)
\end{equation}
which is natural in $A,C\in\C$ by Proposition~\ref{PropFlat}. Compose $(\ref{SeqForvp})$ to define
\[ \vt_{C,A}\co\E(\Om C,A)\ov{\cong}{\lra}\E(C,\Sig A), \]
and obtain a natural isomorphism $\vt\co\Ep^{2}\ov{\cong}{\Longrightarrow}\E^2$ of functors $\C\op\ti\C\to\Ab$.
More generally, for any integer $i\ge 1$, define the natural isomorphism
\[ \vt^{i-1}\co\Ep^{i}\ov{\cong}{\Longrightarrow}\E^i \]
by the $(i-1)$-times iteration of $\vt$. (For $i=1$, we put $\vt^0=\id_{\E}$.) Namely, $\vt_{C,A}^{i-1}$ is the composition of
\[ \E(\Om^{i-1}C,A)\ov{\vt_{\Om^{i-2}C,A}}{\lra}\E(\Om^{i-2}C,\Sig A)\ov{\vt_{\Om^{i-3}C,\Sig A}}{\lra}\cdots\ov{\vt_{C,\Sig^{i-2}A}}{\lra}\E(C,\Sig^{i-1}A) \]
for any $A,C\in\C$.
\end{dfn}

\begin{rem}\label{Remvp}
By definition, $\vt_{C,A}$ is a unique isomorphism which makes the following diagram commutative.
\begin{equation}\label{vpcomm}
\xy
(0,6)*+{\C(\Om C,\Sig A)}="0";
(0,-10)*+{}="1";
(-14,-8)*+{\E(\Om C,A)}="2";
(14,-8)*+{\E(C,\Sig A)}="4";
{\ar_{\iota^A\ssh} "0";"2"};
{\ar^{\om_C\ush} "0";"4"};
{\ar^{\cong}_{\vt_{C,A}} "2";"4"};
{\ar@{}|\circlearrowright "0";"1"};
\endxy
\end{equation}
\end{rem}

\subsection{The cup product and long exact sequences}

\begin{dfn}\label{DefCupProd}
Let $i,j\ge1$ be integers. For any $A,C,X\in\C$, we define the {\it cup product} maps as follows.
\begin{enumerate}
\item Define $\cup\co\E^i(X,A)\ti\E^j(C,X)\to\E^{i+j}(C,A)$ by
\[ \del\cup\thh=\E^j(C,\ovl{\del})(\thh) \]
for any $(\del,\thh)\in\E^i(X,A)\ti\E^j(C,X)$. Here $\ovl{\del}\in\ovl{\C}(X,\Sig^iA)$ is the morphism defined in Definition~\ref{DefBar}.
\item Dually, define $\cupp\co\Ep^{i}(X,A)\ti\Ep^{j}(C,X)\to\Ep^{i+j}(C,A)$ by
\[ \del\cupp\thh=\Ep^{i}(\und{\thh},A)(\del) \]
for any $(\del,\thh)\in\Ep^{i}(X,A)\ti\Ep^{j}(C,X)$.
\end{enumerate}
\end{dfn}

\begin{rem}
In the case of a triangulated category $(\C,[1],\triangle)$, the above {\rm (1)} is just saying $\del\cup\thh=\del[j]\ci\thh$ for $(\del,\thh)\in\C(X,A[i])\ti\C(C,X[j])$.
\end{rem}

\begin{rem}\label{RemCupProd}
Let $\del\in\E(C,A)$ be any extension.
By the definition of the cup product and the homomorphisms in Definition~\ref{DefBar}, the following diagrams become commutative for any $X\in\C$.
\[
\xy
(-12,6)*+{\C(A,\Sig X)}="0";
(12,6)*+{\E(C,\Sig X)}="2";
(-12,-6)*+{\E(A,X)}="4";
(12,-6)*+{\E^2(C,X)}="6";
{\ar^{\del\ush} "0";"2"};
{\ar_{\iota^X\ssh} "0";"4"};
{\ar@{=} "2";"6"};
{\ar_(0.44){-\cup\del} "4";"6"};
{\ar@{}|\circlearrowright "0";"6"};
\endxy
\quad,\quad
\xy
(-12,6)*+{\C(\Om X,C)}="0";
(12,6)*+{\E(\Om X,A)}="2";
(-12,-6)*+{\E(X,C)}="4";
(12,-6)*+{\Ep^{2}(X,A)}="6";
{\ar^{\del\ssh} "0";"2"};
{\ar_{\om_X\ush} "0";"4"};
{\ar@{=} "2";"6"};
{\ar_(0.48){\del\cupp-} "4";"6"};
{\ar@{}|\circlearrowright "0";"6"};
\endxy
\]

\end{rem}

The following proposition ensures dual arguments for the cup product.
\begin{prop}\label{PropCupProdDual}
For any $i,j\ge1$ and any $A,C,X\in\C$,
\begin{equation}\label{DiagMCPtoshow}
\xy
(-18,7)*+{\Ep^{i}(X,A)\ti\Ep^{j}(C,X)}="0";
(18,7)*+{\Ep^{i+j}(C,A)}="2";
(-18,-7)*+{\E^i(X,A)\ti\E^j(C,X)}="4";
(18,-7)*+{\E^{i+j}(C,A)}="6";
{\ar^(0.6){\cupp} "0";"2"};
{\ar_{\vt^{i-1}\ti\vt^{j-1}}^{\cong} "0";"4"};
{\ar^{\vt^{i+j-1}}_{\cong} "2";"6"};
{\ar_(0.6){\cup} "4";"6"};
{\ar@{}|\circlearrowright "0";"6"};
\endxy
\end{equation}
is commutative.
\end{prop}
\begin{proof}
If $i=j=1$, it suffices to show the commutativity of the following.
\begin{equation}\label{CommEEECC}
\xy
(-28,6)*+{\E(X,A)\ti\E(C,X)}="0";
(3,6)*+{\E(\Om C,A)}="2";
(29,6)*+{\ovl{\C}(\Om C,\Sig A)}="4";
(-28,-6)*+{\E(X,A)\ti\E(C,X)}="10";
(3,-6)*+{\E(C,\Sig A)}="12";
(29,-6)*+{\und{\C}(\Om C,\Sig A)}="14";
{\ar^(0.6){\cupp} "0";"2"};
{\ar^(0.44){\ovl{(\ )}}_(0.44){\cong} "2";"4"};
{\ar@{=} "0";"10"};
{\ar@{=} "4";"14"};
{\ar_(0.6){\cup} "10";"12"};
{\ar_(0.44){\und{(\ )}}^(0.44){\cong} "12";"14"};
{\ar@{}|\circlearrowright "0";"14"};
\endxy
\end{equation}
For any element $(\del,\thh)\in\E(X,A)\ti\E(C,X)$, take $d\in\C(X,\Sig A), e\in\C(\Om C,X)$ satisfying $\ovl{d}=\ovl{\del}$ and $\und{e}=\und{\thh}$. Then Proposition~\ref{PropFlat} shows $\ovl{\del\cupp\thh}=\ovl{e\uas\del}=\ovl{\del}\ci\ovl{e}=\ovl{d\ci e}$. Dually we have $\und{\del\cup\thh}=\und{d\sas\thh}=\und{d}\ci\und{\thh}=\und{d\ci e}$. As in Proposition~\ref{PropfTFAE}, this shows the commutativity of $(\ref{CommEEECC})$.

Let $i,j\ge 1$ be arbitrary pair of integers. By the above argument applied to $\Sig^{i-1}A$ and $\Om^{j-1}C$,
\begin{equation}\label{DiagMCPcenter}
\xy
(-36,7)*+{\E(X,\Sig^{i-1}A)\ti\E(\Om^{j-1}C,X)}="0";
(6,7)*+{\E(\Om^jC,\Sig^{i-1}A)}="2";
(36,7)*+{\Ep^{2}(\Om^{j-1}C,\Sig^{i-1}A)}="4";
(-36,-7)*+{\E(X,\Sig^{i-1}A)\ti\E(\Om^{j-1}C,X)}="10";
(6,-7)*+{\E(\Om^{j-1}C,\Sig^iA)}="12";
(36,-7)*+{\E^2(\Om^{j-1}C,\Sig^{i-1}A)}="14";
{\ar^(0.62){\cupp} "0";"2"};
{\ar@{=} "2";"4"};
{\ar@{=} "0";"10"};
{\ar^{\vt_{\Om^{j-1}C,\Sig^{i-1}A}}_{\cong} "4";"14"};
{\ar_(0.62){\cup} "10";"12"};
{\ar@{=} "12";"14"};
{\ar@{}|\circlearrowright "0";"14"};
\endxy
\end{equation}
is commutative.
Remark that the naturality of $\vt^{j-1}$ shows the commutativity of
\[
\xy
(-16,6)*+{\Ep^{j}(C,X)}="0";
(16,6)*+{\Ep^{j}(C,\Sig^iA)}="2";
(-16,-6)*+{\E^j(C,X)}="4";
(16,-6)*+{\E^j(C,\Sig^iA)}="6";
{\ar^{\Ep^{j}(C,d)} "0";"2"};
{\ar_{\vt^{j-1}_{C,X}}^{\cong} "0";"4"};
{\ar^{\vt^{j-1}_{C,\Sig^iA}}_{\cong} "2";"6"};
{\ar_{\E^j(C,d)} "4";"6"};
{\ar@{}|\circlearrowright "0";"6"};
\endxy
\]
for any $d\in\C(X,\Sig^iA)$. Applying this to $\ovl{d}=\ovl{\del}$ for each $\del\in\E^i(X,A)$, we obtain the commutativity of the following.
\begin{equation}\label{DiagMCP}
\xy
(-22,6)*+{\E(X,\Sig^{i-1}A)\ti\E(\Om^{j-1}C,X)}="0";
(22,6)*+{\E(\Om^{j-1}C,\Sig^iA)}="2";
(-22,-6)*+{\E^i(X,A)\ti\E^j(C,X)}="4";
(22,-6)*+{\E^{i+j}(C,A)}="6";
{\ar^(0.62){\cup} "0";"2"};
{\ar_{\id\ti\vt^{j-1}_{C,X}}^{\cong} "0";"4"};
{\ar^{\vt^{j-1}_{C,\Sig^iA}}_{\cong} "2";"6"};
{\ar_(0.56){\cup} "4";"6"};
{\ar@{}|\circlearrowright "0";"6"};
\endxy
\end{equation}
Dually, we can show that the following diagram is commutative.
\begin{equation}\label{DiagMCPdual}
\xy
(-22,6)*+{\Ep^{i}(X,A)\ti\Ep^{j}(C,X)}="0";
(22,6)*+{\Ep^{i+j}(C,A)}="2";
(-22,-6)*+{\E(X,\Sig^{i-1}A)\ti\E(\Om^{j-1}C,X)}="4";
(22,-6)*+{\E(\Om^jC,\Sig^{i-1}A)}="6";
{\ar^(0.56){\cupp} "0";"2"};
{\ar_{\vt^{i-1}_{X,A}\ti\id}^{\cong} "0";"4"};
{\ar^{\vt^{i-1}_{\Om^jC,A}}_{\cong} "2";"6"};
{\ar_(0.62){\cupp} "4";"6"};
{\ar@{}|\circlearrowright "0";"6"};
\endxy
\end{equation}
Pasting $(\ref{DiagMCPcenter}),(\ref{DiagMCP}),(\ref{DiagMCPdual})$ together, we obtain $(\ref{DiagMCPtoshow})$.
\end{proof}

By the above reason, we mainly use $\E$ and $\cup$ in the rest, but seldom $\Ep$ and $\cupp$ unless needed. Cup product is natural, in the following sense.
\begin{prop}\label{PropCupProdNat}
Let $i,j\ge 1$ be integers. For any $a\in\C(A,A\ppr), c\in\C(C,C\ppr)$ and $x\in\C(X,Y)$, the following diagrams are commutative.
\begin{enumerate}
\item $\xy
(-18,7)*+{\E^i(X,A)\ti\E^j(C\ppr,X)}="0";
(18,7)*+{\E^{i+j}(C\ppr,A)}="2";
(-18,-7)*+{\E^i(X,A\ppr)\ti\E^j(C,A\ppr)}="4";
(18,-7)*+{\E^{i+j}(C,A)}="6";
{\ar^(0.6){\cup} "0";"2"};
{\ar_{\E^i(X,a)\ti\E^j(c,X)} "0";"4"};
{\ar^{\E^{i+j}(c,a)} "2";"6"};
{\ar_(0.6){\cup} "4";"6"};
{\ar@{}|\circlearrowright "0";"6"};
\endxy
$
\item $\xy
(-23,7)*+{\E^i(Y,A)\ti\E^j(C,X)}="0";
(23,7)*+{\E^i(X,A)\ti\E^j(C,X)}="2";
(-23,-7)*+{\E^i(Y,A)\ti\E^j(C,Y)}="4";
(23,-7)*+{\E^{i+j}(C,A)}="6";
{\ar^{\E^i(x,A)\ti\id} "0";"2"};
{\ar_{\id\ti\E^j(C,x)} "0";"4"};
{\ar^{\cup} "2";"6"};
{\ar_(0.56){\cup} "4";"6"};
{\ar@{}|\circlearrowright "0";"6"};
\endxy
$
\end{enumerate}
\end{prop}
\begin{proof}
{\rm (1)} For any $(\del,\thh)\in\E^i(X,A)\ti\E^j(C\ppr,X)$, we have $\ovl{\E^i(X,a)(\del)}=(\Sig^i\ovl{a})\ci\ovl{\del}$ by Proposition~\ref{PropFlat}, and thus
\begin{eqnarray*}
\big(\E^i(X,a)(\del)\big)\cup\big(\E^j(c,X)(\thh)\big)
&=&\E^j(C,\ovl{\E^i(X,a)(\del)})\ci\E^j(c,X)(\thh)\\
&=&\E^j(c,(\Sig^i\ovl{a})\ci\ovl{\del})(\thh)\ =\ \E^{i+j}(c,a)\ci\E^j(C\ppr,\ovl{\del})(\thh)\\
&=&\E^{i+j}(c,a)(\del\cup\thh).
\end{eqnarray*}

{\rm (2)} For any $(\rho,\thh)\in\E^i(Y,A)\ti\E^j(C,X)$, we have $\ovl{\E^i(x,A)(\rho)}=\ovl{\rho}\ci\ovl{x}$ by Proposition~\ref{PropFlat}, and thus
\[
\big(\E^i(x,A)(\rho)\big)\cup\thh=\E^j(C,\ovl{\rho}\ci\ovl{x})(\thh)=\E^j(C,\ovl{\rho})\ci\E^j(C,\ovl{x})(\thh)=\rho\cup\big(\E^j(C,x)(\thh)\big). \]
\end{proof}

\begin{prop}\label{PropCup}
The cup product is associative. Namely for $i,j,k\ge1$,
\[ (\del\cup\thh)\cup\tau=\del\cup(\thh\cup\tau) \]
holds for any $\del\in\E^i(Y,A), \thh\in\E^j(X,Y)$ and $\tau\in\E^k(C,X)$.
\end{prop}
\begin{proof}
By Proposition~\ref{PropFlat}, we have $\ovl{\E^j(X,\ovl{\del})(\thh)}=(\Sig^j\ovl{\del})\ci\ovl{\thh}$. Thus it follows
\begin{eqnarray*}
(\del\cup\thh)\cup\tau&=&\E^k(C,\ovl{\del\cup\thh})(\tau)\ =\ \E^k(C,\ovl{\E^j(X,\ovl{\del})(\thh)})(\tau)\\
&=&\E^k(C,(\Sig^j\ovl{\del})\ci\ovl{\thh})(\tau)\ =\ \E^k(C,\Sig^j\ovl{\del})\ci\E^k(C,\ovl{\thh})(\tau)\\
&=&\E^{j+k}(C,\ovl{\del})\big(\E^k(C,\ovl{\thh})(\tau)\big)\ =\ \E^{j+k}(C,\ovl{\del})(\thh\cup\tau)\ =\ \del\cup(\thh\cup\tau).
\end{eqnarray*}
\end{proof}

\begin{prop}\label{PropLongExact}
For any $\sfr$-triangle $A\ov{x}{\lra}B\ov{y}{\lra}C\ov{\del}{\dra}$, we have the following long exact sequences.
\begin{enumerate}
\item $\C(C,-)\ov{\C(y,-)}{\Longrightarrow}\C(B,-)\ov{\C(x,-)}{\Longrightarrow}\C(A,-)\ov{\del\ush}{\Longrightarrow}\E(C,-)\ov{\E(y,-)}{\Longrightarrow}\cdots\\\hspace{-0.7cm}\cdots\ov{-\cup\del}{\Longrightarrow}\E^i(C,-)\ov{\E^i(y,-)}{\Longrightarrow}\E^i(B,-)\ov{\E^i(x,-)}{\Longrightarrow}\E^i(A,-)\ov{-\cup\del}{\Longrightarrow}\E^{i+1}(C,-)\ov{\E^{i+1}(y,-)}{\Longrightarrow}\cdots$.
\medskip
\item $\C(-,A)\ov{\C(-,x)}{\Longrightarrow}\C(-,B)\ov{\C(-,y)}{\Longrightarrow}\C(-,C)\ov{\del\ssh}{\Longrightarrow}\E(-,A)\ov{\E(-,x)}{\Longrightarrow}\cdots\\\hspace{-0.7cm}\cdots\ov{\del\cup-}{\Longrightarrow}\E^i(-,A)\ov{\E^i(-,x)}{\Longrightarrow}\E^i(-,B)\ov{\E^i(-,y)}{\Longrightarrow}\E^i(-,C)\ov{\del\cup-}{\Longrightarrow}\E^{i+1}(-,A)\ov{\E^{i+1}(-,x)}{\Longrightarrow}\cdots$.
\end{enumerate}
\end{prop}
\begin{proof}
This is essentially shown in \cite[Proposition 5.2]{LN}. We briefly review its proof, to make it clear that the connecting homomorphism is indeed given by the cup product.

{\rm (1)} Remark that we already know that
\begin{equation}\label{ExSeqCCC}
\C(C,-)\ov{\C(y,-)}{\Longrightarrow}\C(B,-)\ov{\C(x,-)}{\Longrightarrow}\C(A,-)\ov{\del\ush}{\Longrightarrow}\E(C,-)\ov{\E(y,-)}{\Longrightarrow}\E(B,-)\ov{\E(x,-)}{\Longrightarrow}\E(A,-)
\end{equation}
is exact (\cite[Corollary 3.12]{NP}). Let $X\in\C$ be any object, and take $\sfr(\iota^X)=[X\to I\to\Sig X]$ with $I\in\Ical$. Applying the above sequence to $\Sig X$, we see that
\[ \C(B,\Sig X)\ov{\C(x,\Sig X)}{\lra}\C(A,\Sig X)\ov{\del\ush}{\lra}\E(C,\Sig X)\ov{\E(y,\Sig X)}{\lra}\E(B,\Sig X) \]
is exact. By Remark~\ref{RemCupProd}, we have a commutative diagram
\[
\xy
(-40,14)*+{\C(B,I)}="-10";
(-12,14)*+{\C(A,I)}="-12";
(6,14)*+{0}="-14";
(-40,0)*+{\C(B,\Sig X)}="0";
(-12,0)*+{\C(A,\Sig X)}="2";
(12,0)*+{\E(C,\Sig X)}="4";
(40,0)*+{\E(B,\Sig X)}="6";
(-40,-14)*+{\E(B,X)}="10";
(-12,-14)*+{\E(A,X)}="12";
(12,-14)*+{\E^2(C,X)}="14";
(40,-14)*+{\E^2(B,X)}="16";
(-40,-25)*+{0}="20";
(-12,-25)*+{0}="22";
{\ar^{\C(x,I)} "-10";"-12"};
{\ar^{} "-12";"-14"};
{\ar_{} "-10";"0"};
{\ar_{} "-12";"2"};
{\ar^{\C(x,\Sig X)} "0";"2"};
{\ar^{\del\ush} "2";"4"};
{\ar^{\E(y,\Sig X)} "4";"6"};
{\ar_{\iota^X\ssh} "0";"10"};
{\ar_{\iota^X\ssh} "2";"12"};
{\ar@{=} "4";"14"};
{\ar@{=} "6";"16"};
{\ar_{\E(x,X)} "10";"12"};
{\ar_{-\cup\del} "12";"14"};
{\ar_{\E^2(y,X)} "14";"16"};
{\ar_{} "10";"20"};
{\ar_{} "12";"22"};
{\ar@{}|\circlearrowright "-10";"2"};
{\ar@{}|\circlearrowright "0";"12"};
{\ar@{}|\circlearrowright "2";"14"};
{\ar@{}|\circlearrowright "4";"16"};
\endxy
\]
in which the upper two rows and the left two columns are exact. Thus the third row also becomes exact at $\E(A,X)$ and $\E^2(C,X)$. Together with $(\ref{ExSeqCCC})$, it follows that
\[ \E(C,X)\ov{\E(y,X)}{\Longrightarrow}\E(B,X)\ov{\E(x,X)}{\Longrightarrow}\E(A,X)\ov{-\cup\del}{\Longrightarrow}\E^2(C,X)\ov{\E^2(y,X)}{\Longrightarrow}\E^2(B,X) \]
is exact. Replacing $X$ by $\Sig^{i-1}X$, we obtain the exactness of
\[ \E^i(C,X)\ov{\E^i(y,X)}{\Longrightarrow}\E^i(B,X)\ov{\E^i(x,X)}{\Longrightarrow}\E^i(A,X)\ov{-\cup\del}{\Longrightarrow}\E^{i+1}(C,X)\!\ov{\E^{i+1}(y,X)}{\Longrightarrow}\!\E^{i+1}(B,X) \]
for any $i\ge 1$ and $X\in\C$.

{\rm (2)} Similarly as in {\rm (1)}, we already know the exactness of
\[ \C(-,A)\ov{\C(-,x)}{\Longrightarrow}\C(-,B)\ov{\C(-,y)}{\Longrightarrow}\C(-,C)\ov{\del\ssh}{\Longrightarrow}\E(-,A)\ov{\E(-,x)}{\Longrightarrow}\E(-,B)\ov{\E(-,y)}{\Longrightarrow}\E(-,C). \]
Thus it is enough to show the exactness of
\[ \E^i(X,A)\ov{\E^i(X,x)}{\Longrightarrow}\E^i(X,B)\ov{\E^i(X,y)}{\Longrightarrow}\E^i(X,C)\ov{\del\cup-}{\Longrightarrow}\E^{i+1}(X,A)\!\ov{\E^{i+1}(X,x)}{\Longrightarrow}\!\E^{i+1}(X,B) \]
for any $i\ge 1$ and $X\in\C$. By the naturality of $\vt$ and by Proposition~\ref{PropCupProdDual}, this is isomorphic to
\[ \Ep^i(X,A)\ov{\Ep^i(X,x)}{\Longrightarrow}\Ep^i(X,B)\ov{\Ep^i(X,y)}{\Longrightarrow}\Ep^i(X,C)\ov{\del\cupp-}{\Longrightarrow}\Ep^{i+1}(X,A)\!\ov{\Ep^{i+1}(X,x)}{\Longrightarrow}\!\Ep^{i+1}(X,B), \]
whose exactness can be shown dually to {\rm (1)}.
\end{proof}

\subsection{From $n$-cluster tilting subcategories to $n$-exangulated categories}

Let $n$ be a positive integer, as before.
\begin{dfn}\label{DefnCluster}
Let $\Tcal\se\C$ be a full additive subcategory closed by isomorphisms and direct summands. Such $\Tcal$ is called an {\it $n$-cluster tilting subcategory} of $\C$, if it satisfies the following conditions.
\begin{enumerate}
\item $\Tcal\se\C$ is functorially finite.
\item For any $C\in\C$, the following are equivalent.
\begin{itemize}
\item[{\rm (i)}] $C\in\Tcal$.
\item[{\rm (ii)}] $\E^i(C,\Tcal)=0$ for any $1\le i\le n-1$.
\item[{\rm (iii)}] $\E^i(\Tcal,C)=0$ for any $1\le i\le n-1$.
\end{itemize}
\end{enumerate}
For $n=1$, this is just saying $\Tcal=\C$.
\end{dfn}

In the rest, we mainly deal with the case $n\ge 2$.
\begin{rem}\label{RemnCluster}
For an $n$-cluster tilting subcategory $\Tcal\se\C$, the following holds.
\begin{enumerate}
\item $\Ical,\Pcal\se\Tcal$.
\item For any $C\in\C$, there is a right $\Tcal$-approximation $g_C\co T_C\to C$, which is at the same time an $\sfr$-deflation. Indeed, for any right $\Tcal$-approximation $g\co T\to C$ and an $\sfr$-deflation $p\co P\to C$ satisfying $P\in \Pcal$, the morphism $[g\ p]\co T\oplus P\to C$ has the desired property. (See \cite[Corollary 3.16]{NP}.)
\item Dually, any $A\in\C$ admits a left $\Tcal$-approximation which is also an $\sfr$-inflation.
\end{enumerate}
\end{rem}

Consider the following conditions (Conditions~\ref{CondnCluster} and \ref{CondWIC}).
\begin{cond}\label{CondnCluster}
Let $\Tcal\se\C$ be an $n$-cluster tilting subcategory.
\begin{itemize}
\item[{\rm (c1)}] If $C\in\C$ satisfies $\E^{n-1}(\Tcal,C)=0$, then there is $\sfr$-triangle
\begin{equation}\label{OPC}
D\ov{q}{\lra}P\to C\dra\quad(P\in\Pcal)
\end{equation}
(thus gives $D\cong\Om C$ in $\und{\C}$) for which,
\[ \C(T,q)\co\C(T,D)\to\C(T,P) \]
is injective for any $T\in\Tcal$.
\item[{\rm (c2)}]Dually, if $A\in\C$ satisfies $\E^{n-1}(A,\Tcal)=0$, then there is $\sfr$-triangle
\[ A\to I\ov{j}{\lra}S\dra\quad(I\in\Ical) \]
for which
\[ \C(j,T)\co\C(S,T)\to\C(I,T) \]
is injective for any $T\in\Tcal$.
\end{itemize}
\end{cond}

\begin{ex}\label{Ex_CondnCluster}
$\ \ $
\begin{enumerate}
\item If $\C$ is an exact category with enough projectives and injectives, then Condition~\ref{CondnCluster} is trivially satisfied by any $n$-cluster tilting subcategory $\Tcal\se\C$.
\item If $(\C,[1],\triangle)$ is a triangulated category and if $\Tcal\se\C$ is an $n$-cluster tilting subcategory satisfying $\Tcal [n]=\Tcal$, then Condition~\ref{CondnCluster} is satisfied. Indeed, {\rm (c1),(c2)} are only saying $\Tcal[n-1]^{\perp}\se\Tcal[-1]^{\perp}$ and ${}^{\perp}\Tcal[n-1]\se {}^{\perp}\Tcal[-1]$.
\end{enumerate}
\end{ex}

\begin{cond}\label{CondWIC}(\cite[Condition 5.8]{NP})
Let $\CEs$ be an extriangulated category as before. Consider the following condition.
\begin{enumerate}
\item Let $f\in\mathscr{C}(A,B),\, g\in\mathscr{C}(B,C)$ be any pair of morphisms. If $g\circ f$ is an $\sfr$-inflation, then so is $f$.
\item Let $f\in\mathscr{C}(A,B),\, g\in\mathscr{C}(B,C)$ be any pair of morphisms. If $g\circ f$ is a $\sfr$-deflation, then so is $g$.
\end{enumerate}
\end{cond}

\begin{ex}\label{Ex_CondWIC}
$\ \ $
\begin{enumerate}
\item If $\C$ is an exact category, then Condition~\ref{CondWIC} is equivalent to that $\C$ is weakly idempotent complete (\cite[Proposition 7.6]{B}).
\item If $(\C,[1],\triangle)$ is a triangulated category, then Condition~\ref{CondWIC} is always satisfied.
\end{enumerate}
\end{ex}

Let $\Tcal\se\C$ be an $n$-cluster tilting subcategory, in the rest.
Assuming Conditions \ref{CondnCluster} and \ref{CondWIC}, let us construct $\sfr^n$ which makes $(\Tcal,\E^n,\sfr^n)$ into an $n$-exangulated category. This will give a unification of \cite[Theorem 4.14]{J} and \cite[Theorem 1]{GKO}. In the rest, for $A,C\in\Tcal$ we abbreviately denote $\Cbf_{(\Tcal;A,C)}^{n+2}$ in Definition~\ref{DefSeq} by $\TAC$, to distinguish it from $\Cbf_{(\C;A,C)}^{n+2}$. Namely, $\TAC$ is the full subcategory of $\Cbf_{(\C;A,C)}^{n+2}$, consisting of $X^{\mr}$ satisfying $X^i\in\Tcal$ for any $1\le i\le n$.

\begin{prop}\label{PropABCvanish}
Let $\Tcal\se\C$ be any $n$-cluster tilting subcategory satisfying Condition~\ref{CondnCluster}. Let $C\in\C$ be an object satisfying $\E^{n-1}(\Tcal,C)=0$. Then for any $\sfr$-triangle $A\ov{x}{\lra}B\ov{y}{\lra}C\ov{\del}{\dra}$,
\[ 0\to\C(T,A)\ov{\C(T,x)}{\lra}\C(T,B)\ov{\C(T,y)}{\lra}\C(T,C)\ov{\del\ssh}{\lra}\E(T,A) \]
is exact for any $T\in\Tcal$.
Dually for an object $A\in\C$ satisfying $\E^{n-1}(A,\Tcal)=0$.
\end{prop}
\begin{proof}
It suffices to show that $\C(T,x)$ is injective for any $T\in\Tcal$.
Let $(\ref{OPC})$ be $\sfr$-triangle which satisfies the requirement of {\rm (c1)}. By \cite[Proposition 3.15]{NP}, we obtain the following commutative diagram made of conflations in $\CEs$.
\[
\xy
(-7,21)*+{D}="-12";
(7,21)*+{D}="-14";
(-21,7)*+{A}="0";
(-7,7)*+{N}="2";
(7,7)*+{P}="4";
(-21,-7)*+{A}="10";
(-7,-7)*+{B}="12";
(7,-7)*+{C}="14";
{\ar@{=} "-12";"-14"};
{\ar_{} "-12";"2"};
{\ar^{} "-14";"4"};
{\ar^{a} "0";"2"};
{\ar^{b} "2";"4"};
{\ar@{=} "0";"10"};
{\ar_{c} "2";"12"};
{\ar^{} "4";"14"};
{\ar_{x} "10";"12"};
{\ar_{y} "12";"14"};
{\ar@{}|\circlearrowright "-12";"4"};
{\ar@{}|\circlearrowright "0";"12"};
{\ar@{}|\circlearrowright "2";"14"};
\endxy
\]
Since $P\in\Pcal$, the middle row splits. Applying snake lemma to the following diagram,
\[
\xy
(-42,7)*+{0}="0";
(-24,7)*+{\Ccal(T,A)}="2";
(0,7)*+{\Ccal(T,N)}="4";
(24,7)*+{\Ccal(T,P)}="6";
(42,7)*+{0}="8";
(-42,-7)*+{0}="10";
(-24,-7)*+{\Ccal(T,B)}="12";
(0,-7)*+{\Ccal(T,B)}="14";
(24,-7)*+{0}="16";
{\ar^{} "0";"2"};
{\ar^{\C(T,a)} "2";"4"};
{\ar^{\C(T,b)} "4";"6"};
{\ar^{} "6";"8"};
{\ar_{\C(T,x)} "2";"12"};
{\ar^{\C(T,c)} "4";"14"};
{\ar^{} "6";"16"};
{\ar_{} "10";"12"};
{\ar@{=} "12";"14"};
{\ar_{} "14";"16"};
{\ar@{}|\circlearrowright "2";"14"};
{\ar@{}|\circlearrowright "4";"16"};
\endxy
\]
we see that $\C(T,x)$ is injective.
\end{proof}

\begin{cor}
Condition {\rm (c1)} does not depend on the choice of $(\ref{OPC})$. Similarly for {\rm (c2)}.
\end{cor}
\begin{proof}
This immediately follows from Proposition~\ref{PropABCvanish} and its dual.
\end{proof}

\begin{dfn}\label{DefsDec}
Let $\Tcal\se\C$ be an $n$-cluster tilting subcategory. Let $\del\in\E^n(C,A)$ be any element, with $A,C\in\Tcal$. We say that an object $\Xd\in\AE^{n+2}_{(\Tcal,\E^n)}$ (see Definition~\ref{DefAE})
\begin{equation}\label{sdec1}
A\ov{d_X^0}{\lra}X^1\ov{d_X^1}{\lra}\cdots\ov{d_X^{n-1}}{\lra}X^n\ov{d_X^n}{\lra}C\ov{\del}{\dra}\quad(X^0=A,X^{n+1}=C,X^i\in\Tcal)
\end{equation}
is {\it $\sfr$-decomposable} if it admits $\sfr$-triangles
\begin{equation}\label{sdec2}
\begin{cases}
\ A\ov{d_X^0}{\lra}X^1\ov{e^1}{\lra}M^1\ov{\del_{(1)}}{\dra},&\\
\ M^i\ov{m^i}{\lra}X^{i+1}\ov{e^{i+1}}{\lra}M^{i+1}\ov{\del_{(i+1)}}{\dra}&(i=1,\ldots,n-2),\\
\ M^{n-1}\ov{m^{n-1}}{\lra}X^n\ov{d_X^n}{\lra}C\ov{\del_{(n)}}{\dra}
\end{cases}
\end{equation}
satisfying $d_X^i=m^i\ci e^i$ for $1\le i\le n-1$ and $\del=\del_{(1)}\cup\cdots\cup\del_{(n)}$. We call $(\ref{sdec2})$ an {\it $\sfr$-decomposition} of $(\ref{sdec1})$.
Remark that this definition is self-dual, by virtue of Definition~\ref{Defvp} and Proposition~\ref{PropCupProdDual}.
\end{dfn}

\begin{prop}\label{PropsDecAdded}
Let $\Tcal\se\C$ be any $n$-cluster tilting subcategory. Let $A,B\in\Tcal$ be any pair of objects. For any $f\in\Tcal(A,B)$, the following are equivalent.
\begin{enumerate}
\item $f$ is an $\sfr$-inflation.
\item There is some $\sfr$-decomposable object $\Xd\in\AE^{n+2}_{(\Tcal,\E^n)}$ which is of the form
$A\ov{f}{\lra}B\ov{d_X^1}{\lra}X^2\ov{d_X^2}{\lra}\cdots\ov{d_X^n}{\lra}X^{n+1}\ov{\del}{\dra}$.
\end{enumerate}
Dually for $\sfr$-deflations.
\end{prop}
\begin{proof}
Suppose that $f$ is an $\sfr$-inflation. By definition there is some $\sfr$-triangle $A\ov{f}{\lra}B\ov{e^1}{\lra}M^1\ov{\del_{(1)}}{\dra}$. As in Remark~\ref{RemnCluster} {\rm (2)}, inductively we can take $\sfr$-triangles
\[ M^{i-1}\ov{m^{i-1}}{\lra}X^i\ov{e^i}{\lra}M^i\ov{\del_{(i)}}{\dra} \]
for $2\le i\le n$, in which $X^i$ belongs to $\Tcal$ and $m^{i-1}$ is a left $\Tcal$-approximation. Then for any $2\le i\le n$, we have $\E(M^i,\Tcal)=0$ and
\[ \E^k(M^i,T)\cong\E^{k-1}(M^{i-1},T)\quad(2\le k\le n-1) \]
for any $T\in\Tcal$. This shows
\[ \E^k(M^n,T)\cong\E^{k-1}(M^{n-1},T)\cong\cdots\cong\E(M^{n-k+1},T)=0 \]
for any $T\in\Tcal$ and any $1\le k\le n-1$, which means $M^n\in\Tcal$. Since $\del=\del_{(1)}\cup\cdots\cup\del_{(n)}\in\E^n(M^n,A)$ satisfies
\begin{eqnarray*}
&\E^n(M^n,f)(\del)=(f\sas\del_{(1)})\cup\cdots\cup\del_{(n)}=0,&\\
&\E^n(e^n,A)(\del)=\del_{(1)}\cup\cdots\cup((e^n)\uas\del_{(n)})=0&
\end{eqnarray*}
by Proposition~\ref{PropCupProdNat}, the sequence
\[ A\ov{f}{\lra}B\ov{m^1\ci e^1}{\lra}X^2\ov{m^2\ci e^2}{\lra}\cdots\ov{m^{n-1}\ci e^{n-1}}{\lra}X^n\ov{e^n}{\lra}M^n\ov{\del}{\dra} \]
gives an $\sfr$-decomposable object in $\AE^{n+2}_{(\Tcal,\E^n)}$.
The converse is trivial.
\end{proof}

\begin{lem}\label{LemsDecExist}
Let $\Tcal\se\C$ be any $n$-cluster tilting subcategory, with $n\ge 2$. Let $\ell\in\{2,\ldots,n\}$ be any integer. Let $A,M\in\C$ be arbitrary pair of objects, and let $\del\in\E^{\ell}(M,A)=\E(M,\Sig^{\ell-1}A)$ be any element with an $\sfr$-triangle
\[ \Sig^{\ell-1}A\ov{a}{\lra}B\ov{b}{\lra}M\ov{\del}{\dra}. \]
Then there exist $\sfr$-triangles
\[ \Sig^{\ell-2}A\ov{a\ppr}{\lra}B\ppr\ov{b\ppr}{\lra}M\ppr\ov{\del\ppr}{\dra}%
\quad\text{and}\quad%
M\ppr\ov{m}{\lra}X\ov{e}{\lra}M\ov{\thh}{\dra} \]
with the following properties.
\begin{itemize}
\item[{\rm (i)}] $X\in\Tcal$.
\item[{\rm (ii)}] $\del=\del\ppr\cup\thh$.
\item[{\rm (iii)}] $\E(\Tcal,B\ppr)=0$.
\item[{\rm (iv)}] For $2\le k\le n-1$, we have
\[ \E^k(T,B\ppr)=0\ \Leftrightarrow\ \E^{k-1}(T,B)=0 \]
for any $T\in\Tcal$.
\end{itemize}
\end{lem}
\begin{proof}
As in Remark~\ref{RemnCluster} {\rm (2)}, there is an $\sfr$-triangle $D\to X\ov{x}{\lra}B\ov{\mu}{\dra}$ with $X\in\Tcal$, in which $x$ is a right $\Tcal$-approximation. Remark that this implies $\E(\Tcal,D)=0$ and
\begin{equation}\label{IsoEBD}
\E^{k-1}(T,B)\cong\E^k(T,D)
\end{equation}
for any $T\in\Tcal$ and any $2\le k\le n-1$.
By {\rm (ET4$\op$)} (i.e., the dual of Lemma~\ref{Lem1EAOcta}), we obtain a diagram made of $\sfr$-triangles
\[
\xy
(-21,8)*+{D}="0";
(-7,8)*+{{}^{\exists}M\ppr}="2";
(7,8)*+{\Sig^{\ell-1}A}="4";
(22,8)*+{}="6";
(-21,-7)*+{D}="10";
(-7,-7)*+{X}="12";
(7,-7)*+{B}="14";
(22,-7)*+{}="16";
(-7,-21)*+{M}="22";
(7,-21)*+{M}="24";
(-7,-34)*+{}="32";
(7,-34)*+{}="34";
{\ar^(0.4){{}^{\exists}d} "0";"2"};
{\ar^(0.4){{}^{\exists}p} "2";"4"};
{\ar@{-->}^(0.6){a\uas\mu} "4";"6"};
{\ar@{=} "0";"10"};
{\ar_{{}^{\exists}m} "2";"12"};
{\ar^{a} "4";"14"};
{\ar_{} "10";"12"};
{\ar_{x} "12";"14"};
{\ar@{-->}^{\mu} "14";"16"};
{\ar_{e=b\ci x} "12";"22"};
{\ar^{b} "14";"24"};
{\ar@{=} "22";"24"};
{\ar@{-->}_{{}^{\exists}\thh} "22";"32"};
{\ar@{-->}^{\del} "24";"34"};
{\ar@{}|\circlearrowright "0";"12"};
{\ar@{}|\circlearrowright "2";"14"};
{\ar@{}|\circlearrowright "12";"24"};
\endxy
\]
satisfying $p\sas\thh=\del$. Using an $\sfr$-triangle
\[ \Sig^{\ell-2}A\to I\to\Sig^{\ell-1}A\ov{\iota^{\Sig^{\ell-2}A}}{\dra}\quad({}^{\exists}I\in\Ical) \]
taken as in Assumption~\ref{Assump}, put $\del\ppr=p\uas(\iota^{\Sig^{\ell-2}A})\in\E(M\ppr,\Sig^{\ell-2}A)=\E^{\ell-1}(M\ppr,A)$. By definition, this means $\ovl{\del\ppr}=\ovl{p}$ (Remark~\ref{RemBar} and Definition~\ref{DefBar}). Thus $\del\ppr$ satisfies
\[ \del\ppr\cup\thh=\E(C,\ovl{\del\ppr})(\thh)=\ovl{p}\sas\thh=\del. \]
Realize $\del\ppr$, to obtain an $\sfr$-triangle
\[ \Sig^{\ell-2}A\ov{a\ppr}{\lra}B\ppr\ov{b\ppr}{\lra}M\ppr\ov{\del\ppr}{\dra}. \]
It remains to confirm {\rm (iii)} and {\rm (iv)}. Let $T\in\Tcal$ be any object. Remark that
\[
\xy
(-48,7)*+{\C(T,M\ppr)}="-2";
(-24,7)*+{\C(T,\Sig^{\ell-1}A)}="0";
(0,7)*+{\E(T,D)}="2";
(20,7)*+{\E(T,M\ppr)}="4";
(46,7)*+{\E(T,\Sig^{\ell-1}A)}="6";
(-48,-7)*+{\C(T,M\ppr)}="-12";
(-24,-7)*+{\E(T,\Sig^{\ell-2}A)}="10";
(0,-7)*+{\E(T,B\ppr)}="12";
(20,-7)*+{\E(T,M\ppr)}="14";
(46,-7)*+{\E^2(T,\Sig^{\ell-2}A)}="16";
{\ar^(0.44){p\ci-} "-2";"0"};
{\ar^(0.58){(a\uas\mu)\ssh} "0";"2"};
{\ar^{d\sas} "2";"4"};
{\ar^(0.46){p\sas} "4";"6"};
{\ar@{=} "-2";"-12"};
{\ar^{\iota^{\Sig^{\ell-2}A}\ssh} "0";"10"};
{\ar@{=} "4";"14"};
{\ar@{=} "6";"16"};
{\ar_(0.44){\del\ppr\ssh} "-12";"10"};
{\ar_(0.58){a\ppr\sas} "10";"12"};
{\ar_{b\ppr\sas} "12";"14"};
{\ar_(0.44){\del\ppr\cup-} "14";"16"};
{\ar@{}|\circlearrowright "-2";"10"};
{\ar@{}|\circlearrowright "4";"16"};
\endxy
\]
is commutative, where the two rows are exact. Since $\E(T,D)=0$, we see that $p\ci-$ is surjective and $p\sas$ is injective in the top row. Since $\iota^{\Sig^{\ell-2}A}\ssh\co\C(T,\Sig^{\ell-1}A)\to\E(T,\Sig^{\ell-2}A)$ is surjective, this shows that $\del\ppr\ssh$ is surjective and $\del\ppr\cup-$ is injective in the bottom row. This means $\E(T,B\ppr)=0$ for any $T\in\Tcal$, and thus {\rm (iii)} follows.

Similarly for $k\ge2$, we have the following commutative diagram,
\[
\xy
(-58,7)*+{\E^{k-1}(T,M\ppr)}="-2";
(-24,7)*+{\E^{k-1}(T,\Sig^{\ell-1}A)}="0";
(1,7)*+{\E^k(T,D)}="2";
(21,7)*+{\E^k(T,M\ppr)}="4";
(50,7)*+{\E^k(T,\Sig^{\ell-1}A)}="6";
(-58,-7)*+{\E^{k-1}(T,M\ppr)}="-12";
(-24,-7)*+{\E^{k}(T,\Sig^{\ell-2}A)}="10";
(1,-7)*+{\E^k(T,B\ppr)}="12";
(21,-7)*+{\E^k(T,M\ppr)}="14";
(50,-7)*+{\E^{k+1}(T,\Sig^{\ell-2}A)}="16";
{\ar^(0.46){\E^{k-1}(T,p)} "-2";"0"};
{\ar "0";"2"};
{\ar "2";"4"};
{\ar^(0.46){\E^k(T,p)} "4";"6"};
{\ar@{=} "-2";"-12"};
{\ar@{=} "0";"10"};
{\ar@{=} "4";"14"};
{\ar@{=} "6";"16"};
{\ar_(0.46){\del\ppr\cup-} "-12";"10"};
{\ar "10";"12"};
{\ar "12";"14"};
{\ar_(0.44){\del\ppr\cup-} "14";"16"};
{\ar@{}|\circlearrowright "-2";"10"};
{\ar@{}|\circlearrowright "4";"16"};
\endxy
\]
in which two rows are exact. Thus the same argument shows
\[ \E^k(T,B\ppr)=0\ \Leftrightarrow\ \E^k(T,D)=0. \]
Together with $(\ref{IsoEBD})$, we obtain {\rm (iv)}.
\end{proof}

\begin{prop}\label{PropsDecExist}
Let $\Tcal\se\C$ be any $n$-cluster tilting subcategory. Then for any $A,C\in\C$ and any $\del\in\E^n(C,A)$, there are $\sfr$-triangles
\begin{equation}\label{sdec3}
\begin{cases}
\ A\ov{d_X^0}{\lra}X^1\ov{e^1}{\lra}M^1\ov{\del_{(1)}}{\dra},&\\
\ M^i\ov{m^i}{\lra}X^{i+1}\ov{e^{i+1}}{\lra}M^{i+1}\ov{\del_{(i+1)}}{\dra}&(i=1,\ldots,n-2),\\
\ M^{n-1}\ov{m^{n-1}}{\lra}X^n\ov{d_X^n}{\lra}C\ov{\del_{(n)}}{\dra}
\end{cases}
\end{equation}
which satisfy $X^i\in\Tcal$ for any $1\le i\le n$ and $\del=\del_{(1)}\cup\cdots\cup\del_{(n)}$.

In particular when $A,C\in\Tcal$, then ${}_AX^{\mr}_C\in\TAC$ with $d_X^i$ defined by $d_X^i=m^i\ci e^i$ for $1\le i\le n-1$, gives an $\sfr$-decomposable object $\Xd\in\AE^{n+2}_{(\Tcal,\E^n)}$.
\end{prop}
\begin{proof}
Let $\Sig^{n-1}A\ov{a^n}{\lra}B^n\ov{b^n}{\lra}C\ov{\del}{\dra}$ be an $\sfr$-triangle realizing $\del$. By Lemma~\ref{LemsDecExist}, we obtain $\sfr$-triangles
\[ \Sig^{n-2}A\ov{a^{n-1}}{\lra}B^{n-1}\ov{b^{n-1}}{\lra}M^{n-1}\ov{\del^{n-1}}{\dra}%
\ \ \text{and}\ \ %
M^{n-1}\ov{m^{n-1}}{\lra}X^n\ov{e^n}{\lra}C\ov{\del_{(n)}}{\dra} \]
satisfying $\del=\del^{n-1}\cup\del_{(n)}$, $X^n\in\Tcal$, $\E(\Tcal,B^{n-1})=0$ and
\[ \E^k(\Tcal,B^{n-1})=0\ \ \Leftrightarrow\ \ \E^{k-1}(\Tcal,B^n)=0 \]
for $2\le k\le n-1$.

Inductively, we obtain $\sfr$-triangles
\[ \Sig^{i-1}A\ov{a^i}{\lra}B^i\ov{b^i}{\lra}M^i\ov{\del^i}{\dra}%
\ \ \text{and}\ \ %
M^i\ov{m^i}{\lra}X^{i+1}\ov{e^{i+1}}{\lra}M^{i+1}\ov{\del_{(i+1)}}{\dra} \]
for each $1\le i\le n-2$, satisfying $\del^{i+1}=\del^i\cup\del_{(i+1)}$, $X^{i+1}\in\Tcal$, $\E(\Tcal,B^i)=0$ and
\[ \E^k(\Tcal,B^i)=0\ \ \Leftrightarrow\ \ \E^{k-1}(\Tcal,B^{i+1})=0 \]
for $2\le k\le n-1$.

Then $B^1$ satisfies $\E^k(\Tcal,B^1)=0$ for any $1\le k\le n-1$, and thus $B^1\in\Tcal$ holds. Moreover, we have
\[ \del=\del^{n-1}\cup\del_{(n)}=\del^{n-2}\cup\del_{(n-1)}\cup\del_{(n)}=\cdots=\del^1\cup\del_{(2)}\cup\cdots\cup\del_{(n)}. \]
Re-naming the $\sfr$-triangle $A\ov{a^1}{\lra}B^1\ov{b^1}{\lra}M^1\ov{\del^1}{\dra}$ as $A\ov{d_X^0}{\lra}X^1\ov{e^1}{\lra}M^1\ov{\del_{(1)}}{\dra}$ and putting $d_X^n=e^n$, we obtain $(\ref{sdec3})$. The latter part is obvious.
\end{proof}

\begin{lem}\label{LemMvanish}
Let $\Tcal\se\C$ be any $n$-cluster tilting subcategory. For any $\sfr$-decomposable object ${}_A\Xd_C\in\AE^{n+2}_{(\Tcal,\E^n)}$, its $\sfr$-decomposition $(\ref{sdec2})$ satisfies the following.
\begin{enumerate}
\item For any $1\le i\le n-1$,
\[ \E^k(\Tcal,M^i)=0 \]
holds for $k\in\{1,\ldots,n-1\}\setminus\{ n-i\}$. As for the terms with $i+k=n$, we have a sequence of homomorphisms
\[ 
\E(T,M^{n-1})\ov{\del_{(n-1)}\cup-}{\un{\cong}{\lra}}\cdots\ov{\del_{(2)}\cup-}{\un{\cong}{\lra}}\E^{n-1}(T,M^1)\ov{\del_{(1)}\cup-}{\lra}\E^n(T,A) \]
for any $T\in\Tcal$, in which
\begin{itemize}
\item $\del_{(i)}\cup-$ are isomorphisms for $2\le i\le n-1$,
\item $\del_{(1)}\cup-$ is injective.
\end{itemize}
In particular, their composition
\[ (\del_{(1)}\cup\cdots\cup\del_{(n-1)})\cup-\co\E(T,M^{n-1})\to\E^n(T,A) \]
is injective.
\item Dually, for any $1\le i\le n-1$,
\[ \E^k(M^i,\Tcal)=0 \]
holds for $k\in\{1,\ldots,n-1\}\setminus\{ i\}$.
\end{enumerate}
\end{lem}
\begin{proof}
{\rm (1)} Since $X^i\in\Tcal$, we have isomorphisms
\begin{eqnarray*}
\del_{(1)}\cup-&\co&\E^k(T,M^1)\ov{\cong}{\lra}\E^{k+1}(T,A),\\
\del_{(i+1)}\cup-&\co&\E^k(T,M^{i+1})\ov{\cong}{\lra}\E^{k+1}(T,M^i)\ \ (1\le i\le n-2),\\
\del_{(n)}\cup-&\co&\E^k(T,C)\ov{\cong}{\lra}\E^{k+1}(T,M^{n-1})
\end{eqnarray*}
for any $1\le k\le n-2$ and any $T\in\Tcal$. Thus the former half of the statement follows from
\[ \E^k(T,A)=0\ \ \text{and}\ \ \E^k(T,C)=0\quad(1\le k\le n-1). \]
The latter part is also straightforward.
{\rm (2)} is dual to {\rm (1)}.
\end{proof}

\begin{prop}\label{PropDecompEx}
Let $\Tcal\se\C$ be any $n$-cluster tilting subcategory satisfying Condition~\ref{CondnCluster}. Then any $\sfr$-decomposable object ${}_A\Xd_C\in\AE^{n+2}_{(\Tcal,\E^n)}$ is an $n$-exangle.
\end{prop}
\begin{proof}
By the duality, let us only show the exactness of
\begin{equation}\label{ExTTT}
\Tcal(T,A)\ov{\Tcal(T,d_X^0)}{\lra}\Tcal(T,X^1)\ov{\Tcal(T,d_X^1)}{\lra}\cdots\ov{\Tcal(T,d_X^n)}{\lra}\Tcal(T,C)\ov{\del\ssh}{\lra}\E^n(T,A)
\end{equation}
for any $T\in\Tcal$.

Let $(\ref{sdec2})$ be an $\sfr$-decomposition of $\Xd$. By Lemma~\ref{LemMvanish}, we have $\E^{n-1}(T,M^i)=0$ for any $2\le i\le n-1$, and $\E(T,M^i)=0$ for any $1\le i\le n-2$. By Proposition~\ref{PropABCvanish}, we obtain short exact sequences
\[ 0\to\C(T,M^i)\ov{\C(T,m^i)}{\lra}\C(T,X^{i+1})\ov{\C(T,e^{i+1})}{\lra}\C(T,M^{i+1})\to0\quad(1\le i\le n-2). \]
Since the sequences
\begin{eqnarray*}
&\C(T,A)\ov{\C(T,d_X^0)}{\lra}\C(T,X^1)\ov{\C(T,e^1)}{\lra}\C(T,M^1)\to0,&\\
&0\to\C(T,M^{n-1})\ov{\C(T,m^{n-1})}{\lra}\C(T,X^n)\ov{\C(T,d_X^n)}{\lra}\C(T,C)\ov{(\del_{(n)})\ssh}{\lra}\E(T,M^{n-1})&
\end{eqnarray*}
are also exact by $\E(T,A)=0$ and $\E^{n-1}(T,C)=0$, we obtain an exact sequence
\[
\Tcal(T,A)\ov{\Tcal(T,d_X^0)}{\lra}\Tcal(T,X^1)\ov{\Tcal(T,d_X^1)}{\lra}\cdots\ov{\Tcal(T,d_X^n)}{\lra}\Tcal(T,C)\ov{(\del_{(n)})\ssh}{\lra}\E(T,M^{n-1}).
\]
Composing with the injection $(\del_{(1)}\cup\cdots\cup\del_{(n-1)})\cup-\co\E(T,M^{n-1})\to\E^n(T,A)$, we obtain $(\ref{ExTTT})$.
\end{proof}

\begin{prop}\label{PropDecompExAdded}
Let $\Tcal\se\C$ be any $n$-cluster tilting subcategory satisfying Condition~\ref{CondnCluster}. Let ${}_A\Xd_C,{}_B\Yr_D\in\AE^{n+2}_{(\Tcal,\E^n)}$ be any pair of $\sfr$-decomposable objects. Then for any pair of morphisms $a\in\C(A,B),c\in\C(C,D)$ satisfying
\[ \E^n(C,a)(\del)=\E^n(c,B)(\rho), \]
there exists a morphism $f^{\mr}\in \Cbf_{\Tcal}^{n+2}(X^{\mr},Y^{\mr})$ satisfying $f^0=a$ and $f^{n+1}=c$.
\end{prop}
\begin{proof}
Let $(\ref{sdec2})$ be an $\sfr$-decomposition of $\Xd$, and let
\begin{equation}\label{sdecofY}
\begin{cases}
\ B\ov{d_Y^0}{\lra}Y^1\ov{t^1}{\lra}N^1\ov{\rho_{(1)}}{\dra}&\\
\ N^i\ov{s^i}{\lra}Y^{i+1}\ov{t^{i+1}}{\lra}N^{i+1}\ov{\rho_{(i+1)}}{\dra}&(i=1,\ldots,n-1)\\
\ N^{n-1}\ov{s^{n-1}}{\lra}Y^n\ov{d_Y^n}{\lra}D\ov{\rho_{(n)}}{\dra}
\end{cases}
\end{equation}
be an $\sfr$-decomposition of $\Yr$.
Since $\Tcal(X^1,Y^1)\ov{-\ci d_X^0}{\lra}\Tcal(A,Y^1)\to\E^n(C,Y^1)$ is exact by Proposition~\ref{PropDecompEx}, there is $f^1\in\Tcal(X^1,Y^1)$ which gives $f^1\ci d_X^0=d_Y^0\ci a$. Then we obtain $\ell^1\in\Tcal(M^1,N^1)$ which satisfies
\[ t^1\ci f^1=\ell^1\ci e^1\ \ \text{and}\ \ (\ell^1)\uas(\rho_{(1)})=a\sas(\del_{(1)}). \]
Since $\C(X^i,Y^i)\ov{-\ci m^{i-1}}{\lra}\C(M^{i-1},Y^i)$ is surjective for $i\ge 2$, inductively we obtain $\ell^i\in\C(M^i,N^i)$ and $f^i\in\Tcal(X^i,Y^i)$ for any $1\le i\le n-1$, satisfying
\[ f^i\ci m^{i-1}=s^{i-1}\ci\ell^{i-1},\ \ \ \ell^i\ci e^i=t^i\ci f^i,\ \ \ (\ell^i)\uas(\rho_{(i)})=(\ell^{i-1})\sas(\del_{(i)}) \]
for all $2\le i\le n-1$.
Then since $(\ell^{n-1})\sas(\del_{(n)})\in\E(C,N^{n-1})$ satisfies
\begin{eqnarray*}
(\rho_{(1)}\cup\cdots\cup\rho_{(n-1)})\cup((\ell^{n-1})\sas(\del_{(n)}))%
&=&((a\sas\del_{(1)})\cup\cdots\cup\del_{(n-1)})\cup\del_{(n)}\\
&=&\E^n(C,a)(\del)\ =\ \E^n(c,B)(\rho)\\
&=& (\rho_{(1)}\cup\cdots\cup\rho_{(n-1)})\cup(c\uas\rho_{(n)})
\end{eqnarray*}
by Proposition~\ref{PropCupProdNat}, it follows $(\ell^{n-1})\sas(\del_{(n)})=c\uas\rho_{(n)}$ from the injectivity of
\[ (\rho_{(1)}\cup\cdots\cup\rho_{(n-1)})\cup-\co\E(C,N^{n-1})\to\E^n(C,B). \]
Thus there exists $f^n\in\Tcal(X^n,Y^n)$ which makes
\[
\xy
(-18,6)*+{M^{n-1}}="0";
(0,6)*+{X^n}="2";
(18,6)*+{C}="4";
(-18,-6)*+{N^{n-1}}="10";
(0,-6)*+{Y^n}="12";
(18,-6)*+{D}="14";
{\ar^(0.56){m^{n-1}} "0";"2"};
{\ar^{d_X^n} "2";"4"};
{\ar_{\ell^{n-1}} "0";"10"};
{\ar^{f^n} "2";"12"};
{\ar^{c} "4";"14"};
{\ar_{s^{n-1}} "10";"12"};
{\ar_{d_Y^n} "12";"14"};
{\ar@{}|\circlearrowright "0";"12"};
{\ar@{}|\circlearrowright "2";"14"};
\endxy
\]
commutative. Then $f^{\mr}=(a,f^1,\ldots,f^n,c)\in\Cbf_{\Tcal}^{n+2}(X^{\mr},Y^{\mr})$ gives a morphism. 
\end{proof}

\begin{cor}\label{CorDecompEx}
Let $\Tcal\se\C$ be any $n$-cluster tilting subcategory satisfying Condition~\ref{CondnCluster}. For any $\del\in\E^n(C,A)$, an object $X^{\mr}\in\TAC$ which gives $\sfr$-decomposable $\Xd\in\AE^{n+2}_{(\Tcal,\E^n)}$, is unique up to homotopical equivalence in $\TAC$. Namely, if $\Yd\in\AE^{n+2}_{(\Tcal,\E^n)}$ is also $\sfr$-decomposable, then $Y^{\mr}$ is homotopically equivalent to $X^{\mr}$ in $\TAC$.
\end{cor}
\begin{proof}
By Proposition~\ref{PropDecompExAdded}, we have $\TAC(X^{\mr},Y^{\mr})\ne\emptyset$ and $\TAC(Y^{\mr},X^{\mr})\ne\emptyset$. Thus Proposition~\ref{PropCAC1} shows that $X^{\mr}$ and $Y^{\mr}$ are homotopically equivalent in $\TAC$.
\end{proof}

\begin{cor}\label{CorShowEA2}
Let $\Tcal\se\C$ be any $n$-cluster tilting subcategory satisfying Condition~\ref{CondnCluster}. Let $\Xd,\Yr\in\AE^{n+2}_{(\Tcal,\E^n)}$ be any pair of $\sfr$-decomposable objects, with $\sfr$-decompositions $(\ref{sdec2})$ and $(\ref{sdecofY})$, respectively.

Suppose $A=B$. Then for any $c\in\C(C,D)$ satisfying $\E^n(c,A)(\rho)=\del$, there exists $f^{\mr}\in\Cbf^{n+2}_{\Tcal}(X^{\mr},Y^{\mr})$ with $f^0=\id_A$ and $f^{n+1}=c$, whose mapping cone $\Mf$ gives $\sfr$-decomposable object $\langle \Mf,\E^n(D,d_X^0)(\rho)\rangle\in\AE^{n+2}_{(\Tcal,\E^n)}$.
\end{cor}
\begin{proof}
As in the proof of Proposition~\ref{PropDecompExAdded}, we have some sequence of morphisms $\{\ell^i\in\C(M^i,N^i)\}_{1\le i\le n-1}$ in $\C$ satisfying
\begin{eqnarray*}
\del_{(1)}&=&(\ell^1)\uas\rho_{(1)},\\
(\ell^{i-1})\sas\del_{(i)}&=&(\ell^i)\uas\rho_{(i)}\quad(2\le i\le n-1),\\
(\ell^{n-1})\sas\del_{(n)}&=&c\uas\rho_{(n)}.
\end{eqnarray*}
For a convenience of the description, put
\begin{eqnarray*}
&M^0=N^0=A,\ M^n=C,\ N^n=D,&\\
&m^0=d_X^0,\ e^n=d_X^n,\ s^0=d_Y^0,\ t^n=d_Y^n,&
\end{eqnarray*}
and $\ell^0=\id_A$, $\ell^n=c$. The above equations can be written uniformly as
\[ (\ell^{i-1})\sas\del_{(i)}=(\ell^i)\uas\rho_{(i)}\quad(1\le i\le n). \]
For each $1\le i\le n$, put $\tau_{(i)}=(\ell^{i-1})\sas\del_{(i)}\big(=(\ell^i)\uas\rho_{(i)}\big)$  and realize it to obtain an $\sfr$-triangle
\begin{equation}\label{Realtau}
N^{i-1}\ov{u^{i-1}}{\lra}L^i\ov{v^i}{\lra}M^i\ov{\tau_{(i)}}{\dra}.
\end{equation}
For $i=1$, we choose $(\ref{Realtau})$ as $A\ov{d_X^0}{\lra}X^1\ov{e^1}{\lra}M^1\ov{\del_{(1)}}{\dra}$.

In the sequel, we construct morphisms $x^i\in\C(X^i,L^i)$ and $y^i\in\C(L^i,Y^i)$ for $1\le i\le n$. Firstly, let us give $x^i$. We just put $x^1=\id_{X^1}$ for $i=1$. For $2\le i\le n$, as a lift of $(\ell^{i-1},\id_{M^i})\co\del_{(i)}\to\tau_{(i)}$, take a morphism of $\sfr$-triangles
\begin{equation}\label{strMorphMXM}
\xy
(-16,6)*+{M^{i-1}}="0";
(0,6)*+{X^i}="2";
(16,6)*+{M^i}="4";
(30,6)*+{}="6";
(-16,-6)*+{N^{i-1}}="10";
(0,-6)*+{L^i}="12";
(16,-6)*+{M^i}="14";
(30,-6)*+{}="16";
{\ar^(0.56){m^{i-1}} "0";"2"};
{\ar^{e^i} "2";"4"};
{\ar@{-->}^{\del_{(i)}} "4";"6"};
{\ar_{\ell^{i-1}} "0";"10"};
{\ar^{x^i} "2";"12"};
{\ar@{=} "4";"14"};
{\ar_{u^{i-1}} "10";"12"};
{\ar_{v^i} "12";"14"};
{\ar@{-->}_{\tau_{(i)}} "14";"16"};
{\ar@{}|\circlearrowright "0";"12"};
{\ar@{}|\circlearrowright "2";"14"};
\endxy
\end{equation}
which makes
\[ M^{i-1}\ov{\Big[\raise1ex\hbox{\leavevmode\vtop{\baselineskip-8ex \lineskip1ex \ialign{#\crcr{$\ \scriptstyle{m^{i-1}}$}\crcr{$\scriptstyle{-\ell^{i-1}}$}\crcr}}}\Big]}{\lra}X^i\oplus N^{i-1}\ov{[x^i\ u^{i-1}]}{\lra}L^i\ov{(v^i)\uas\del_{(i)}}{\dra} \]
an $\sfr$-triangle. Secondly, let us give $y^i$. For $i=n$, we take a good lift  of $(\id_{N^{n-1}},c)$
\begin{equation}\label{GoodLiftAdd}
\xy
(-16,6)*+{N^{n-1}}="0";
(0,6)*+{L^n}="2";
(16,6)*+{C}="4";
(30,6)*+{}="6";
(-16,-6)*+{N^{n-1}}="10";
(0,-6)*+{Y^n}="12";
(16,-6)*+{D}="14";
(30,-6)*+{}="16";
{\ar^(0.56){u^{n-1}} "0";"2"};
{\ar^{v^n} "2";"4"};
{\ar@{-->}^{\tau_{(n)}} "4";"6"};
{\ar@{=} "0";"10"};
{\ar^{y^n} "2";"12"};
{\ar^{c} "4";"14"};
{\ar_{s^{n-1}} "10";"12"};
{\ar_{d_Y^n} "12";"14"};
{\ar@{-->}_{\rho_{(n)}} "14";"16"};
{\ar@{}|\circlearrowright "0";"12"};
{\ar@{}|\circlearrowright "2";"14"};
\endxy
\end{equation}
so that
\begin{equation}\label{sdecMn}
L^n\ov{\Big[\raise1.5ex\hbox{\leavevmode\vtop{\baselineskip-8ex \lineskip1ex \ialign{#\crcr{$\ \scriptstyle{-v^n}$}\crcr{$\scriptstyle{\ y^n}$}\crcr}}}\Big]}{\lra}C\oplus Y^n\ov{[c\ d_Y^n]}{\lra}D\ov{(u^{n-1})\sas\rho_{(n)}}{\dra}
\end{equation}
becomes an $\sfr$-triangle. For $1\le i\le n-1$, remark that
\[ N^{i-1}\ov{\Big[\raise1.4ex\hbox{\leavevmode\vtop{\baselineskip-8ex \lineskip1ex \ialign{#\crcr{\ $\scriptstyle{0}$\ }\crcr{$\scriptstyle{s^{i-1}}$}\crcr}}}\Big]}{\lra}X^{i+1}\oplus Y^i\ov{\id\oplus t^i}{\lra}X^{i+1}\oplus N^i\ov{(p^i)\uas\rho_{(i)}}{\dra} \]
is an $\sfr$-triangle, where $p^i=[0\ 1]\co X^{i+1}\oplus N^i\to N^i$ is the projection. If we put $r^{i+1}=[x^{i+1}\ u^i\ci t^i]$, then by {\rm (ET4$\op$)} (i.e., the dual of Lemma~\ref{Lem1EAOcta}) we obtain a diagram made of $\sfr$-triangles
\begin{equation}\label{TempsNLM}
\xy
(-36,8)*+{N^{i-1}}="0";
(-12,8)*+{{}^{\exists}L\ppr}="2";
(12,8)*+{M^i}="4";
(36,8)*+{}="6";
(-36,-8)*+{N^{i-1}}="10";
(-12,-8)*+{X^{i+1}\oplus Y^i}="12";
(12,-8)*+{X^{i+1}\oplus N^i}="14";
(36,-8)*+{}="16";
(-12,-24)*+{L^{i+1}}="22";
(12,-24)*+{L^{i+1}}="24";
(-12,-39)*+{}="32";
(12,-39)*+{}="34";
{\ar^{{}^{\exists}u\ppr} "0";"2"};
{\ar^{{}^{\exists}v\ppr} "2";"4"};
{\ar@{-->}^(0.7){(p^i)\uas\rho_{(i)}} "4";"6"};
(20,12)*+{\Big[\ \ \ \Big]};
(20,14.4)*+{\scriptstyle{m^i}};
(20,10.8)*+{\scriptstyle{-\ell^i}};
(23.8,13.6)*+{{}\uas};
(16,0)*+{\Big[\ \ \ \Big]};
(16,2.4)*+{\scriptstyle{m^i}};
(16,-1.2)*+{\scriptstyle{-\ell^i}};
{\ar@{=} "0";"10"};
{\ar_{{}^{\exists}q\ppr} "2";"12"};
{\ar^{} "4";"14"};
{\ar_{} "10";"12"};
{\ar_{\id\oplus t^i} "12";"14"};
{\ar@{-->}^(0.64){(p^i)\uas\rho_{(i)}} "14";"16"};
(-26,-12.42)*+{\Big[\ \ \ \ \Big]};
(-26,-11)*+{\scriptstyle{0}};
(-26,-14)*+{\scriptstyle{s^{i-1}}};
{\ar_{r^{i+1}} "12";"22"};
{\ar^{[x^{i+1}\ u^i]} "14";"24"};
{\ar@{=} "22";"24"};
{\ar@{-->}_{{}^{\exists}\partial\ppr} "22";"32"};
{\ar@{-->}^{(v^{i+1})\uas\del_{(i+1)}} "24";"34"};
{\ar@{}|\circlearrowright "0";"12"};
{\ar@{}|\circlearrowright "2";"14"};
{\ar@{}|\circlearrowright "12";"24"};
\endxy
\end{equation}
satisfying $v\ppr\sas\partial\ppr=(v^{i+1})\uas\del_{(i+1)}$ and $[x^{i+1}\ u^i]\uas\partial\ppr=u\ppr\sas(p^i)\uas\rho_{(i)}$. Since
\[ \Big[\raise1ex\hbox{\leavevmode\vtop{\baselineskip-8ex \lineskip1ex \ialign{#\crcr{$\ \scriptstyle{m^i}$}\crcr{$\scriptstyle{-\ell^i}$}\crcr}}}\Big]\uas (p^i)\uas\rho_{(i)}=-(\ell^i)\uas\rho_{(i)}=-\tau_{(i)}, \]
there is an isomorphism $j\co L\ppr\ov{\cong}{\lra}L^i$ which makes
\[
\xy
(-16,0)*+{N^{i-1}}="0";
(3,0)*+{}="1";
(0,8)*+{L\ppr}="2";
(0,-8)*+{L^i}="4";
(-3,0)*+{}="5";
(16,0)*+{M^i}="6";
{\ar^{u\ppr} "0";"2"};
{\ar^{v\ppr} "2";"6"};
{\ar_{u^{i-1}} "0";"4"};
{\ar_{-v^i} "4";"6"};
{\ar^{j}_{\cong} "2";"4"};
{\ar@{}|\circlearrowright "0";"1"};
{\ar@{}|\circlearrowright "5";"6"};
\endxy
\]
commutative in $\C$. Composing with this isomorphism, if we put $q^i=q\ppr\ci j\iv$ and $\partial_{(i+1)}=j\sas\partial\ppr$, we obtain a diagram made of $\sfr$-triangles
\begin{equation}\label{ObtsNLM}
\xy
(-36,8)*+{N^{i-1}}="0";
(-12,8)*+{L^i}="2";
(12,8)*+{M^i}="4";
(36,8)*+{}="6";
(-36,-8)*+{N^{i-1}}="10";
(-12,-8)*+{X^{i+1}\oplus Y^i}="12";
(12,-8)*+{X^{i+1}\oplus N^i}="14";
(36,-8)*+{}="16";
(-12,-24)*+{L^{i+1}}="22";
(12,-24)*+{L^{i+1}}="24";
(-12,-39)*+{}="32";
(12,-39)*+{}="34";
{\ar^{u^{i-1}} "0";"2"};
{\ar^{-v^i} "2";"4"};
{\ar@{-->}^{-\tau_{(i)}} "4";"6"};
(16,0)*+{\Big[\ \ \ \Big]};
(16,2.4)*+{\scriptstyle{m^i}};
(16,-1.2)*+{\scriptstyle{-\ell^i}};
{\ar@{=} "0";"10"};
{\ar_{q^i} "2";"12"};
{\ar^{} "4";"14"};
{\ar_{} "10";"12"};
{\ar_{\id\oplus t^i} "12";"14"};
{\ar@{-->}^(0.64){(p^i)\uas\rho_{(i)}} "14";"16"};
(-26,-12.42)*+{\Big[\ \ \ \ \Big]};
(-26,-11)*+{\scriptstyle{0}};
(-26,-14)*+{\scriptstyle{s^{i-1}}};
{\ar_{r^{i+1}} "12";"22"};
{\ar^{[x^{i+1}\ u^i]} "14";"24"};
{\ar@{=} "22";"24"};
{\ar@{-->}_{\partial_{(i+1)}} "22";"32"};
{\ar@{-->}^{(v^{i+1})\uas\del_{(i+1)}} "24";"34"};
{\ar@{}|\circlearrowright "0";"12"};
{\ar@{}|\circlearrowright "2";"14"};
{\ar@{}|\circlearrowright "12";"24"};
\endxy
\end{equation}
satisfying $-v^i\sas\partial_{(i+1)}=(v^{i+1})\uas\del_{(i+1)}$ and $[x^{i+1}\ u^i]\uas\partial_{(i+1)}=u^{i-1}\sas(p^i)\uas\rho_{(i)}$. In particular
\begin{equation}\label{Relofdelandrho}
(u^i)\uas\partial_{(i+1)}=u^{i-1}\sas\rho_{(i)}
\end{equation}
holds, and
\begin{equation}\label{sdecMi}
L^i\ov{q^i}{\lra}X^{i+1}\oplus Y^i\ov{r^{i+1}}{\lra}L^{i+1}\ov{\partial_{(i+1)}}{\dra}
\end{equation}
becomes an $\sfr$-triangle, for any $1\le i\le n-1$. By the commutativity of $(\ref{ObtsNLM})$, morphism $q^i$ is of the form
\[ q^i=\begin{bmatrix}-m^i\ci v^i\\ y^i\end{bmatrix} \]
for some $y^i\in\C(L^i,Y^i)$, which satisfies
\begin{equation}\label{propertyofy}
y^i\ci u^{i-1}=s^{i-1}\quad\text{and}\quad t^i\ci y^i=\ell^i\ci v^i.
\end{equation}

Thus we have obtained $x^i\in\C(X^i,L^i)$ and $y^i\in\C(L^i,Y^i)$ for $1\le i\le n$. If we put $f^i=y^i\ci x^i\in\C(X^i,Y^i)$, they form a morphism
\[ f^{\mr}=(\id_A,f^1,\ldots,f^n,c)\in\Cbf^{n+2}_{\Tcal}(X^{\mr},Y^{\mr}) \]
by the commutativity of $(\ref{strMorphMXM})$, (\ref{GoodLiftAdd}) and $(\ref{propertyofy})$. Let us show the $\sfr$-decomposability of $\langle \Mf,\E^n(D,d_X^0)(\rho)\rangle$.
By the argument so far, we have $\sfr$-triangles
\begin{equation}\label{sdecMC}
\begin{cases}
\ X^1\ov{q^1}{\lra}X^2\oplus Y^1\ov{r^2}{\lra}L^2\ov{\partial_{(2)}}{\dra},&\\
\ L^i\ov{q^i}{\lra}X^{i+1}\oplus Y^i\ov{r^{i+1}}{\lra}L^{i+1}\ov{\partial_{(i+1)}}{\dra}&(i=2,\ldots,n-1),\\
\ L^n\ov{\Big[\raise1.5ex\hbox{\leavevmode\vtop{\baselineskip-8ex \lineskip1ex \ialign{#\crcr{$\ \scriptstyle{-v^n}$}\crcr{$\scriptstyle{\ y^n}$}\crcr}}}\Big]}{\lra}C\oplus Y^n\ov{[c\ d^n_Y]}{\lra}D\ov{u^{n-1}\sas\rho_{(n)}}{\dra}
\end{cases}
\end{equation}
as in $(\ref{sdecMi})$ and $(\ref{sdecMn})$. Commutativity of the related diagrams shows
\[ q^1=\begin{bmatrix}-m^1\ci e^1\\ y^1\end{bmatrix}=\begin{bmatrix}-d_X^1\\ f^1\end{bmatrix}=d_{M_f}^0, \]
\[ q^i\ci r^i=\left[\begin{array}{cc}-m^i\ci e^i&0\\ y^i\ci x^i&s^{i-1}\ci t^{i-1}\end{array}\right]=\left[\begin{array}{cc}-d_X^i&0\\ f^i&d_Y^{i-1}\end{array}\right]=d_{M_f}^{i-1}\quad(2\le i\le n), \]
\[ \begin{bmatrix}-v^n\\ y^n\end{bmatrix}\ci r^n=\left[\begin{array}{cc}-e^n&0\\ f^n&s^{n-1}\ci t^{n-1}\end{array}\right]=\left[\begin{array}{cc}-d_X^n&0\\ f^n&d_Y^{n-1}\end{array}\right]=d_{M_f}^n. \]
Moreover by Proposition~\ref{PropCupProdNat}, equality $(\ref{Relofdelandrho})$ shows
\begin{eqnarray*}
\partial_{(2)}\cup\partial_{(3)}\cup\cdots\cup\partial_{(n)}\cup u^{n-1}\sas\rho_{(n)}&=&\partial_{(2)}\cup\partial_{(3)}\cup\cdots\cup\big((u^{n-1})\uas\partial_{(n)}\big)\cup \rho_{(n)}\\
&=&\cdots\ =\big( u^0\sas\rho_{(1)}\big)\cup\rho_{(2)}\cup\cdots\cup\rho_{(n)}\\
&=&\E^n(D,d_X^0)(\rho)
\end{eqnarray*}
since $u^0=d_X^0$. Thus 
$(\ref{sdecMC})$ gives an $\sfr$-decomposition of $\langle \Mf,\E^n(D,d_X^0)(\rho)\rangle$.
\end{proof}

Let us show the converse of Corollary \ref{CorDecompEx}, under the assumption of Condition~\ref{CondWIC}.
\begin{rem}\label{RemsDecHtpyInv}
Let $\CEs$ be an extriangulated category, and let
\[ A\ov{x}{\lra}B\ov{y}{\lra}C\ov{\del}{\dra},\quad A\ov{x\ppr}{\lra}B\ppr\ov{y\ppr}{\lra}C\ppr\ov{\del\ppr}{\dra} \]
be $\sfr$-triangles. By {\rm (EA1)} and by Proposition~\ref{PropShiftedReal},
\begin{itemize}
\item any morphism $(\id_A,c)\co\del\to\del\ppr$
\end{itemize}
or
\begin{itemize}
\item any morphism $b\in\C(B,B\ppr)$ satisfying $b\ci x=x\ppr$
\end{itemize}
can be completed into a morphism of $\sfr$-triangles
\[
\xy
(-12,6)*+{A}="0";
(0,6)*+{B}="2";
(12,6)*+{C}="4";
(24,6)*+{}="6";
(-12,-6)*+{A}="10";
(0,-6)*+{B\ppr}="12";
(12,-6)*+{C\ppr}="14";
(24,-6)*+{}="16";
{\ar^{x} "0";"2"};
{\ar^{y} "2";"4"};
{\ar@{-->}^{\del} "4";"6"};
{\ar@{=} "0";"10"};
{\ar^{b} "2";"12"};
{\ar^{c} "4";"14"};
{\ar_{x\ppr} "10";"12"};
{\ar_{y\ppr} "12";"14"};
{\ar@{-->}_{\del\ppr} "14";"16"};
{\ar@{}|\circlearrowright "0";"12"};
{\ar@{}|\circlearrowright "2";"14"};
\endxy
\]
which gives an $\sfr$-triangle $B\ov{\Big[\raise1ex\hbox{\leavevmode\vtop{\baselineskip-8ex \lineskip1ex \ialign{#\crcr{$\scriptstyle{-y}$}\crcr{$\scriptstyle{\ b}$}\crcr}}}\Big]}{\lra}C\oplus B\ppr\ov{[c\ y\ppr]}{\lra}C\ppr\ov{x\sas\del\ppr}{\dra}$.
If $\del\ppr$ moreover satisfies $x\sas\del\ppr=0$, this means that
\[ 0\to B\ov{\Big[\raise1ex\hbox{\leavevmode\vtop{\baselineskip-8ex \lineskip1ex \ialign{#\crcr{$\scriptstyle{-y}$}\crcr{$\scriptstyle{\ b}$}\crcr}}}\Big]}{\lra}C\oplus B\ppr\ov{[c\ y\ppr]}{\lra}C\ppr\to0 \]
is a split short exact sequence in $\C$.
\end{rem}

\begin{prop}\label{PropsDecHtpyInv}
Assume that an extriangulated category $\CEs$ satisfies Condition~\ref{CondWIC}, and that its $n$-cluster tilting subcategory $\Tcal\se\C$ satisfies Condition~\ref{CondnCluster}.

Let ${}_A\Xd_C\in\AE^{n+2}_{(\Tcal,\E^n)}$ be an $\sfr$-decomposable object. If $Y^{\mr}\in\TAC$ is homotopically equivalent to $X^{\mr}$ in $\TAC$, then $\Yd$ is also $\sfr$-decomposable.
\end{prop}
\begin{proof}
Since the case $n=1$ is trivial, we assume $n\ge 2$.
Let $(\ref{sdec2})$ be an $\sfr$-decomposition of $\Xd$, and let $f^{\mr}\in\TAC(X^{\mr},Y^{\mr})$, $g^{\mr}\in\TAC(Y^{\mr},X^{\mr})$ be morphisms which give a homotopical equivalence. Remark that $\Xd$ is an $n$-exangle by Proposition~\ref{PropDecompEx}. 
Since $Y^{\mr}$ is homotopically equivalent to $X^{\mr}$ in $\TAC$, the pair $\Yd$ is also an $n$-exangle.

Let us first show the following.
\begin{claim}\label{ClaimDecHtpyInv}
For any $0\le i\le n-1$, there is a $6$-tuple $\Rfr^i=(N^i,s^i,t^i,\ell^i,q^i,\rho_{(i)})$ of
\begin{itemize}
\item[-] object $N^i\in\C$,
\item[-] morphisms $s^i\in\C(N^i,Y^{i+1})$, $t^i\in\C(Y^i,N^i)$, $\ell^i\in\C(M^i,N^i)$, $q^i\in\C(N^i,M^i)$
\item[-] extension $\rho_{(i)}\in\E(N^i,N^{i-1})$
\end{itemize}
satisfying the following properties for all $0\le i\le n-1$ (in which we put $M^0=A$, $N^{-1}=M^{-1}=0$ and $m^0=d_X^0$, $e^0=\id_A, \del_{(0)}=0$ for the convenience of the description).
\begin{itemize}
\item[$(A_i)$] $N^{i-1}\ov{s^{i-1}}{\lra}Y^i\ov{t^i}{\lra}N^i\ov{\rho_{(i)}}{\dra}$ is an $\sfr$-triangle.
\item[$(B_i)$] For any $T\in\Tcal$,
\[ \E^k(\ell^i,T)\co\E^k(N^i,T)\to\E^k(M^i,T)\ \ \text{and}\ \ \E^k(T,\ell^i)\co\E^k(T,M^i)\to\E^k(T,N^i) \]
are isomorphisms for all $1\le k\le n-1$.
\item[$(C_i)$] The following are morphisms of $\sfr$-triangles.
\[
\xy
(-14,6)*+{M^{i-1}}="0";
(0,6)*+{X^i}="2";
(12,6)*+{M^i}="4";
(24,6)*+{}="6";
(-14,-6)*+{N^{i-1}}="10";
(0,-6)*+{Y^i}="12";
(12,-6)*+{N^i}="14";
(24,-6)*+{}="16";
{\ar^(0.6){m^{i-1}} "0";"2"};
{\ar^{e^i} "2";"4"};
{\ar@{-->}^{\del_{(i)}} "4";"6"};
{\ar_{\ell^{i-1}} "0";"10"};
{\ar_{f^i} "2";"12"};
{\ar^{\ell^i} "4";"14"};
{\ar_(0.56){s^{i-1}} "10";"12"};
{\ar_{t^i} "12";"14"};
{\ar@{-->}_{\rho_{(i)}} "14";"16"};
{\ar@{}|\circlearrowright "0";"12"};
{\ar@{}|\circlearrowright "2";"14"};
\endxy\ \ ,\ \ 
\xy
(-14,6)*+{N^{i-1}}="0";
(0,6)*+{Y^i}="2";
(12,6)*+{N^i}="4";
(24,6)*+{}="6";
(-14,-6)*+{M^{i-1}}="10";
(0,-6)*+{X^i}="12";
(12,-6)*+{M^i}="14";
(24,-6)*+{}="16";
{\ar^(0.6){s^{i-1}} "0";"2"};
{\ar^{t^i} "2";"4"};
{\ar@{-->}^{\rho_{(i)}} "4";"6"};
{\ar_{q^{i-1}} "0";"10"};
{\ar_{g^i} "2";"12"};
{\ar^{q^i} "4";"14"};
{\ar_(0.56){m^{i-1}} "10";"12"};
{\ar_{e^i} "12";"14"};
{\ar@{-->}_{\del_{(i)}} "14";"16"};
{\ar@{}|\circlearrowright "0";"12"};
{\ar@{}|\circlearrowright "2";"14"};
\endxy
\]
\item[$(D_i)$] $d_Y^i=s^i\ci t^i$ holds.
\end{itemize}
\end{claim}
\begin{proof}
We show this by an induction on $i$. For $i=0$, $\Rfr^0=(A,d_Y^0,\id_A,\id_A,\id_A,0)$ trivially satisfies conditions $(A_0),(B_0),(C_0),(D_0)$.

For $1\le i\le n-1$, suppose that we have obtained $\Rfr^0,\Rfr^1,\ldots,\Rfr^{i-1}$ satisfying $(A_j),(B_j),(C_j),(D_j)$ for all $0\le j\le i-1$. Let us construct $\Rfr^i$. Realize $\ell^{i-1}\sas\del_{(i)}$, to obtain an $\sfr$-triangle
\begin{equation}\label{TrNZM}
N^{i-1}\ov{w^{i-1}}{\lra}Z^i\ov{z^i}{\lra}M^i\ov{\ell^{i-1}\sas\del_{(i)}}{\dra}.
\end{equation}
We claim $Z^i\in\Tcal$. Indeed when $i=1$, we take $(\ref{TrNZM})$ as $A\ov{d_X^0}{\lra}X^1\ov{e^1}{\lra}M^1\ov{\del_{(1)}}{\dra}$. When $2\le i\le n-1$, remark that
\begin{equation}\label{ExXM}
\C(X^i,T)\ov{-\ci m^{i-1}}{\lra}\C(M^{i-1},T)\to0
\end{equation}
is exact for any $T\in\Tcal$ as in (the dual of) the proof of Proposition~\ref{PropDecompEx}. Take a lift of $(\ell^{i-1},\id)\co\del_{(i)}\to\ell^{i-1}\sas\del_{(i)}$
\begin{equation}\label{DiagPDHI1}
\xy
(-14,6)*+{M^{i-1}}="0";
(0,6)*+{X^i}="2";
(12,6)*+{M^i}="4";
(24,6)*+{}="6";
(-14,-6)*+{N^{i-1}}="10";
(0,-6)*+{Z^i}="12";
(12,-6)*+{M^i}="14";
(24,-6)*+{}="16";
{\ar^(0.6){m^{i-1}} "0";"2"};
{\ar^{e^i} "2";"4"};
{\ar@{-->}^{\del_{(i)}} "4";"6"};
{\ar_{\ell^{i-1}} "0";"10"};
{\ar_{x^i} "2";"12"};
{\ar@{=} "4";"14"};
{\ar_(0.56){w^{i-1}} "10";"12"};
{\ar_{z^i} "12";"14"};
{\ar@{-->}_(0.62){\ell^{i-1}\sas\del_{(i)}} "14";"16"};
{\ar@{}|\circlearrowright "0";"12"};
{\ar@{}|\circlearrowright "2";"14"};
\endxy
\end{equation}
which gives the following $\sfr$-triangle, where we put $u^{i-1}=\begin{bmatrix}m^{i-1}\\ -\ell^{i-1}\end{bmatrix}$.
\begin{equation}\label{TrPDHI1}
M^{i-1}\ov{u^{i-1}}{\lra}X^i\oplus N^{i-1}\ov{[x^i\ w^{i-1}]}{\lra}Z^i\ov{(z^i)\uas\del_{(i)}}{\dra}
\end{equation}
Let $T\in\Tcal$ be any object. By the exactness of $(\ref{ExXM})$,
\[ \C(X^i\oplus N^{i-1},T)\ov{-\ci u^{i-1}}{\lra}\C(M^{i-1},T) \]
is also surjective. By property $(B_{i-1})$ and $X^i\in\Tcal$, we see that
\[ \E^k(u^{i-1},T)\co\E^k(X^i\oplus N^{i-1},T)\to\E^k(M^{i-1},T) \]
is an isomorphism for any $1\le k\le n-1$. Then the long exact sequence in Proposition~\ref{PropLongExact} obtained from $(\ref{TrPDHI1})$ shows $\E^k(Z^i,\Tcal)=0$ for $1\le k\le n-1$, which means $Z^i\in\Tcal$.

Thus for any $1\le i\le n-1$, we have shown $Z^i\in\Tcal$. Let us give a morphism $y^i\in\Tcal(Y^i,Z^i)$ as follows. When $2\le i\le n-1$, then
\[ \Tcal(Y^i,Z^i)\ov{-\ci d_Y^{i-1}}{\lra}\Tcal(Y^{i-1},Z^i)\ov{-\ci d_Y^{i-2}}{\lra}\Tcal(Y^{i-2},Z^i) \]
becomes exact since $\Yd$ is an $n$-exangle. By $(w^{i-1}\ci t^{i-1})\ci d_Y^{i-2}=w^{i-1}\ci(t^{i-1}\ci s^{i-2})\ci t^{i-2}=0$, there is a morphism $y^i\in\Tcal(Y^i,Z^i)$ satisfying $y^i\ci d_Y^{i-1}=w^{i-1}\ci t^{i-1}$. For any $T\in\Tcal$, since
\[ \E^{n-1}(N^{i-2},T)\cong\E^{n-1}(M^{i-2},T)=0 \]
by $(B_{i-2})$ and Lemma~\ref{LemMvanish} {\rm (2)}, the $\sfr$-triangle in $(A_{i-1})$ gives an exact sequence
\begin{equation}\label{---}
0\to\C(N^{i-1},T)\ov{-\ci t^{i-1}}{\lra}\C(Y^{i-1},T)
\end{equation}
by Proposition~\ref{PropABCvanish}. When $i=1$, we just put $y^1=g^1$. Sequence $(\ref{---})$ is equal to $0\to\C(A,T)\ov{\id}{\lra}\C(A,T)$, which is obviously exact also in this case.

Thus in any case, we have morphism $y^i\in\C(Y^i,Z^i)$ and exact sequence $(\ref{---})$. Applying $(\ref{---})$ to $T=X^i,Y^i,Z^i$ respectively, we obtain
\begin{eqnarray}
&g^i\ci s^{i-1}=m^{i-1}\ci q^{i-1},&\label{CancEqX}\\
&d_Y^i\ci s^{i-1}=0,&\label{CancEqY}\\
&y^i\ci s^{i-1}=w^{i-1}&\label{CancEqZ}
\end{eqnarray}
from the equalities $g^i\ci d_Y^{i-1}=d_X^{i-1}\ci g^{i-1}$, $d_Y^i\ci d_Y^{i-1}=0$ and $y^i\ci d_Y^{i-1}=w^{i-1}\ci t^{i-1}$. Similarly, the exactness of
$0\to\C(M^{i-1},T)\ov{-\ci e^{i-1}}{\lra}\C(X^{i-1},T)$
and the equality $f^i\ci d_X^{i-1}=d_Y^{i-1}\ci f^{i-1}$ show
\begin{equation}\label{Eqfmsl}
f^i\ci m^{i-1}=s^{i-1}\ci\ell^{i-1}.
\end{equation}

Since $w^{i-1}$ is an $\sfr$-inflation, $(\ref{CancEqZ})$ implies that so is $s^{i-1}$ by Condition~\ref{CondWIC}. Thus we can complete it into some $\sfr$-triangle
\begin{equation}\label{NYN}
N^{i-1}\ov{s^{i-1}}{\lra}Y^i\ov{t^i}{\lra}N^i\ov{\rho_{(i)}}{\dra}
\end{equation}
which shows $(A_i)$. Let us give a morphism from $(\ref{TrNZM})$ to it.
The exactness of
\[ \C(Z^i,Y^i)\ov{-\ci [x^i\ w^{i-1}]}{\lra}\C(X^i\oplus N^{i-1},Y^i)\ov{-\ci u^{i-1}}{\lra}\C(M^{i-1},Y^i) \]
induced from $(\ref{TrPDHI1})$ and the equality
\[ [f^i\ s^{i-1}]\ci u^{i-1}=f^i\ci m^{i-1}-s^{i-1}\ci \ell^{i-1}=0 \]
give $v^i\in\C(Z^i,Y^i)$ satisfying $v^i\ci x^i=f^i$ and $v^i\ci w^{i-1}=s^{i-1}$. As in Remark~\ref{RemsDecHtpyInv}, we can find a morphism $\ell^i\in\C(M^i,N^i)$ which gives the morphism of $\sfr$-triangles
\begin{equation}\label{DiagPDHI2}
\xy
(-14,6)*+{N^{i-1}}="0";
(0,6)*+{Z^i}="2";
(12,6)*+{M^i}="4";
(24,6)*+{}="6";
(-14,-6)*+{N^{i-1}}="10";
(0,-6)*+{Y^i}="12";
(12,-6)*+{N^i}="14";
(24,-6)*+{}="16";
{\ar^(0.6){w^{i-1}} "0";"2"};
{\ar^{z^i} "2";"4"};
{\ar@{-->}^(0.66){\ell^{i-1}\sas\del_{(i)}} "4";"6"};
{\ar@{=} "0";"10"};
{\ar^{v^i} "2";"12"};
{\ar^{\ell^i} "4";"14"};
{\ar_{s^{i-1}} "10";"12"};
{\ar_{t^i} "12";"14"};
{\ar@{-->}_{\rho_{(i)}} "14";"16"};
{\ar@{}|\circlearrowright "0";"12"};
{\ar@{}|\circlearrowright "2";"14"};
\endxy
\end{equation}
with a short exact sequence
\[ 0\to Z^i\ov{\Big[\raise1ex\hbox{\leavevmode\vtop{\baselineskip-8ex \lineskip0.4ex \ialign{#\crcr{$\scriptstyle{-z^i}$}\crcr{$\scriptstyle{\ v^i}$}\crcr}}}\Big]}{\lra}M^i\oplus Y^i\ov{[\ell^i\ t^i]}{\lra}N^i\to0 \]
since $w^{i-1}\sas\rho_{(i)}=y^i\sas s^{i-1}\sas\rho_{(i)}=0$ by $(\ref{CancEqZ})$. Since $Y^i,Z^i\in\Tcal$, this $\ell^i$ satisfies $(B_i)$.
Composition of $(\ref{DiagPDHI1})$ and $(\ref{DiagPDHI2})$ gives the left morphism of $\sfr$-triangles in $(C_i)$. Also by $(\ref{CancEqX})$, there is $q^i\in\C(N^i,M^i)$ which gives the right morphism of $\sfr$-triangles in $(C_i)$.
By $(\ref{CancEqY})$ and the exactness of
\[ \C(N^i,Y^{i+1})\ov{-\ci t^i}{\lra}\C(Y^i,Y^{i+1})\ov{-\ci s^{i-1}}{\lra}\C(N^{i-1},Y^{i+1}) \]
induced from $(\ref{NYN})$, we obtain $s^i\in\C(N^i,Y^{i+1})$ satisfying $d_Y^i=s^i\ci t^i$ which shows $(D_i)$. This completes the proof of Claim~\ref{ClaimDecHtpyInv}.
\end{proof}
By Claim~\ref{ClaimDecHtpyInv}, we have obtained
\[
\xy
(-54,0)*+{A}="-4";
(-40,0)*+{Y^1}="-2";
(-20,0)*+{Y^2}="0";
(-0.4,0)*+{\cdots}="2";
(4,0)*+{\cdots}="3";
(24,0)*+{Y^{n-1}}="4";
(44,0)*+{Y^n}="6";
(58,0)*+{C}="8";
(-30,-10)*+{N^1}="14";
(-30,2)*+{}="15";
(-10,10)*+{N^2}="10";
(-10,-3)*+{}="11";
(14,-10)*+{N^{n-2}}="16";
(14,2)*+{}="17";
(34,10)*+{N^{n-1}}="12";
(34,-3)*+{}="13";
(-21,-19)*+{}="24";
(-1,19)*+{}="20";
(23,-19)*+{}="26";
(43,19)*+{}="22";
{\ar^{d_Y^0} "-4";"-2"};
{\ar^{d_Y^1} "-2";"0"};
{\ar_{d_Y^2} "0";"2"};
{\ar^{d_Y^{n-2}} "3";"4"};
{\ar_{d_Y^{n-1}} "4";"6"};
{\ar^{d_Y^n} "6";"8"};
{\ar_{t^1} "-2";"14"};
{\ar_{s^1} "14";"0"};
{\ar^{t^2} "0";"10"};
{\ar^{s^2} "10";"2"};
{\ar_{t^{n-2}} "3";"16"};
{\ar_{s^{n-2}} "16";"4"};
{\ar^{t^{n-1}} "4";"12"};
{\ar^{s^{n-1}} "12";"6"};
{\ar@{-->}_{\rho_{(1)}} "14";"24"};
{\ar@{-->}^{\rho_{(2)}} "10";"20"};
{\ar@{-->}_{\rho_{(n-2)}} "16";"26"};
{\ar@{-->}^{\rho_{(n-1)}} "12";"22"};
{\ar@{}|\circlearrowright "14";"15"};
{\ar@{}|\circlearrowright "10";"11"};
{\ar@{}|\circlearrowright "16";"17"};
{\ar@{}|\circlearrowright "12";"13"};
\endxy
\]
with the above-mentioned properties. Put $\rho_{(n)}=\ell^{n-1}\sas\del_{(n)}$, and realize it to obtain an $\sfr$-triangle
\begin{equation}\label{NZC}
N^{n-1}\ov{w^{n-1}}{\lra}Z^n\ov{z^n}{\lra}C\ov{\rho_{(n)}}{\dra}.
\end{equation}
Remark that $(C_i)$ for $1\le i\le n-1$ shows
\begin{eqnarray*}
\del&=&\del_{(1)}\cup\del_{(2)}\cup\cdots\cup\del_{(n)}%
\ =\ ((\ell^1)\uas\rho_{(1)})\cup\del_{(2)}\cup\cdots\cup\del_{(n)}\\
&=&\rho_{(1)}\cup(\ell^1\sas\del_{(2)})\cup\cdots\cup\del_{(n)}%
\ =\ \rho_{(1)}\cup((\ell^2)\uas\rho_{(2)})\cup\cdots\cup\del_{(n)}\\
&=&\cdots\ =\ \rho_{(1)}\cup\cdots\cup\rho_{(n-1)}\cup(\ell^{n-1}\sas\del_{(n)})\ =\ \rho_{(1)}\cup\cdots\cup\rho_{(n-1)}\cup\rho_{(n)}
\end{eqnarray*}
by Proposition~\ref{PropCupProdNat}.
Thus, to give an $\sfr$-decomposition of $\Yd$, it remains to show that
\[ N^{n-1}\ov{s^{n-1}}{\lra}Y^n\ov{d_Y^n}{\lra}C\ov{\rho_{(n)}}{\dra} \]
is an $\sfr$-triangle. Namely, it is enough to show that this is isomorphic to $(\ref{NZC})$.

If we put $r^i=q^i\ci\ell^i$ for $0\le i\le n-1$, then property $(C_i)$ shows $r^{i-1}\sas\del_{(i)}=(r^i)\uas\del_{(i)}$ for any $1\le i\le n-1$, which implies
\begin{eqnarray*}
\del_{(1)}\cup\cdots\cup\del_{(n-1)}\cup(r^{n-1}\sas\del_{(n)})&=&\!\!\cdots\ =((r^0)\sas\del_{(1)})\cup\cdots\cup\del_{(n-1)}\cup\del_{(n)}\\
&=&\del_{(1)}\cup\cdots\cup\del_{(n-1)}\cup\del_{(n)}\ =\ \del
\end{eqnarray*}
by Proposition~\ref{PropCupProdNat}. 
Since $(\del_{(1)}\cup\cdots\cup\del_{(n-1)})\cup-\co\E(C,M^{n-1})\to\E^n(C,A)$ is injective by Lemma~\ref{LemMvanish}, the above equality implies $\del_{(n)}=r^{n-1}\sas\del_{(n)}=q^{n-1}\sas\rho_{(n)}$. By the dual of Remark~\ref{RemsDecHtpyInv}, we can find a morphism of $\sfr$-triangles
\[
\xy
(-14,6)*+{M^{n-1}}="0";
(0,6)*+{X^n}="2";
(12,6)*+{C}="4";
(24,6)*+{}="6";
(-14,-6)*+{N^{n-1}}="10";
(0,-6)*+{Z^n}="12";
(12,-6)*+{C}="14";
(24,-6)*+{}="16";
{\ar^(0.6){m^{n-1}} "0";"2"};
{\ar^{d_X^n} "2";"4"};
{\ar@{-->}^(0.66){\del_{(n)}} "4";"6"};
{\ar_{\ell^{n-1}} "0";"10"};
{\ar_{x^n} "2";"12"};
{\ar@{=} "4";"14"};
{\ar_(0.56){w^{n-1}} "10";"12"};
{\ar_{z^n} "12";"14"};
{\ar@{-->}_{\rho_{(n)}} "14";"16"};
{\ar@{}|\circlearrowright "0";"12"};
{\ar@{}|\circlearrowright "2";"14"};
\endxy
\]
which gives a split short exact sequence
\[ 0\to M^{n-1}\ov{\Big[\raise1ex\hbox{\leavevmode\vtop{\baselineskip-8ex \lineskip0.4ex \ialign{#\crcr{$\scriptstyle{\ m^{n-1}}$}\crcr{$\scriptstyle{-\ell^{n-1}}$}\crcr}}}\Big]}{\lra}X^n\oplus N^{n-1}\ov{[x^n\ w^{n-1}]}{\lra}Z^n\to0 \]
since $(z^n)\uas\del_{(n)}=q^{n-1}\sas(z^n)\uas\rho_{(n)}=0$. Thus $(B_{n-1})$ and $X^n\in\Tcal$ shows $Z^n\in\Tcal$.
Remark that $\E^{n-1}(M^{n-2},\Tcal)=0$ implies the exactness of $0\to\C(M^{n-1},Y^n)\ov{-\ci e^{n-1}}{\lra}\C(X^{n-1},Y^n)$ by Proposition~\ref{PropABCvanish}, and thus the equality $f^n\ci d_X^{n-1}=d_Y^{n-1}\ci f^{n-1}$ shows the commutativity of the following diagram.
\[
\xy
(-14,6)*+{M^{n-1}}="0";
(0,6)*+{X^n}="2";
(12,6)*+{C}="4";
(24,6)*+{}="6";
(-14,-6)*+{N^{n-1}}="10";
(0,-6)*+{Y^n}="12";
(12,-6)*+{C}="14";
(24,-6)*+{}="16";
{\ar^(0.6){m^{n-1}} "0";"2"};
{\ar^{d_X^n} "2";"4"};
{\ar_{\ell^{n-1}} "0";"10"};
{\ar_{f^n} "2";"12"};
{\ar@{=} "4";"14"};
{\ar_(0.56){s^{n-1}} "10";"12"};
{\ar_{d_Y^n} "12";"14"};
{\ar@{}|\circlearrowright "0";"12"};
{\ar@{}|\circlearrowright "2";"14"};
\endxy
\]
This gives rise to a (unique) morphism $v^n\in\C(Z^n,Y^n)$ satisfying $v^n\ci w^{n-1}=s^{n-1}$, $v^n\ci x^n=f^n$ and then $d_Y^n\ci v^n=z^n$ by the universality of the cokernel.
In particular,
\begin{equation}\label{NZCC}
\xy
(-14,6)*+{N^{n-1}}="0";
(0,6)*+{Z^n}="2";
(12,6)*+{C}="4";
(24,6)*+{}="6";
(-14,-6)*+{N^{n-1}}="10";
(0,-6)*+{Y^n}="12";
(12,-6)*+{C}="14";
(24,-6)*+{}="16";
{\ar^(0.6){w^{n-1}} "0";"2"};
{\ar^{z^n} "2";"4"};
{\ar@{=} "0";"10"};
{\ar_{v^n} "2";"12"};
{\ar@{=} "4";"14"};
{\ar_(0.56){s^{n-1}} "10";"12"};
{\ar_{d_Y^n} "12";"14"};
{\ar@{}|\circlearrowright "0";"12"};
{\ar@{}|\circlearrowright "2";"14"};
\endxy
\end{equation}
is commutative. It remains to show that this $v^n$ is an isomorphism.

Similarly as above, by the dual of Remark~\ref{RemsDecHtpyInv}, we find a good lift $(q^{n-1},d,\id)$ of $(q^{n-1},\id)\co\rho_{(n)}\to\del_{(n)}$ which gives a split short exact sequence
\[ 0\to N^{n-1}\ov{\Big[\raise1ex\hbox{\leavevmode\vtop{\baselineskip-8ex \lineskip0.4ex \ialign{#\crcr{$\scriptstyle{\ w^{n-1}}$}\crcr{$\scriptstyle{-q^{n-1}}$}\crcr}}}\Big]}{\lra}Z^n\oplus M^{n-1}\ov{[d\ m^{n-1}]}{\lra}X^n\to0 \]
since $(d_X^n)\uas\rho_{(n)}=(x^n)\uas(z^n)\uas\rho_{(n)}=0$. For any $T\in\Tcal$, since $X^n,Z^n\in\Tcal$, this shows in particular that
\[ \E^k(T,q^{n-1})\co\E^k(T,N^{n-1})\to\E^k(T,M^{n-1}) \]
is an isomorphism for any $1\le k\le n-1$. Thus if we denote the composition of
\[ \E(T,N^{n-1})\ov{\E(T,q^{n-1})}{\lra}\E(T,M^{n-1})\ov{(\del_{(1)}\cup\cdots\cup\del_{(n-1)})\cup-}{\lra}\E^n(T,A) \]
by $\iota$, this is an injective homomorphism which makes
\begin{equation}\label{CETN}
\xy
(-14,6)*+{\Tcal(T,C)}="0";
(14,6)*+{\E(T,N^{n-1})}="2";
(-14,-6)*+{\Tcal(T,C)}="4";
(14,-6)*+{\E^n(T,A)}="6";
{\ar^(0.46){(\rho_{(n)})\ssh} "0";"2"};
{\ar@{=} "0";"4"};
{\ar^{\iota} "2";"6"};
{\ar_{\del\ssh} "4";"6"};
{\ar@{}|\circlearrowright "0";"6"};
\endxy
\end{equation}
commutative.

We claim that
\begin{equation}\label{ExactSaigo}
\Tcal(T,Y^{n-2})\ov{d_Y^{n-2}\ci-}{\lra}\Tcal(T,Y^{n-1})\ov{t^{n-1}\ci-}{\lra}\C(T,N^{n-1})\to0
\end{equation}
is exact. If $n=2$, this is nothing but the sequence $\Tcal(T,A)\ov{d_Y^0}{\lra}\Tcal(T,Y^1)\ov{t^1\ci-}{\lra}\C(T,N^1)\to0$ induced from $(A_1)$, which is exact since $\E(T,A)=0$. If $n\ge 3$, then $(B_{n-3})$ and $(B_{n-2})$ show $\E(T,N^{n-3})=\E(T,N^{n-2})=0$, which implies that
\begin{eqnarray*}
&
\Tcal(T,Y^{n-2})\ov{t^{n-2}\ci-}{\lra}\C(T,N^{n-2})\to 0&\\
&\C(T,N^{n-2})\ov{s^{n-2}\ci-}{\lra}\Tcal(T,Y^{n-1})\ov{t^{n-1}\ci-}{\lra}\C(T,N^{n-1})\to 0&
\end{eqnarray*}
are exact. This shows the exactness of $(\ref{ExactSaigo})$.

Thus $(\ref{ExactSaigo})$ is exact in any case.
Since $\Yd\in\AE_{(\Tcal,\E^n)}^{n+2}$ is an $n$-exangle, it follows that
\[ 0\to\C(T,N^{n-1})\ov{s^{n-1}\ci-}{\lra}\Tcal(T,Y^n)\ov{d_Y^n\ci-}{\lra}\Tcal(T,C)\ov{\del\ssh}{\lra}\E^n(T,A) \]
is exact. Thus the commutativity of $(\ref{NZCC})$ and $(\ref{CETN})$ gives the following morphism of exact sequences
\[
\xy
(-40,7)*+{\C(T,N^{n-1})}="0";
(-12,7)*+{\Tcal(T,Z^n)}="2";
(12,7)*+{\Tcal(T,C)}="4";
(40,7)*+{\E(T,N^{n-1})}="6";
(-58,-7)*+{0}="-12";
(-40,-7)*+{\C(T,N^{n-1})}="10";
(-12,-7)*+{\Tcal(T,Y^n)}="12";
(12,-7)*+{\Tcal(T,C)}="14";
(40,-7)*+{\E^n(T,A)}="16";
{\ar^(0.52){w^{n-1}\ci-} "0";"2"};
{\ar^{z^n\ci-} "2";"4"};
{\ar^(0.44){(\rho_{(n)})\ssh} "4";"6"};
{\ar@{=} "0";"10"};
{\ar^{v^n\ci-} "2";"12"};
{\ar@{=} "4";"14"};
{\ar^{\iota} "6";"16"};
{\ar_{} "-12";"10"};
{\ar_{s^{n-1}\ci-} "10";"12"};
{\ar_{d_Y^n\ci-} "12";"14"};
{\ar_{\del\ssh} "14";"16"};
{\ar@{}|\circlearrowright "0";"12"};
{\ar@{}|\circlearrowright "2";"14"};
{\ar@{}|\circlearrowright "4";"16"};
\endxy
\]
for any $T\in\Tcal$. The middle $\Tcal(T,v^n)\co\Tcal(T,Z^n)\to\Tcal(T,Y^n)$ becomes an isomorphism by the five lemma. Yoneda lemma shows that $v^n\co Z^n\to Y^n$ is an isomorphism.
\end{proof}

Let $\CEs$ be an extriangulated category satisfying Assumption~\ref{Assump}, as usual in this section. The arguments so far show the following.
\begin{thm}
Assume that an extriangulated category $\CEs$ satisfies Condition~\ref{CondWIC}, and that its $n$-cluster tilting subcategory $\Tcal\se\C$ satisfies Condition~\ref{CondnCluster}.

For any $A,C\in\Tcal$ and any $\del\in\E^n(C,A)$, define
\begin{equation}\label{sfrn}
\sfr^n(\del)=[X^{\mr}]
\end{equation}
to be the homotopic equivalence class of $X^{\mr}$ in $\TAC$, using $X^{\mr}$ which gives an $\sfr$-decomposable object $\Xd\in\AE^{n+2}_{(\Tcal,\E^n)}$.

Then $(\Tcal,\E^n,\sfr^n)$ becomes an $n$-exangulated category.
\end{thm}
\begin{proof}
Remark that $(\ref{sfrn})$ is well-defined, by Proposition~\ref{PropsDecExist}, Corollary \ref{CorDecompEx} and Proposition~\ref{PropsDecHtpyInv}.

{\rm (R0)} and {\rm (R1)} are shown in Propositions~\ref{PropDecompExAdded} and \ref{PropDecompEx}, respectively. {\rm (R2)} is obvious. {\rm (EA1)} follows from Proposition~\ref{PropsDecAdded}. {\rm (EA2)} follows from Corollary \ref{CorShowEA2}. Dually for {\rm (EA2$\op$).}
\end{proof}

\section{Examples}\label{section_Examples}
I this section we provide explicit examples of $n$-exangulated categories $\C$. First we consider three known sources of $n$-exangulated categories, namely $n$-abelian, $n$-exact and $(n+2)$-angulated categories. Then we consider a class of examples, which are not of this form. In each case we will give concrete examples for $n=2$ by providing Auslander-Reiten quivers of the relevant categories (see \cite{ASS} for background on Auslander-Reiten quivers).

Most of our examples will be given by a finite dimensional algebra $A$ over a field $k$. More precisely $\C$ will be given as a subcategory of the category of finitely generated right $A$-modules $\modu A$ or its bounded derived category $\Dbdd(\modu A)$. For explicit examples we will describe $A$ as the path algebra of a quiver with relations. For these we follow the conventions in \cite{ASS}.

\subsection{Examples of $n$-abelian and $n$-exact categories}
In \cite{J} it is shown that any $n$-cluster tilting subcategory $\C$ of an abelian respectively exact category is $n$-abelian respectively $n$-exact. Hence by Proposition~\ref{nExactIsnExangulated}, $\C$ may be considered as an $n$-exangulated category (provided that the extensions groups are sets). 

To obtain an $n$-abelian example suppose that there is an $n$-cluster tilting subcategory $\C \subseteq \modu A$. Then $\C$ is $n$-abelian and so $n$-exangulated. There are many known examples of such algebras $A$ and subcategories $\C$ (see for example \cite{IO}, \cite{HI1}, \cite{HI2}, \cite{JK2} \cite{P2}, \cite{V}). To be explicit let $A$ be the path algebra of the quiver
\[
\xy 
(0,0)*+{1}="1";
(10,0)*+{2}="2";
(20,0)*+{3}="3";
(30,0)*+{4}="4";
{\ar^{a} "1";"2"};
{\ar^{b} "2";"3"};
{\ar^{c} "3";"4"};
\endxy
\quad\mbox{with relation}\quad
abc = 0.
\]
Then $\modu A$ has a unique $2$-cluster tilting subcategory $\C$ consisting of all direct sums of projective and injective modules. Below is the Auslander-Reiten quiver of $\modu A$ with the indcomposables in $\C$ labelled $C_i$, where $1 \le i \le 6$.

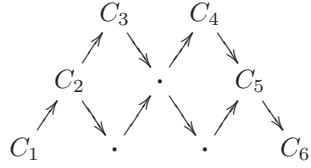
\[
\xy 
(0,0)*+{C_1}="1";
(6,9)*+{C_2}="2";
(12,18)*+{C_3}="3";
(12,0)*+{\mr}="11";
(18,9)*+{\mr}="12";
(24,18)*+{C_4}="13";
(24,0)*+{\mr}="21";
(30,9)*+{C_5}="22";
(36,0)*+{C_6}="23";
{\ar^{} "1";"2"};
{\ar^{} "2";"3"};
{\ar^{} "2";"11"};
{\ar^{} "3";"12"};
{\ar^{} "11";"12"};
{\ar^{} "12";"13"};
{\ar^{} "12";"21"};
{\ar^{} "13";"22"};
{\ar^{} "21";"22"};
{\ar^{} "22";"23"};
\endxy
\]
There are $2$-exangles
\[ C_1\ov{}{\lra}C_2\ov{}{\lra}C_4\ov{}{\lra}C_5\ov{}{\dra}, \]
\[ C_1\ov{}{\lra}C_3\ov{}{\lra}C_4\ov{}{\lra}C_6\ov{}{\dra}, \]
\[ C_2\ov{}{\lra}C_3\ov{}{\lra}C_5\ov{}{\lra}C_6\ov{}{\dra}. \]

To obtain an $n$-exact example we may choose $A$ to be $n$-complete in the sense of \cite{I2}. In particular, this means that there is a distinguished tilting module $T \in \modu A$ such that the exact category $T^{\perp_{\ge1}} = \{X \in \modu A \mid \Ext_A^i(T,X) = 0, \, i \ge 1\}$ has a certain $n$-cluster tilting subcategory $\C$. Hence $\C$ is $n$-exact and therefore $n$-exangulated. Since $n$-complete algebras are closed under taking higher Auslander algebras by \cite{I2} and tensor products by \cite{P1} they provides a rich source of examples. For instance, let $B$ be the path algebra of the quiver
\[
\xy 
(0,0)*+{1}="1";
(10,0)*+{2.}="2";
{\ar^{a} "1";"2"};
\endxy
\]
Then $B$ is $1$-complete and so by \cite{P1}, $A = B \otimes_k B$ is $2$-complete. Below is the Auslander-Reiten quiver of $\modu A$ with the indcomposables in $\C$ labelled $C_i$, where $1 \le i \le 5$. The tilting module $T$ equals $C_1 \oplus C_2 \oplus C_3 \oplus C_4$ and $T^{\perp_{\ge1}}$ contains precisely one indecomposable not in $\C$, which is labelled $\circ$.

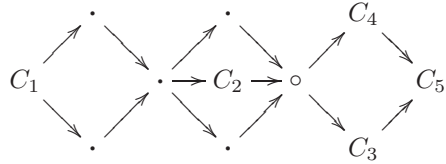
\[
\xy 
(0,9)*+{C_1}="1";
(9,0)*+{\mr}="2";
(9,18)*+{\mr}="3";
(18,9)*+{\mr}="4";
(27,0)*+{\mr}="5";
(27,9)*+{C_2}="6";
(27,18)*+{\mr}="7";
(36,9)*+{\circ}="8";
(45,0)*+{C_3}="9";
(45,18)*+{C_4}="10";
(54,9)*+{C_5}="11";
{\ar^{} "1";"2"};
{\ar^{} "1";"3"};
{\ar^{} "2";"4"};
{\ar^{} "3";"4"};
{\ar^{} "4";"5"};
{\ar^{} "4";"6"};
{\ar^{} "4";"7"};
{\ar^{} "5";"8"};
{\ar^{} "6";"8"};
{\ar^{} "7";"8"};
{\ar^{} "8";"9"};
{\ar^{} "8";"10"};
{\ar^{} "9";"11"};
{\ar^{} "10";"11"};
\endxy
\]
There is a $2$-exangle 
\[ C_1\ov{}{\lra}C_2\ov{}{\lra}C_3 \oplus C_4\ov{}{\lra}C_5\ov{}{\dra}. \]

For another source of $n$-exact examples we may consider a Geigle-Lenzing complete intersection $(R, \mathbb{L})$ as introduced in \cite{HIMO}. This is a certain kind of commutative ring $R$ graded by an abelian group $\mathbb{L}$. To $(R, \mathbb{L})$ is associated a non-commutative projective scheme $\mathbb{X}$ in the sense of \cite{AZ} and an abelian category $\coh \mathbb{X}$ of coherent sheaves on $\mathbb{X}$. Moreover, there is an extension closed subcategory $\vect \mathbb{X} \subseteq \coh \mathbb{X}$, of vector bundles, which is thus exact. For detailed definitions see \cite{HIMO}. In \cite[Theorem 1.3.12]{HIMO}, a sufficient condition is given for when $\vect \mathbb{X}$ admits an $n$-cluster tilting subcategory $\C$, which is therefore $n$-exact and so $n$-exangulated. Similarly, $n$-cluster tilting subcategories of the category of $\mathbb{L}$-graded maximal Cohen-Macaulay $R$-modules are also studied in \cite{HIMO} providing further examples (see for instance \cite[Example 4.4.14]{HIMO}).

\subsection{Examples of $(n+2)$-angulated categories}\label{Examplen+2angle}
Next we consider $(n+2)$-angulated categories $\C$, which by Proposition~\ref{PropAtoEA} can be considered as $n$-exangulated. To obtain such a category we may take a triangulated category $\T$ and $\C \subseteq \T$ an $n$-cluster tilting subcategory such that $\C[n] = \C$ (see \cite[Theorem 1.]{GKO}).

A standard example is to take $\T = \Dbdd(\modu A)$, where $A$ is a so-called $n$-representation finite algebra, meaning that $A$ has global dimension $n$ and admits an $n$-cluster tilting subcategory $\C_0 \subseteq \modu A$, with finitely many indecomposables. Then
\[
\C = \add\{\C_0[nt] \mid t \in \mathbb{Z}\} \subseteq \Dbdd(\modu A)
\] 
is $n$-cluster tilting and satisfies $\C[n] = \C$. To be explicit let $n= 2$ and $A$ be the path algebra of the quiver
\[
\xy 
(0,0)*+{1}="1";
(10,0)*+{2}="2";
(20,0)*+{3}="3";
{\ar^{a} "1";"2"};
{\ar^{b} "2";"3"};
\endxy
\quad\mbox{with relation}\quad
ab = 0.
\]
Then $A$ has global dimension $2$ and $\modu A$ admits a unique $2$-cluster tilting subcategory $\C_0$ consisting of all direct sums of projective and injective modules. Since $A$ is derived equivalent to the path algebra of the same quiver without relations it is easy to describe $\Dbdd(\modu A)$. Below is the Auslander-Reiten quiver of $\Dbdd(\modu A)$ with indecomposables in $\C$ labelled $C_i$ where $i \in \mathbb{Z}$. The indecomposables in $\C_0$ are $C_0$, $C_1$, $C_2$, $C_3$ and for all $i$ we have $C_i[2t] = C_{i+4t}$.
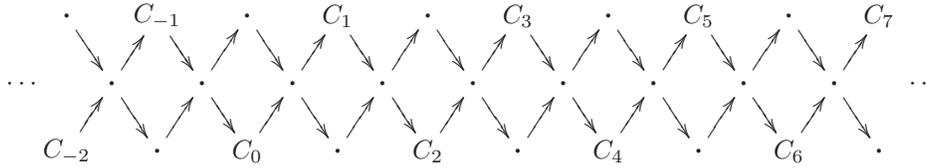
\[
\xy 
(-6,9)*+{\cdots}="d";
(0,18)*+{\mr}="0";
(0,0)*+{C_{-2}}="1";
(6,9)*+{\mr}="2";
(12,18)*+{C_{-1}}="3";
(12,0)*+{\mr}="11";
(18,9)*+{\mr}="12";
(24,18)*+{\mr}="13";
(24,0)*+{C_0}="21";
(30,9)*+{\mr}="22";
(36,18)*+{C_1}="23";
(36,0)*+{\mr}="31";
(42,9)*+{\mr}="32";
(48,18)*+{\mr}="33";
(48,0)*+{C_2}="41";
(54,9)*+{\mr}="42";
(60,18)*+{C_3}="43";
(60,0)*+{\mr}="51";
(66,9)*+{\mr}="52";
(72,18)*+{\mr}="53";
(72,0)*+{C_4}="61";
(78,9)*+{\mr}="62";
(84,18)*+{C_5}="63";
(84,0)*+{\mr}="71";
(90,9)*+{\mr}="72";
(96,18)*+{\mr}="73";
(96,0)*+{C_6}="81";
(102,9)*+{\mr}="82";
(108,18)*+{C_7}="83";
(108,0)*+{\mr}="91";
(114,9)*+{\cdots}="dd";
{\ar^{} "0";"2"};
{\ar^{} "1";"2"};
{\ar^{} "2";"3"};
{\ar^{} "2";"11"};
{\ar^{} "3";"12"};
{\ar^{} "11";"12"};
{\ar^{} "12";"13"};
{\ar^{} "12";"21"};
{\ar^{} "13";"22"};
{\ar^{} "21";"22"};
{\ar^{} "22";"23"};
{\ar^{} "22";"31"};
{\ar^{} "23";"32"};
{\ar^{} "31";"32"};
{\ar^{} "32";"33"};
{\ar^{} "32";"41"};
{\ar^{} "33";"42"};
{\ar^{} "41";"42"};
{\ar^{} "42";"43"};
{\ar^{} "42";"51"};
{\ar^{} "43";"52"};
{\ar^{} "51";"52"};
{\ar^{} "52";"53"};
{\ar^{} "52";"61"};
{\ar^{} "53";"62"};
{\ar^{} "61";"62"};
{\ar^{} "62";"63"};
{\ar^{} "62";"71"};
{\ar^{} "63";"72"};
{\ar^{} "71";"72"};
{\ar^{} "72";"73"};
{\ar^{} "72";"81"};
{\ar^{} "73";"82"};
{\ar^{} "81";"82"};
{\ar^{} "82";"83"};
{\ar^{} "82";"91"};
\endxy
\]
There are $2$-exangles
\[ C_{i}\ov{}{\lra}C_{i+1}\ov{}{\lra}C_{i+2}\ov{}{\lra}C_{i+3}\ov{}{\dra}\]
for all $i \in \mathbb{Z}$.

Another source of examples comes from \cite{K}. Let $\Dsing(A)$ be the singularity category of $A$, i.e., the Verdier quotient of $\Dbdd(\modu A)$ by the perfect derived category of $A$. Assume that there is an $n\mathbb{Z}$-cluster tilting subcategory $\C_0$ of $\modu A$ (meaning that $\C_0$ is $n$-cluster tilting and $\Ext_A^i(\C_0,\C_0) = 0$ for all $i \not \in n\mathbb{Z}$). Then it is shown in \cite{K} that 
\[
\C = \add\{\C_0[nt] \mid t \in \mathbb{Z}\} \subseteq \Dsing(A)
\] 
is $n$-cluster tilting. Since $\C[n] =\C$, it follows that $\C$ is $(n+2)$-angulated and hence $n$-exangulated.

\subsection{Examples which are neither $n$-exact nor $(n+2)$-angulated}
In Remark~\ref{RemRelativeNor} it was shown how applying relative theory to $(n+2)$-angulated categories can give $n$-exangulated categories that are neither $n$-exact nor $(n+2)$-angulated. In this section we provide another source of such examples. To this end let $\T$ be a triangulated category, and $\Tcal\se\T$ be an $n$-cluster tilting subcategory satisfying $\Tcal[n-1]^{\perp}\se\Tcal[-1]^{\perp}$ and ${}^{\perp}\Tcal[n-1]\se {}^{\perp}\Tcal[-1]$. Recall that the triplet $(\Tcal,\Ext^n,\sfr^n)$ becomes an $n$-exangulated category, as seen in Example~\ref{Ex_CondnCluster}. 

Our aim is to produce an $n$-exangulated subcategory of $\Tcal$. First, let us introduce some notation: for full subcategories $\Xcal,\Ycal\se\T$, let $\Xcal\ast\Ycal$ denote the full subcategory of $\T$ consisting of objects $C\in\T$ which admit distinguished triangles $X\to C\to Y\to X[1]$ in $\T$ for some $X\in\Xcal,Y\in\Ycal$.
\begin{lem}\label{LemForExample}
Let $\T$ and $\Tcal$ be as above. Let $\Ecal\se\T$ be an extension-closed subcategory, and put $\Tcal\ppr=\Ecal\cap\Tcal$. Assume that $\Tcal\ppr$ satisfies
\begin{equation}\label{CondForTprime}
\Tcal\ppr [i]\ast\Tcal\ppr\se\Tcal\ppr_{[0,i]}\quad(1\le i< n),
\end{equation}
where we put $\Tcal\ppr_{[0,i]}=\Tcal\ppr\ast\Tcal\ppr[1]\ast\cdots\ast\Tcal\ppr[i]$.
Then $\Tcal\ppr$ is an extension-closed subcategory of $(\Tcal,\Ext^n,\sfr^n)$ in the sense of Definition~\ref{DefnExtClosed}, and has an induced $n$-exangulated structure.
\end{lem}
\begin{proof}
As a first step, we show that $(\ref{CondForTprime})$ is equivalent to the extension-closedness of $\Tcal\ppr_{[0,i]}$ in $\T$ for any $1\le i<n$. Since the converse is trivial, we only show that $(\ref{CondForTprime})$ implies the extension-closedness of $\Tcal\ppr_{[0,i]}\se\T$ by an induction on $i$. If $i=1$, this is obvious. For $i>1$, suppose that $\Tcal\ppr_{[0,i-1]}\se\T$ has been shown to be extension-closed. Then
\begin{eqnarray*}
\Tcal\ppr_{[0,i]}\ast\Tcal\ppr_{[0,i]}&=&\Tcal\ppr_{[0,i-1]}\ast(\Tcal\ppr [i]\ast\Tcal\ppr)\ast\Tcal\ppr_{[0,i-1]}[1]\ \se\ \Tcal\ppr_{[0,i-1]}\ast\Tcal\ppr_{[0,i]}\ast\Tcal\ppr_{[0,i-1]}[1]\\
&\se&\Tcal\ppr_{[0,i-1]}\ast(\Tcal\ppr_{[0,i-1]}\ast\Tcal\ppr_{[0,i-1]}[1])\ast\Tcal\ppr_{[0,i-1]}[1]\ = \Tcal\ppr_{[0,i-1]}\ast\Tcal\ppr_{[0,i-1]}[1]\\
&\se&\Tcal\ppr_{[0,i-1]}\ast\Tcal\ppr_{[0,i-1]}\ast\Tcal\ppr [i]\ = \Tcal\ppr_{[0,i-1]}\ast\Tcal\ppr [i]\ =\ \Tcal\ppr_{[0,i]}
\end{eqnarray*}
shows that $\Tcal\ppr_{[0,i]}\se\T$ is also extension-closed.

Secondly, let us refine the proof of Proposition \ref{PropsDecExist} in the current situation, to construct an $n$-exangle out of a distinguished triangle. Suppose we are given a distinguished triangle
\[ E[-1]\ov{a^{n-1}}{\lra}B^n\ov{b^n}{\lra}C\ov{\del}{\lra}E \]
in $\T$ satisfying $B^n\in\Tcal\ppr_{[0,n-1]}$. Then there is a distinguished triangle
\[ B^{n-1}\to X^n\ov{x^n}{\lra} B^n\to B^{n-1}[1] \]
with $X^n\in\Tcal\ppr$ and $B^{n-1}\in\Tcal\ppr_{[0,n-2]}$. By the octahedron axiom, we obtain the following.
\[
\xy
(-27,23)*+{B^{n-1}}="0";
(-9,23)*+{M^{n-1}}="2";
(9,23)*+{E[-1]}="4";
(27,23)*+{B^{n-1}[1]}="6";
(-27,8)*+{B^{n-1}}="10";
(-9,8)*+{X^n}="12";
(9,8)*+{B^n}="14";
(27,8)*+{B^{n-1}[1]}="16";
(-9,-7)*+{C}="22";
(9,-7)*+{C}="24";
(27,-7)*+{M^{n-1}[1]}="26";
(-9,-22)*+{M^{n-1}[1]}="32";
(9,-22)*+{E}="34";
{\ar^{b^{n-1}} "0";"2"};
{\ar^{\del^{n-1}} "2";"4"};
{\ar^{} "4";"6"};
{\ar@{=} "0";"10"};
{\ar^{m^{n-1}} "2";"12"};
{\ar^{a^{n-1}} "4";"14"};
{\ar@{=} "6";"16"};
{\ar_{} "10";"12"};
{\ar_{x^n} "12";"14"};
{\ar_{} "14";"16"};
{\ar_{e^n} "12";"22"};
{\ar^{b^n} "14";"24"};
{\ar^{} "16";"26"};
{\ar@{=} "22";"24"};
{\ar_{} "24";"26"};
{\ar_{\del_{(n)}} "22";"32"};
{\ar^{\del} "24";"34"};
{\ar_(0.6){\del^{n-1}[1]} "32";"34"};
{\ar@{}|\circlearrowright "0";"12"};
{\ar@{}|\circlearrowright "2";"14"};
{\ar@{}|\circlearrowright "4";"16"};
{\ar@{}|\circlearrowright "12";"24"};
{\ar@{}|\circlearrowright "14";"26"};
{\ar@{}|\circlearrowright "22";"34"};
\endxy
\]
Inductively we obtain distinguished triangles
\begin{eqnarray*}
&E[i-n-1]\ov{a^{i-1}}{\lra}B^i\ov{b^i}{\lra}M^i\ov{\del^i}{\lra}E[i-n],&\\
&M^{i-1}\ov{m^{i-1}}{\lra}X^i\ov{e^i}{\lra}M^i\ov{\del_{(i)}}{\lra}M^{i-1}[1],%
\quad B^{i-1}\to X^i\ov{x^i}{\lra}B^i\to B^{i-1}[1]&
\end{eqnarray*}
for $1\le i<n$, which satisfy $X^i\in\Tcal\ppr$, $B^i\in\Tcal\ppr_{[0,i-1]}$, $B^0=0$, $x^1=\id_{B^1}$, $\delta^0=\id_{E[-n]}$ and $\del^i=\del^{i-1}[1]\ci\del_{(i)}$.

We use this construction to show the extension-closedness of $\Tcal\ppr$ in $(\Tcal,\Ext^n,\sfr^n)$. Let $\E\ppr$ be the restriction of $\Ext^n$ to $\Tcal\ppr$. Take any pair of objects $A,C\in\Tcal\ppr$, and any element $\del\in\E\ppr(C,A)=\T(C,A[n])$. Complete $\del$ into a distinguished triangle
\[ A[n-1]\ov{a^{n-1}}{\lra}B^n\ov{b^n}{\lra}C\ov{\del}{\lra}A[n] \]
in $\T$. Since $B^n\in\Tcal\ppr [n-1]\ast\Tcal\ppr\se\Tcal\ppr_{[0,n-1]}$, we may apply the above construction to obtain distinguished triangles
\[ M^{i-1}\ov{m^{i-1}}{\lra}X^i\ov{e^i}{\lra}M^i\ov{\del_{(i)}}{\lra}M^{i-1}[1] \]
with $X^i\in\Tcal\ppr$ for all $1\le i<n$, which satisfy $M^0=A$ and
\[ \del=\del_{(1)}[n-1]\ci\del_{(2)}[n-2]\ci\cdots\ci\del_{(n)}. \]
This gives a $\sfr^n$-distinguished $n$-exangle in $(\Tcal,\Ext^n,\sfr^n)$
\[ A\to X^1\to X^2\to\cdots\to X^n\to C\ov{\del}{\dra} \]
satisfying $X^i\in\Tcal\ppr$ for any $1\le i<n$. Thus $\Tcal\ppr$ is extension-closed in $(\Tcal,\Ext^n,\sfr^n)$.
As in Proposition \ref{PropnExtClosed}, define $\tfr$ by $\tfr(\del)=[X^{\mr}]$ using the above $X^{\mr}$. This gives an exact realization of $\E\ppr$.

It remains to show that $(\Tcal\ppr,\E\ppr,\tfr)$ satisfies {\rm (EA1)}.
By duality, we only show that $\tfr$-deflations are closed by compositions.
Let $\Xd$ and $\Yr$ be any pair of $\tfr$-distinguished $n$-exangles satisfying $X^n=C$, $X^{n+1}=Y^n=D$, $Y^{n+1}=E$, and let us show $d_Y^n\ci d_X^n$ is a $\tfr$-deflation. Remark that these $n$-exangles yield distinguished triangles in $\T$
\[ M\to C\ov{d_X^n}{\lra}D\to M[1],\quad N\to D\ov{d_Y^n}{\lra}E\to N[1] \]
satisfying $M,N\in\Tcal\ppr_{[0,n-1]}$. Complete $d_Y^n\ci d_X^n$ into a distinguished triangle in $\T$
\begin{equation}\label{DistTrLCE}
L^n\ov{\ell^n}{\lra}C\ov{d_Y^n\ci d_X^n}{\lra}E\ov{\lam^n}{\lra}L^n[1].
\end{equation}
By the octahedron axiom, we also have a distinguished triangle $M\to L^n\to N\to M[1]$.
Since $\Tcal\ppr_{[0,n-1]}\se\T$ is extension-closed as shown in the first step, it follows $L^n\in\Tcal\ppr_{[0,n-1]}$. If we apply the construction in the second step to $(\ref{DistTrLCE})$, it gives the following sequence of distinguished triangles
\[ L^{i-1}\ov{u^{i-1}}{\lra}Z^i\ov{v^i}{\lra}L^i\ov{\lam_{(i)}}{\lra}L^{i-1}[1] \]
with $Z^i\in\Tcal\ppr$ and $L^i\in\Tcal\ppr_{[0,i-1]}$ for $1\le i\le n$.
\[
\xy
(-54,0)*+{L^1}="-4";
(-40,0)*+{Z^2}="-2";
(-20,0)*+{Z^3}="0";
(-0.4,0)*+{\cdots}="2";
(4,0)*+{\cdots}="3";
(24,0)*+{Z^n}="4";
(44,0)*+{C}="6";
(58,0)*+{E}="8";
(-30,-10)*+{L^2}="14";
(-30,2)*+{}="15";
(-10,10)*+{L^3}="10";
(-10,-3)*+{}="11";
(14,-10)*+{L^{n-1}}="16";
(14,2)*+{}="17";
(34,10)*+{L^n}="12";
(34,-3)*+{}="13";
(-21,-19)*+{}="24";
(-1,19)*+{}="20";
(23,-19)*+{}="26";
(43,19)*+{}="22";
{\ar^{u^1} "-4";"-2"};
{\ar^{u^2\ci v^2} "-2";"0"};
{\ar_{u^3\ci v^3} "0";"2"};
{\ar^{u^{n-1}\ci v^{n-1}} "3";"4"};
{\ar_{u^n\ci v^n} "4";"6"};
{\ar^{d_Y^n\ci d_X^n} "6";"8"};
{\ar_{v^2} "-2";"14"};
{\ar_{u^2} "14";"0"};
{\ar^{v^3} "0";"10"};
{\ar^{u^3} "10";"2"};
{\ar_{v^{n-1}} "3";"16"};
{\ar_{u^{n-1}} "16";"4"};
{\ar^{v^n} "4";"12"};
{\ar^{u^n} "12";"6"};
{\ar_{\lam_{(2)}} "14";"24"};
{\ar^{\lam_{(3)}} "10";"20"};
{\ar_{\lam_{(n-1)}} "16";"26"};
{\ar^{\lam_{(n)}} "12";"22"};
{\ar@{}|\circlearrowright "14";"15"};
{\ar@{}|\circlearrowright "10";"11"};
{\ar@{}|\circlearrowright "16";"17"};
{\ar@{}|\circlearrowright "12";"13"};
\endxy
\]
This constitutes a $\tfr$-distinguished $n$-exangle
\[ L^1\to Z^2\to \cdots Z^n\to C\ov{d_Y^n\ci d_X^n}{\lra} E\ov{\thh}{\dra} \]
for $\thh=\lam_{(2)}[n-1]\ci\lam_{(3)}[n-2]\ci\cdots\ci\lam_{(n)}[1]\ci\lam^n\in\T(E,L^1[n])=\E\ppr(E,L^1)$.
In particular, $d_Y^n\ci d_X^n$ is a deflation.
\end{proof}

Next we specialize our setting. Let $A$ be a finite dimensional $k$-algebra of global dimension $n$ such that $\modu A$ admits an $n$-cluster tilting subcategory $\Ccal_0$. Set $\T = \Dbdd(\modu A)$. 
Then 
\[
\Tcal = \add\{\Ccal_0[nt] \mid t \in \Zbb\} \subseteq \Dbdd(\modu A)
\]
is $n$-cluster tilting such that $\Tcal[n] = \Tcal$ and so $\Tcal$ is $n$-exangulated. Let
\[
\Ecal = \Dtstr(\modu A) : = \{X \in \Dbdd(\modu A) \mid \Hom_{\T}(A[j],X) = 0, \, j < 0\}
\]
and
\[
\Ttstr := \Tcal \cap \Ecal= \add\{\Ccal_0[nt] \mid t \ge 0\}.
\]
Then $\Ecal$ is extension closed. 

\begin{prop}\label{halfDerived}
In the notation above $\Ttstr$ is an extension closed subcategory of $(\Tcal,\Ext^n,\sfr^n)$ with induced $n$-exangulated structure. 
\end{prop}
\begin{proof}
Without loss of generality assume that $n \ge 2$.
We verify that Lemma~\ref{LemForExample} can be applied. Since $\Tcal$ is contravariantly finite so is $\Ttstr$. Let $X \in \Ecal$ and let $f: T_0 \to X$ be a right $\Ttstr$-approximation. Then there is a triangle
\[ X_1\ov{}{\lra}T_0\ov{f}{\lra}X\ov{}{\lra}X_1[1] \]
with $X_1 \in \Ecal[-1]$. 
Using that $f$ is a right $\Ttstr$-approximation together with $\Hom_{\T}(\Ttstr,\Ttstr[1]) = 0$ we find that $\Hom_{\T}(\Ttstr,X_1[1]) = 0$. In particular, $A \in \Ttstr$ gives $\Hom_{\T}(A,X_1[1]) = 0$ and so $X_1 \in \Ecal$. 
Iterating we obtain triangles
\[ X_{j+1}\ov{}{\lra}T_j\ov{}{\lra}X_j\ov{}{\lra}X_{j+1}[1] \]
with $X_{j} \in \Ecal$ and $T_{j} \in \Ttstr$ and $\Hom_{\T}(\Ttstr,X_{j}[1]) = 0$. Now assume that $X \in \Ttstr[i] * \Ttstr$ for some $1\le i \le n-1$. We need to show that $X \in \Ttstr * \cdots * \Ttstr[i]$. By the above it is enough to show that $X_{i} \in \Ttstr$. Now consider the category
\[
(\Ttstr)^{\perp_{[1,n-1-i]}} := \{X\in \Dbdd(\modu A) \mid \Hom_{\T}(\Ttstr,X[j]) = 0 \ \text{for}\  1 \le j \le n-1-i\}.
\]
It is extension closed and contains both $\Ttstr$ and $\Ttstr[i]$. So in particular $X \in (\Ttstr)^{\perp_{[1,n-1-i]}}$. Now let $T \in \Ttstr$. Applying $\Hom_{\T}(T,-)$ to the above triangles we find
\[
\Hom_{\T}(T,X_i[k]) \cong \Hom_{\T}(T,X_{i-1}[k-1]) \cong \cdots \cong \Hom_{\T}(T,X[k-i]) = 0 
\]
for each $i+1 \le k \le n-1$ and 
\[
\Hom_{\T}(T,X_i[k]) \cong \Hom_{\T}(T,X_{i-1}[k-1]) \cong \cdots \cong \Hom_{\T}(T,X_{i-k+1}[1]) = 0 
\]
for each $1 \le k \le i$. On the other hand, for $t \le -1$ and $k \ge 1$ we get
\[
\Hom_{\T}(\Ccal_0[nt],X_i[k]) = 0
\]
since the global dimension of $A$ is $n$ and $X_i \in \Ecal$. Hence $X_i \in \Tcal^{\perp_{[1,n-1]}} = \Tcal$ and so $X_i \in \Ecal \cap \Tcal = \Ttstr$.
\end{proof}

Proposition~\ref{halfDerived} can be applied to any $n$-representation finite algebra. Note that the $n$-exangulated category $\Ttstr$ is not $(n+2)$-angulated since $\Ttstr$ is not closed under $[n]$. Moreover, it is not $n$-exact since $\Ccal_0, \Ccal_0[n] \in \Ttstr$ implies that $C \to 0$ is an inflation for any $C \in \Ccal_0$ but not a monomorphism if $C \neq 0$. To obtain an explicit example we may take $A$ to be the path algebra of the quiver
\[
\xy 
(0,0)*+{1}="1";
(10,0)*+{2}="2";
(20,0)*+{3}="3";
{\ar^{a} "1";"2"};
{\ar^{b} "2";"3"};
\endxy
\quad\mbox{with relation}\quad
ab = 0
\]
as in subsection~\ref{Examplen+2angle}. Then $\modu A$ admits a unique $n$-cluster tilting subcategory $\Ccal_0$ and so our construction applies. Below is the Auslander-Reiten quiver of $\Dbdd(\modu A)$ with the indecomposables in $\Ttstr$ labelled $C_i$ for $i \ge 0$, and the other indecomposables in $\Ecal$ labelled $\circ$.

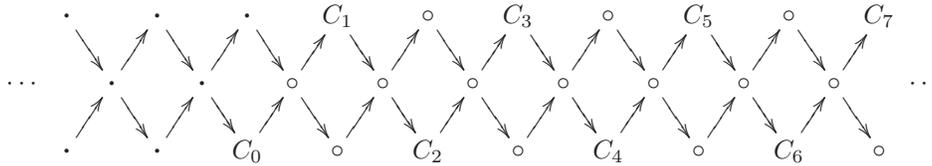
\[
\xy 
(-6,9)*+{\cdots}="d";
(0,18)*+{\mr}="0";
(0,0)*+{\mr}="1";
(6,9)*+{\mr}="2";
(12,18)*+{\mr}="3";
(12,0)*+{\mr}="11";
(18,9)*+{\mr}="12";
(24,18)*+{\mr}="13";
(24,0)*+{C_0}="21";
(30,9)*+{\circ}="22";
(36,18)*+{C_1}="23";
(36,0)*+{\circ}="31";
(42,9)*+{\circ}="32";
(48,18)*+{\circ}="33";
(48,0)*+{C_2}="41";
(54,9)*+{\circ}="42";
(60,18)*+{C_3}="43";
(60,0)*+{\circ}="51";
(66,9)*+{\circ}="52";
(72,18)*+{\circ}="53";
(72,0)*+{C_4}="61";
(78,9)*+{\circ}="62";
(84,18)*+{C_5}="63";
(84,0)*+{\circ}="71";
(90,9)*+{\circ}="72";
(96,18)*+{\circ}="73";
(96,0)*+{C_6}="81";
(102,9)*+{\circ}="82";
(108,18)*+{C_7}="83";
(108,0)*+{\circ}="91";
(114,9)*+{\cdots}="dd";
{\ar^{} "0";"2"};
{\ar^{} "1";"2"};
{\ar^{} "2";"3"};
{\ar^{} "2";"11"};
{\ar^{} "3";"12"};
{\ar^{} "11";"12"};
{\ar^{} "12";"13"};
{\ar^{} "12";"21"};
{\ar^{} "13";"22"};
{\ar^{} "21";"22"};
{\ar^{} "22";"23"};
{\ar^{} "22";"31"};
{\ar^{} "23";"32"};
{\ar^{} "31";"32"};
{\ar^{} "32";"33"};
{\ar^{} "32";"41"};
{\ar^{} "33";"42"};
{\ar^{} "41";"42"};
{\ar^{} "42";"43"};
{\ar^{} "42";"51"};
{\ar^{} "43";"52"};
{\ar^{} "51";"52"};
{\ar^{} "52";"53"};
{\ar^{} "52";"61"};
{\ar^{} "53";"62"};
{\ar^{} "61";"62"};
{\ar^{} "62";"63"};
{\ar^{} "62";"71"};
{\ar^{} "63";"72"};
{\ar^{} "71";"72"};
{\ar^{} "72";"73"};
{\ar^{} "72";"81"};
{\ar^{} "73";"82"};
{\ar^{} "81";"82"};
{\ar^{} "82";"83"};
{\ar^{} "82";"91"};
\endxy
\]
There are $2$-exangles
\[ C_{i}\ov{}{\lra}C_{i+1}\ov{}{\lra}C_{i+2}\ov{}{\lra}C_{i+3}\ov{}{\dra}\]
for all $i \ge 0$.

\end{document}